\documentclass{amsart}

\usepackage{amssymb}
\usepackage{amsthm}
\usepackage{amsmath}

\usepackage{tikz}
\usetikzlibrary{arrows,calc}
\usepackage{tikz-cd}

\usepackage{hyperref}
\usepackage[foot]{amsaddr}

\usepackage[scr=dutchcal]{mathalfa}
\usepackage{accents}

\theoremstyle{plain}
\newtheorem{proposition}{Proposition}[section]
\newtheorem{theorem}[proposition]{Theorem}
\newtheorem{lemma}[proposition]{Lemma}
\newtheorem{corollary}[proposition]{Corollary}
\theoremstyle{definition}
\newtheorem{example}[proposition]{Example}
\newtheorem{definition}[proposition]{Definition}
\newtheorem{observation}[proposition]{Observation}

\theoremstyle{remark}
\newtheorem{remark}[proposition]{Remark}

\DeclareMathOperator{\Ad}{Ad}

\DeclareMathOperator{\Aut}{Aut}

\DeclareMathOperator{\Hom}{Hom}

\DeclareMathOperator{\End}{End}

\DeclareMathOperator{\Spanset}{span} 
\DeclareMathOperator{\Gr}{Gr} 
\DeclareMathOperator{\id}{id}

\DeclareMathOperator{\Isom}{Isom}

\DeclareMathOperator{\SL}{\mathsf{SL}}
\DeclareMathOperator{\GL}{\mathsf{GL}}

\DeclareMathOperator{\SU}{\mathsf{SU}}
\DeclareMathOperator{\PSL}{\mathsf{PSL}}
\DeclareMathOperator{\PGL}{\mathsf{PGL}}

\DeclareMathOperator{\dist}{d}

\DeclareMathOperator{\Bc}{\mathcal{B}}
\DeclareMathOperator{\Cc}{\mathcal{C}}
\DeclareMathOperator{\Fc}{\mathcal{F}}
\DeclareMathOperator{\Gc}{\mathcal{G}}
\DeclareMathOperator{\Hc}{\mathcal{H}}

\DeclareMathOperator{\Nc}{\mathcal{N}}
\DeclareMathOperator{\Oc}{\mathcal{O}}

\DeclareMathOperator{\Rc}{\mathcal{R}}

\DeclareMathOperator{\Pc}{\mathcal{P}}

\DeclareMathOperator{\Cb}{\mathbb{C}}

\DeclareMathOperator{\Hb}{\mathbb{H}}

\DeclareMathOperator{\Kb}{\mathbb{K}}
\DeclareMathOperator{\Nb}{\mathbb{N}}
\DeclareMathOperator{\Rb}{\mathbb{R}}

\DeclareMathOperator{\Zb}{\mathbb{Z}}

\DeclareMathOperator{\Psf}{\mathsf{P}}

\DeclareMathOperator{\aL}{\mathfrak{a}}

\DeclareMathOperator{\gL}{\mathfrak{g}}

\DeclareMathOperator{\kL}{\mathfrak{k}}
\DeclareMathOperator{\nL}{\mathfrak{n}}
\DeclareMathOperator{\pL}{\mathfrak{p}}

\newcommand{\abs}[1]{\left|#1\right|}

\newcommand{\norm}[1]{\left\|#1\right\|}

\newcommand{\wh}[1]{\widehat{#1}}
\newcommand{\ip}[1]{\left\langle #1\right\rangle}

\newcommand{\geodflow}{\phi}
\newcommand{\flatflow}{\varphi}
\newcommand{\homflow}{\psi}

\newcommand{\into}{\hookrightarrow}
\newcommand{\proj}{\mathbf{P}}
\newcommand{\peripherals}{\Pc}

\DeclareMathOperator{\Stab}{Stab}

\begin{document}

\title[Relatively Anosov representations]{Relatively Anosov representations via flows I: theory}
\author{Feng Zhu}\email{fzhu52@wisc.edu}
\author{Andrew Zimmer}\address{University of Wisconsin-Madison}\email{amzimmer2@wisc.edu}
\date{\today}
\keywords{(relatively) Anosov representations, relatively hyperbolic groups, geometrically finite groups, dominated splittings, convex real projective geometry}
\subjclass[2020]{Primary 22E40; Secondary 37D20, 37D30, 20H10, 20F67, 37B05}

\begin{abstract} This is the first in a series of two papers that develops a theory of relatively Anosov representations using the original ``contracting flow on a bundle" definition of Anosov representations introduced by Labourie and Guichard--Wienhard. In this paper we will mostly focus on general theory while in the second paper we will focus on examples.  In the case of relatively hyperbolic groups, this bundle construction involves several choices: the model Gromov-hyperbolic space the group acts on and the norms on the fibers of the bundle. We use the properties of these bundles to define a subclass of nicely behaved relatively Anosov representations, which we call uniformly relatively Anosov.  We also prove a stability result.

\end{abstract}

\maketitle

\tableofcontents

\section{Introduction}

Anosov representations were introduced by Labourie~\cite{L2006}, and further developed by Guichard--Wienhard~\cite{GW}, as a generalization of convex cocompact representations into the isometry group of real hyperbolic space. Informally speaking, an Anosov representation is a representation of a word-hyperbolic group into a semisimple Lie group which has a equivariant boundary map into a flag manifold with good dynamical properties. Since their initial introduction there have been a number of different interpretations due to, among others, Kapovich--Leeb--Porti \cite{KLP2017,KLP2018,KLP2018b}, Gu\'eritaud--Guichard--Kassel--Wienhard~\cite{GGKW}, Bochi--Potrie--Sambarino \cite{BPS}, and Kassel--Potrie~\cite{KP2022}. 

It is natural to wonder if the theory of Anosov representations can be extended to relatively hyperbolic groups. In this direction, Kapovich--Leeb~\cite{KL} developed relative versions of the characterizations in~\cite{KLP2017,KLP2018,KLP2018b} and in recent work the first author~\cite{reldomreps, rdr2} developed relative versions of the characterizations in ~\cite{BPS} and~\cite{KP2022}. In recent work, Weisman~\cite{W2022} introduces a new class of representations of relatively hyperbolic groups called \emph{extended geometrically finite representations} which includes the class of relatively Anosov representations (as in Definition~\ref{defn:Pk Anosov}) and also convex co-compact representations (in the sense of~\cite{DGK_convex_cocopt_realproj}) of relatively hyperbolic groups. 

This is the first in a series of two papers whose purpose is to develop a theory of relatively Anosov representations using the original  ``contracting flow on a bundle" definition of Labourie and Guichard--Wienhard. In the relative setting this construction involves several choices. First one needs to specify a model space the group acts on and second, since the base of the bundle is non-compact, one needs to specify the norms on the fibers of the bundle. In this paper we will mostly focus on general theory while in the second paper we will focus on examples.  

Previously, Canary, Zhang, and second author developed a ``contracting flow on a bundle" theory for geometrically finite Fuchsian groups~\cite{CZZ2021}. This theory was then used in~\cite{BCKM2021a,BCKM2021b, CZZ2022}. The general case considered here is much more complicated. In the Fuchsian case the peripherals subgroups are always cyclic and there is a canonical flow space coming from the action of the group on the real hyperbolic plane. Further, as we will explain below, relatively Anosov representations of geometrically finite Fuchsian groups fall into a special class of well-behaved relatively Anosov representations which we call \emph{uniformly relatively Anosov}, see Section~\ref{sec:introducing uniformly Anosov} and Corollary~\ref{cor:geometrically finite Fuchsian groups}. 

Throughout the paper, we will let $\Kb$ denote either the real numbers $\Rb$ or the complex numbers $\Cb$.

\subsection{Relatively Anosov representations via a boundary map} There is some choice in how exactly one defines an Anosov representation of a relatively hyperbolic group, see the discussion in~\cite{KL}, but the following is perhaps the most natural. 

\begin{definition}\label{defn:Pk Anosov} Suppose that $(\Gamma,\peripherals)$ is relatively hyperbolic with Bowditch boundary $\partial(\Gamma, \peripherals)$. A representation $\rho\colon \Gamma \to \SL(d,\Kb)$ is \emph{$\Psf_k$-Anosov relative to $\peripherals$} if there exists a continuous map 
$$
\xi = (\xi^k, \xi^{d-k}) \colon \partial(\Gamma, \peripherals) \to \Gr_k(\Kb^d) \times \Gr_{d-k}(\Kb^d)
$$
which is 
\begin{enumerate} 
\item \emph{$\rho$-equivariant}: if $\gamma \in \Gamma$, then $\rho(\gamma) \circ \xi = \xi \circ \gamma$, 
\item \emph{transverse}: if $x,y \in \partial(\Gamma, \peripherals)$ are distinct, then $\xi^k(x) \oplus \xi^{d-k}(y) = \Kb^d$, 
\item \emph{strongly dynamics-preserving}: if $(\gamma_n)_{n \geq 1}$ is a sequence of elements in $\Gamma$ where $\gamma_n \to x \in \partial(\Gamma, \peripherals)$ and $\gamma_n^{-1} \to y \in \partial(\Gamma, \peripherals)$, then 
$$
\lim_{n \to \infty} \rho(\gamma_n)V = \xi^k(x)
$$
uniformly on compact subsets of $\left\{ V \in \Gr_k(\Kb^d) : V \text{ transverse to } \xi^{d-k}(y) \right\}$. 
\end{enumerate}
\end{definition}

We note that the above definition is equivalent to being ``asymptotically embedded'' in the sense of Kapovich--Leeb~\cite{KL}, see Proposition~\ref{prop:relAnosov_relae} below. Using the theory developed in this paper, we will also show that the above definition is equivalent to being ``relatively dominated'' in the sense of~\cite{reldomreps}, see Corollary~\ref{cor:relAnosov = reldom} below. Further, when $\rho$ is sufficiently irreducible, the strongly dynamics-preserving property is a consequence of the other two conditions, see Proposition~\ref{prop:sufficiently irreducible} below. 

Given a semisimple Lie group $\mathsf{G}$ and a parabolic subgroup $\mathsf{P} \leq \mathsf{G}$, one can define $\mathsf{P}$-Anosov representations into $\mathsf{G}$ in a completely analogous way, see Section~\ref{sec:general case}. As in the classical word-hyperbolic case, see~\cite[Prop. 4.3]{GW}, we will show that there exists an irreducible linear representation $\Psi \colon \mathsf{G} \to \SL(d,\Rb)$ such that $\rho \colon \Gamma \to \mathsf{G}$ is $\mathsf{P}$-Anosov relative to $\peripherals$ if and only if $\Psi \circ \rho \colon \Gamma \to \SL(d,\Rb)$ is $\Psf_1$-Anosov relative to $\peripherals$ (see Proposition~\ref{prop:adapted representations}). Thus in this paper we mostly restrict our attention to relatively Anosov representations into $\SL(d,\Kb)$.

\subsection{Relatively Anosov representations via a flow space}\label{sec:intro to flow} We now introduce the ``contracting flow'' definition for Anosov representations of relatively hyperbolic groups and state our first main result. 

Given a relatively hyperbolic group $(\Gamma, \peripherals)$ we can realize $\Gamma$ as a subgroup of $\Isom(X)$ where $X$ is a proper geodesic Gromov-hyperbolic metric space such that every point in $X$ is within a uniformly bounded distance of a geodesic, $\Gamma$ acts geometrically finitely on the Gromov boundary $\partial_\infty X$ of $X$, and the stabilizers of the parabolic fixed points are exactly the conjugates of $\peripherals$. Following the terminology in~\cite{HealyHruska}, we call such an $X$ a \emph{weak cusp space for $(\Gamma, \peripherals)$}.

Given such an $X$, let $\Gc(X)$ denote the space of parametrized geodesic lines in $X$ and let $\geodflow^t$ denote the natural flow on $\Gc(X)$ given by 
$$
\geodflow^t(\sigma) = \sigma(\cdot + t).
$$
We let $\wh{\Gc}(X) := \Gamma \backslash \Gc(X)$ denote the quotient. We also have natural maps $\cdot^\pm \colon \Gc(X) \to \partial_\infty X$ given by 
$$
\sigma^\pm := \lim_{t \to \pm \infty} \sigma(t). 
$$

Next given a representation $\rho \colon \Gamma \to \SL(d,\Kb)$, let 
$$
E(X) := \Gc(X) \times \Kb^d \quad \text{and} \quad \wh{E}_\rho(X) := \Gamma \backslash E(X)
$$
where $\Gamma$ acts on  $E(X)$ by 
$$
\gamma \cdot (\sigma, Y) = (\gamma \circ \sigma, \rho(\gamma) Y).
$$ 
Notice that $\wh{E}_\rho(X) \to \wh{\Gc}(X)$ is a vector bundle. The flow $\geodflow^t$ extends to a flow on $E(X)$, which we call $\flatflow^t$, which acts trivially on the second factor. This in turn descends to a flow on $\wh{E}_\rho(X)$ which we also call $\flatflow^t$. 

Given a continuous, $\rho$-equivariant, transverse map 
$$
\xi = (\xi^k, \xi^{d-k}) \colon\partial(\Gamma, \peripherals) \to \Gr_k(\Kb^d) \times \Gr_{d-k}(\Kb^d)
$$
we can define vector bundles $\Theta^k, \Xi^{d-k} \to \Gc(X)$ by setting 
$$
\Theta^k(\sigma) := \xi^k(\sigma^+) \quad \text{and} \quad \Xi^{d-k}(\sigma) := \xi^{d-k}(\sigma^-).
$$
(here we use the fact that $\partial_\infty X$ is equivariantly homeomorphic to $\partial(\Gamma, \peripherals)$). Since $\xi$ is transverse, we have $E(X) = \Theta^k \oplus \Xi^{d-k}$. Since $\xi$ is $\rho$-equivariant, this descends to a vector bundle decomposition  $\wh{E}_\rho(X) = \wh{\Theta}^k \oplus \wh{\Xi}^{d-k}$. We can then consider the bundle 
$$
\Hom(\wh{\Xi}^{d-k}, \wh{\Theta}^k) \to \wh{\Gc}(X).
$$
and, since the subbundles are $\flatflow^t$-invariant, we can define a flow on $\Hom(\wh{\Xi}^{d-k}, \wh{\Theta}^k)$ by 
$$
\homflow^t(f) := \flatflow^t \circ f \circ \flatflow^{-t}.
$$
Finally, we note that any metric on $\wh{E}_\rho(X) \to \wh{\Gc}(X)$ induces, via the operator norm, a continuous family of norms on the fibers of $\Hom(\wh{\Xi}^{d-k}, \wh{\Theta}^k) \to \wh{\Gc}(X)$. 

\begin{definition}
With the notation above, we say that $\rho$ is  \emph{$\Psf_k$-Anosov relative to $X$} if there exists a metric $\norm{\cdot}$ on the vector bundle $\wh{E}_\rho(X) \rightarrow \wh{\Gc}(X)$ such that the flow $\homflow^t$ on $\Hom(\wh{\Xi}^{d-k}, \wh{\Theta}^k)$ is exponentially contracting (with respect to the associated operator norms).
\end{definition}

A relatively hyperbolic group $(\Gamma, \peripherals)$ can have many non-quasi-isometric weak cusp spaces (see~\cite{Healy2020}) and, at least initially, it is not entirely clear which weak cusp space one should or can use when constructing the bundles above. 

In~\cite{GrovesManning}, Groves--Manning constructed weak cusp spaces, which are now often called \emph{Groves--Manning cusp spaces}, by attaching so-called combinatorial horoballs to a Cayley graph of the group. These spaces are perhaps the most canonical choice of weak cusp space, see~\cite{HealyHruska}. The first main result of this paper is that given a relatively Anosov representation and any Groves--Manning cusp space, one can always construct families of norms on the associated vector bundle so that the flow on the Hom bundle is exponentially contracting.

\begin{theorem}[see Sections~\ref{sec:contraction_Anosov} and~\ref{sec:building_norms}] \label{thm:main}
Suppose that $(\Gamma,\peripherals)$ is relatively hyperbolic and $\rho\colon \Gamma \to \SL(d,\Kb)$ is a representation. Then the following are equivalent: 
\begin{enumerate}
\item $\rho$ is $\Psf_k$-Anosov relative to $\peripherals$,
\item there is a weak cusp space $X$ for $(\Gamma,\peripherals)$ such that $\rho$ is $\Psf_k$-Anosov relative to $X$,
\item if $X$ is any Groves--Manning cusp space for $(\Gamma,\peripherals)$, then $\rho$ is $\Psf_k$-Anosov relative to $X$.
\end{enumerate}
\end{theorem}

\begin{remark}\label{rmk:directions in the proof of theorem main} By definition (3) $\implies$ (2) and  ``standard arguments'' (e.g.\ as in~\cite{Canary_notes}) from the theory of Anosov representations imply that (2) $\implies$ (1). So the new content in Theorem~\ref{thm:main} is that (1) $\implies$ (3). 

\end{remark} 

\begin{remark} In the ``classical'' word-hyperbolic case, the flow space, $\Gc(X)$, used in Theorem~\ref{thm:main} is slightly different than the construction in~\cite{GW}. In particular, in their paper they consider a bundle over the ``geodesic flow space'' of a word-hyperbolic group, which informally is the quotient of the space of geodesics in Cayley graph where geodesic lines joining the same points at infinity are identified. The construction of this geodesic flow space is somewhat technical, see~\cite{C1994,M2005}, and one  observation in this work is that it is not necessary to use this construction to obtain many of the basic properties of Anosov representations like stability, quasi-isometric embeddings, and H\"older reguarlity of the boundary maps. 
 \end{remark}

As an application of Theorem~\ref{thm:main}, we can use standard dynamical arguments to prove a relative stability result. We note that for representations of relatively hyperbolic groups, being relatively Anosov is not an open condition. For instance suppose $\Gamma = \ip{a,b} \leq \PSL(2,\Rb)$ is a geometrically finite free group where $b$ is parabolic. Fix lifts $\tilde{a}, \tilde{b} \in \SL(2,\Rb)$ of $a,b$, then consider the representations $\rho_t \colon \Gamma \to \SL(4,\Rb)$ defined by 
$$
\rho_t(a) = \id_2 \oplus \tilde{a}  \quad \text{and} \quad \rho_t(b) =  \begin{pmatrix} 1 & t   \\ 0 & 1 \end{pmatrix} \oplus  \tilde{b}.
$$
Then $\rho_0$ is $\Psf_1$-Anosov relative to $\peripherals:=\{ \ip{b}\}$, but $\rho_t$ is not when $t \neq 0$. 

To avoid examples like these, given a representation $\rho_0\colon \Gamma \to \SL(d,\Kb)$ of a relatively hyperbolic group $(\Gamma, \peripherals)$, we let $\Hom_{\rho_0} (\Gamma, \SL(d,\Kb))$ denote the set of representations $\rho\colon \Gamma \to \SL(d,\Kb)$ such that for each $P \in \peripherals$, the representations $\rho|_P$ and $\rho_0|_P$ are conjugate.

\begin{theorem}[see Section~\ref{sec:stability}]\label{thm:stability}
Suppose that $(\Gamma, \peripherals)$ is relatively hyperbolic and $X$ is a weak cusp space for $(\Gamma, \peripherals)$. If $\rho_0\colon\Gamma \to \SL(d,\Kb)$ is $\Psf_k$-Anosov relative to $X$, then there exists an open neighborhood $\Oc$ of $\rho_0$ in $\Hom_{\rho_0}(\Gamma, \SL(d,\Kb))$ such that every representation in $\Oc$ is $\Psf_k$-Anosov relative to $X$. 

Moreover:
\begin{enumerate}
\item If $\xi_\rho$ is the Anosov boundary map of $\rho \in \Oc$, then the map 
$$
(\rho, x) \in \Oc \times \partial(\Gamma, \peripherals) \mapsto \xi_\rho(x) \in \Gr_k(\Kb^d) \times \Gr_{d-k}(\Kb^d)
$$
is continuous. 
\item If $h\colon M \to \mathcal{O}$ is a real-analytic family of representation and $x \in \partial(\Gamma,\peripherals)$, then the map 
$$
u \in M \mapsto \xi_{h(u)}(x) \in \Gr_k(\Kb^d) \times \Gr_{d-k}(\Kb^d)
$$
is real-analytic.
\end{enumerate}
\end{theorem}

In the special case when $\Gamma$ is a geometrically finite Fuchsian group, Theorem~\ref{thm:stability} was established in~\cite{CZZ2021} and using the characterization in Theorem~\ref{thm:main} the argument from~\cite{CZZ2021}  can be refined to work in the general case. 

As mentioned above, in recent work, Weisman~\cite{W2022} introduces a new class of representations of relatively hyperbolic groups called extended geometrically finite representations which includes the class of relatively Anosov representations. For this class of representations, Weisman  proves a general stability result which implies, in the context of  Theorem~\ref{thm:stability}, that being $\Psf_k$-Anosov relative to $\peripherals$ is an open condition in $\Hom_{\rho_0}(\Gamma, \SL(d,\Kb))$. In the general setting Weisman considers, the ``moreover'' part of Theorem~\ref{thm:stability} is not true (see~\cite[Th.\ 1.4]{W2022}). 

It seems unlikely to us that the contracting flow approach of this paper can be used in the general setting considered by Weisman. However, in the setting of relatively Anosov representations it seems like this approach is better suited to extracting quantitative stability results, for instance the quasi-isometry and H\"older regularity results in Theorems~\ref{thm:singular value and eigenvalue estimates} and~\ref{thm:uniform_stability_case} below. 

\subsection{Quantitative estimates on singular values and eigenvalues} Using the flow space characterization, we can obtain quantitative estimates on singular values and eigenvalues. 

Given $g \in \SL(d,\Kb)$ let 
$$
\lambda_1(g) \geq \cdots \geq \lambda_d(g) 
$$
denote the absolute values of the eigenvalues of $g$ and let 
$$
 \mu_1(g) \geq \dots \geq \mu_d(g)
$$
denote the singular values of $g$. 

Also, given a metric space $X$ and an isometry $g \in \Isom(X)$ we define 
$$
\ell_X(g) := \lim_{n \to \infty} \frac{1}{n} \dist_X(g^n(x_0), x_0)
$$
where $x_0 \in X$ is some (any) point. 

\begin{theorem}[see Sections~\ref{sec:contraction_Anosov} and~\ref{sec:proof of singular value and eigenvalue estimates}]\label{thm:singular value and eigenvalue estimates} If $(\Gamma,\peripherals)$ is relatively hyperbolic, $X$ is a Groves--Manning cusp space for $(\Gamma, \peripherals)$, $x_0 \in X$, and $\rho \colon \Gamma \to \SL(d,\Kb)$ is $\Psf_k$-Anosov relative to $\peripherals$, then:
\begin{itemize}
\item There exist $\alpha,\beta > 0$ such that: if $\gamma \in \Gamma$, then 
$$
-\beta+\alpha d_X(\gamma(x_0), x_0) \leq \log \frac{\mu_k}{\mu_{k+1}}(\rho(\gamma))
$$
and 
$$
\alpha \ell_X(\gamma) \leq \log \frac{\lambda_k}{\lambda_{k+1}}(\rho(\gamma)).
$$
Moreover, we can choose $\alpha, \beta$ to be constant on a sufficiently small neighborhood of $\rho$ in $\Hom_{\rho}(\Gamma, \SL(d,\Kb))$. 
\item For any $p_0$ in the symmetric space $\SL(d,\Kb)/ \SU(d,\Kb)$ the orbits $\Gamma(x_0)$ and $\rho(\Gamma)(p_0)$ are quasi-isometric. Further, the quasi-isometry constants can be chosen to be constant on a sufficiently small neighborhood of $\rho$ in $\Hom_{\rho}(\Gamma, \SL(d,\Kb))$. 
\end{itemize} 
\end{theorem} 

The fact that orbits in a Groves--Manning cusp space and the symmetric space $\SL(d,\Kb)/ \SU(d,\Kb)$ are quasi-isometric is somewhat surprising since one can construct weak cusp spaces $X^\prime$ for $(\Gamma, \peripherals)$ where the $\Gamma$ orbits in $X^\prime$ are not quasi-isometric to the $\Gamma$ orbits in a Groves--Manning cusp space (this follows from the proof of Theorem B in~\cite{Healy2020}).  

We also note that the singular value gap estimate in Theorem~\ref{thm:singular value and eigenvalue estimates} and well-known distance estimates for Groves--Manning cusp spaces, see Proposition~\ref{prop:cusp_space_est} below, imply the following growth condition for the peripheral subgroups. 

\begin{corollary} Suppose that $(\Gamma,\peripherals)$ is relatively hyperbolic and $\rho \colon \Gamma \to \SL(d,\Kb)$ is $\Psf_k$-Anosov relative to $\peripherals$. Then for every $P \in \peripherals$ and finite symmetric generating set $S$ of $P$ there exist $\alpha, \beta > 0$ such that 
$$
 -\beta + \alpha \log \abs{\gamma}_S \leq \log \frac{\mu_k}{\mu_{k+1}}(\rho(\gamma)) 
$$
for all $\gamma \in P$. 
\end{corollary}

As a further corollary, Theorem~\ref{thm:singular value and eigenvalue estimates} implies the following equivalence between relatively Anosov representations (in the sense of Definition~\ref{defn:Pk Anosov}) and the relatively dominated representations introduced by the first author in~\cite{reldomreps}. This equivalence was previously known assuming some technical assumptions on the peripheral subgroups (which now follow from Theorem~\ref{thm:singular value and eigenvalue estimates}). 

\begin{corollary}[see Corollary~\ref{cor:relAnosov = reldom in body}] \label{cor:relAnosov = reldom}
Suppose that $(\Gamma, \peripherals)$ is relatively hyperbolic and $\rho\colon \Gamma \to \SL(d,\Kb)$ is a representation. Then the following are equivalent: 
\begin{enumerate}
\item $\rho$ is $\Psf_k$-Anosov relative to $\peripherals$.
\item $\rho$ is $\Psf_k$-dominated relative to $\peripherals$ in the sense of ~\cite{reldomreps}. 
\end{enumerate}
\end{corollary}

\subsection{Locally uniform norms}\label{sec:introducing uniformly Anosov} In the relatively hyperbolic case, the space $\wh{\Gc}(X)$ will be non-compact and thus it is possible for a metric on the vector bundle $\wh{E}_\rho(X) \to \wh{\Gc}(X)$ to be quite badly behaved. 

We introduce subclasses of relatively Anosov representation based on the regularity properties of the metric on the bundle. The following technical definition is inspired by the so-called admissible metrics studied in~\cite[Def.\ 5.17]{Shub1987}.

\begin{definition}\label{defn:locally uniform norms} Suppose that $(\Gamma,\peripherals)$ is relatively hyperbolic, $X$ is a weak cusp space for $(\Gamma, \peripherals)$, and $\rho \colon \Gamma \to \SL(d,\Kb)$ is a representation. A metric $\norm{\cdot}$ on $\wh{E}_\rho(X) \to \wh{\Gc}(X)$ is \emph{locally uniform} if its lift to $\Gc(X) \times \Kb^d \to \Gc(X)$ has the following property:
\begin{itemize}
\item For any $r > 0$ there exists $L_r > 1$ such that:
\begin{align*}
\frac{1}{L_r} \norm{\cdot}_{\sigma_1} \leq \norm{\cdot}_{\sigma_2} \leq L_r \norm{\cdot}_{\sigma_1}
\end{align*}
for all $\sigma_1, \sigma_2 \in \Gc(X)$ with $\dist_X(\sigma_1(0), \sigma_2(0)) \leq r$. 
\end{itemize}
\end{definition} 

\begin{definition}\label{defn:uniformly Anosov} Suppose that $(\Gamma,\peripherals)$ is relatively hyperbolic and $\rho \colon \Gamma \to \SL(d,\Kb)$ is $\Psf_k$-Anosov relative to $\peripherals$. If $X$ is a weak cusp space for $(\Gamma, \peripherals)$, then $\rho$ is \emph{uniformly $\Psf_k$-Anosov relative to $X$} if there exists a locally uniform metric $\norm{\cdot}$ on $\wh{E}_\rho(X) \to \wh{\Gc}(X)$ such that the flow $\homflow^t$ on $\Hom(\wh{\Xi}^{d-k}, \wh{\Theta}^k)$ is exponentially contracting (with respect to the associated operator norms).
\end{definition}

The next theorem will show that uniformly Anosov representations are very nicely behaved. In particular, there is an equivariant quasi-isometric embedding of the entire weak cusp space into the symmetric space and the boundary map is H\"older regular relative to any visual metric on the Bowditch boundary and Riemannian distance on the Grassmanian. 

In Example~\ref{ex:non uniform to GM cusp space}  we will describe a relatively Anosov representation which is not uniform relative to any Groves--Manning cusp space, but is uniform relative to some weak cusp space. This shows that there is value in studying bundles associated to general weak cusp spaces and in future work we will further explore how to select the ``best'' weak cusp space to study a given relatively Anosov representation.

\begin{theorem}[see Sections~\ref{sec: uniformly Anosov} and~\ref{sec:uniform stability}]\label{thm:uniform_stability_case} Suppose that $(\Gamma,\peripherals)$ is relatively hyperbolic, $X$ is a weak cusp space for $(\Gamma, \peripherals)$, and $\rho_0 \colon \Gamma \to \SL(d,\Kb)$ is uniformly $\Psf_k$-Anosov relative to $X$. Then there exists an open neighborhood $\Oc \subset \Hom_{\rho_0}(\Gamma, \SL(d,\Kb))$ of $\rho_0$ where every $\rho \in \Oc$ is uniformly $\Psf_k$-Anosov relative to $X$. 

Moreover:
\begin{enumerate}
\item If $\rho \in \Oc$, then there exists a $\rho$-equivariant quasi-isometric embedding 
$$
X \to \SL(d,\Kb) / \SU(d,\Kb).
$$
Further, the quasi-isometry constants can be chosen to be constant on $\Oc$. 

\item If $\rho \in \Oc$, then the Anosov boundary map 
$$
\xi_\rho \colon \partial_\infty X \to \Gr_k(\Kb^d) \times \Gr_{d-k}(\Kb^d)
$$
is H\"older relative to any visual metric on $\partial_\infty X$ and any Riemannian distance on $\Gr_k(\Kb^d) \times \Gr_{d-k}(\Kb^d)$. Further, the  H\"older constants can be chosen to be constant on $\Oc$. 

\end{enumerate}
\end{theorem}

Part (1) of Theorem~\ref{thm:uniform_stability_case} shows that uniformly relatively Anosov representations are similar to the relatively Morse representations introduced in~\cite{KL}. In fact, in Section~\ref{sec:relMorse_relunifAnosov} we will show that the two notions essentially coincide.

\subsection{Outline of the paper and proofs} 

\subsubsection{Expository sections} 
Sections \ref{sec:examples}, \ref{sec:prelims}, and \ref{sec:rem_var}  are expository in nature. In Section~\ref{sec:examples} we describe some of the examples of (uniformly) relatively Anosov representations that we construct in the sequel to this paper. Section~\ref{sec:prelims} is devoted to describing the background material needed for our proofs. 

Section \ref{sec:rem_var} is devoted to some basic observations about Definition~\ref{defn:Pk Anosov}. In particular, we explain why this definition is equivalent to one of Kapovich-Leeb's notion of relatively Anosov representations and we also describe why the main results of this paper imply that Definition~\ref{defn:Pk Anosov} is equivalent to the notion of relatively Anosov representations introduced by the first author in~\cite{reldomreps}. 

\subsubsection{Weakly unipotent groups} Sections~\ref{sec:char of type-preserving representations}, \ref{sec: growth rates for rational functions}, and \ref{sec: structure of weakly unipotent groups} are devoted to studying \emph{weakly unipotent groups},  that is a linear group where every element $g$ in the group satisfies 
$$
\lambda_1(g) = \dots = \lambda_d(g) = 1.
$$
(recall that the $\lambda_j(g)$ denote the absolute values of the eigenvalues of $g$). 

In Proposition~\ref{prop:eigenvalue data in rel Anosov repn} we observe that the image of a peripheral subgroup under a relatively Anosov representation is always a weakly unipotent group. A key part of  this paper is developing some structure theory for weakly unipotent discrete groups and in particular establishing singular value estimates in terms of word length. 

In Section~\ref{sec:char of type-preserving representations} we establish the following characterization of representations of relatively hyperbolic groups whose images of peripherals subgroups are weakly unipotent. 

\begin{proposition}[see Proposition~\ref{prop:type-preserving<=>good upper bound on singular values}]
Suppose that $(\Gamma, \peripherals)$ is relatively hyperbolic,  $X$ is a Groves--Manning cusp space for $(\Gamma, \peripherals)$, and $\rho \colon \Gamma \to \SL(d,\Kb)$ is a representation. Then the following are equivalent: 
\begin{enumerate}
\item $\rho(P)$ is weakly unipotent for every $P \in \peripherals$.
\item For any $x_0 \in X$, there exist $\alpha, \beta > 0$ such that
\begin{equation}
\label{eqn:singular value bounds in the outline of paper section}
\log  \frac{\mu_1}{\mu_d}(\rho(\gamma)) \leq \alpha \dist_X(\gamma(x_0), x_0) +\beta
\end{equation}
for all $\gamma \in \Gamma$. 
\end{enumerate}
\end{proposition} 

Note that Proposition~\ref{prop:eigenvalue data in rel Anosov repn} implies that the estimate in Equation~\eqref{eqn:singular value bounds in the outline of paper section} holds for relatively Anosov representations. 

The image of a relative $\Psf_k$-Anosov representation is \emph{$\Psf_k$-divergent}, that is for any escaping sequence $(g_n)_{n \geq 1}$ in the image, the ratio $\frac{\mu_k}{\mu_{k+1}}(g_n)$ converges to infinity (see Observation~\ref{obs:strongly_dynamics_pres_div_cartan}). In Section~\ref{sec: structure of weakly unipotent groups}, we study the structure of weakly unipotent discrete groups and in particular establish the following uniform growth condition on a $\Psf_k$-divergent discrete weakly unipotent group.

\begin{theorem}[see Theorem~\ref{thm:structure of weakly unipotent discrete groups}]\label{thm:structure of weakly unipotent discrete groups in intro}
Suppose that $\Gamma \leq \SL(d,\Rb)$ is a weakly unipotent discrete group. If $\Gamma$ is $\mathsf{P}_k$-divergent and $S$ is a finite symmetric generating set of $\Gamma$, then there exist $\alpha, \beta > 0$ such that 
$$
\log  \frac{\mu_k}{\mu_{k+1}}(\gamma) \geq \alpha \log \abs{\gamma}_S +\beta
$$
for all $\gamma \in \Gamma$. 
\end{theorem} 

The proof of this estimate relies on studying real rational functions $R \colon \Rb^d \to \Rb$ which extend continuously to all of $\Rb^d$ and uses a recent version of the Nullstellensatz for such functions established in~\cite{FHMM2016}, see Section~\ref{sec: growth rates for rational functions} for details. 

\subsubsection{Proof of Theorem~\ref{thm:main}} As mentioned in Remark~\ref{rmk:directions in the proof of theorem main}, the implication (3) $\implies$ (2) is by definition, and  ``standard arguments'' (e.g.\ as in~\cite{Canary_notes})  imply that (2) $\implies$ (1). In Section~\ref{sec:contraction_Anosov} we explain these standard arguments.  

The proof that  (1) $\implies$ (3) is considerably more complicated and involves carefully constructing norms on the fibers above the ``cusps'' in the flow space. This is accomplished in Section~\ref{sec:building_norms} and requires the singular value estimate in Theorem~\ref{thm:structure of weakly unipotent discrete groups in intro}. 

\subsubsection{Proof of Theorem~\ref{thm:stability}} We prove Theorem~\ref{thm:stability} in Section~\ref{sec:stability}. The proof has three main steps, the first two closely follow the arguments in~\cite{CZZ2021} for geometrically finite Fuchsian groups while the third is more complicated due to the more general setting. 

\subsubsection{Proof of Theorem~\ref{thm:singular value and eigenvalue estimates}} We prove Theorem~\ref{thm:singular value and eigenvalue estimates} for a single representation in Section~\ref{sec:contraction_Anosov}. Showing that the orbits are quasi-isometric requires the estimate in Equation~\eqref{eqn:singular value bounds in the outline of paper section}. Later in Section~\ref{sec:proof of singular value and eigenvalue estimates} we explain why the proof of Theorem~\ref{thm:stability} implies that the constants can be chosen to be constant under sufficiently small type-preserving deformations.

\subsubsection{Proof of Theorem~\ref{thm:uniform_stability_case}} We prove Theorem~\ref{thm:uniform_stability_case} for a single representation in Section~\ref{sec: uniformly Anosov}. Later in Section~\ref{sec:uniform stability} we explain why the proof of Theorem~\ref{thm:stability} implies that the constants can be chosen to be constant under sufficiently small type-preserving deformations.

\subsubsection{The appendices} In Appendix~\ref{sec: proofs for section SVD background}, we prove some linear algebra observations which are stated in Sections~\ref{sec:SVD background} and~\ref{sec: background on proximal and weakly unipotent elements}. In Appendix~\ref{appendix: basic properties}, we prove some (probably well-known) facts about Gromov-hyperbolic metric spaces.

\subsection*{Acknowledgements} 
The authors thank Fanny Kassel and  Ilia Smilga for pointing out a mistake in the first version of this paper. 

Zhu was partially supported by Israel Science Foundation grants 18/171 and 737/20.
Zimmer was partially supported by grants DMS-2105580 and DMS-2104381 from the National Science Foundation.

\section{Examples}\label{sec:examples}

In this section we summarize some results from the sequel to this paper, where we will explore a variety of particular examples. 

\subsection{Representations of geometrically finite groups} Suppose $X$ is a negatively-curved symmetric space and let $\mathsf{G} := \Isom_0(X)$,  the connected component of the identity in the isometry group of $X$. Let $\partial_\infty X$ denote the geodesic boundary of $X$. Then given a discrete group $\Gamma \leq \mathsf{G}$, let $\Lambda_X(\Gamma) \subset \partial_\infty X$ denote the limit set of $\Gamma$ and let $\Cc_X(\Gamma)$ denote the convex hull of the limit set in $X$. 

When $\Gamma \leq \mathsf{G}$ is geometrically finite, we will let $\peripherals(\Gamma)$ denote a set of representatives of the conjugacy classes of maximal parabolic subgroups in $\Gamma$. Then $(\Gamma,\peripherals(\Gamma))$ is relatively hyperbolic and $\Cc_X(\Gamma)$ is a weak cusp space for $(\Gamma,\peripherals(\Gamma) )$. 

We will observe that restricting a proximal linear representation of $\mathsf{G}$ to a geometrically finite subgroup produces a uniformly relatively Anosov representation. 

\begin{proposition}[{\cite[Prop.\ 1.7]{ZZ2022b}}] \label{prop:homog_cusps} Suppose that $\tau \colon \mathsf{G} \to \SL(d,\Kb)$ is $\Psf_k$-proximal (i.e.\ $\tau(\mathsf{G})$ contains a $\Psf_k$-proximal element). If $\Gamma \leq \mathsf{G}$ is geometrically finite, then $\rho := \tau|_{\Gamma}$ is uniformly $\Psf_k$-Anosov relative to $\Cc_X(\Gamma)$.
 \end{proposition}

 In the context of Proposition~\ref{prop:homog_cusps}, we can obtain additional examples by starting with the representation $\rho_0 : =\tau|_\Gamma$ and deforming it in $\Hom_{\rho_0}(\Gamma, \SL(d,\Kb))$. By Theorem~\ref{thm:stability}, any sufficiently small deformation will be a relatively Anosov representation. 
 
 Using Proposition~\ref{prop:homog_cusps} we will also construct the following example. 
 
 \begin{example}[{\cite[Ex.\ 1.8]{ZZ2022b}}] \label{ex:non uniform to GM cusp space} Let $X:=\Hb^2_{\Cb}$ denote complex hyperbolic 2-space. There exists a geometrically finite subgroup $\Gamma \leq \Isom_0(X)$ and a representation $\rho\colon \Gamma \to \SL(3,\Cb)$ which is uniformly $\Psf_1$-Anosov relative to $\Cc_{X}(\Gamma)$, but not uniformly $\Psf_1$-Anosov relative to any Groves--Manning cusp space for $(\Gamma,\peripherals(\Gamma) )$. 
 \end{example} 
 
We can relax the condition in Proposition~\ref{prop:homog_cusps} to only assuming that the representation extends on each peripheral  subgroup. More precisely, if $\Gamma \leq \mathsf{G}$ is geometrically finite and  $\rho \colon \Gamma \to \SL(d,\Kb)$ is $\Psf_k$-Anosov relative to $\peripherals(\Gamma)$, then we say that $\rho$ has \emph{almost homogeneous cusps} if there exists a finite cover $\pi\colon \widetilde{\mathsf{G}} \to \mathsf{G}$ such that for each $P \in\peripherals(\Gamma)$ there is a representation $\tau_P\colon \widetilde{\mathsf{G}} \to \SL(d,\Kb)$ where 
\begin{align*}
\left\{ \tau_P(g)(\rho\circ \pi)(g)^{-1} : g \in \pi^{-1}(P)\right\}
\end{align*}
is relatively compact in $\SL(d,\Kb)$. This technical definition informally states that the representation restricted to each peripheral subgroup extends to a representation of $\mathsf{G}$. 

\begin{theorem}[{\cite[Th.\ 1.9]{ZZ2022b}}] Suppose that $\Gamma \leq\mathsf{G}$ is geometrically finite and  $\rho \colon \Gamma \to \SL(d,\Kb)$ is $\Psf_k$-Anosov relative to $\peripherals(\Gamma)$. If $\rho$ has almost homogeneous cusps, then $\rho$ is uniformly $\Psf_k$-Anosov relative to $\Cc_X(\Gamma)$.
\end{theorem} 

Proposition 3.6 in~\cite{CZZ2021} implies that every relatively Anosov representation of a geometrically finite Fuchsian group has almost homogeneous cusps and hence is uniform. This also follows from the construction of canonical norms in~\cite[Sec.\ 3.1]{CZZ2021}. 

\begin{corollary}[{\cite[Cor.\ 1.10]{ZZ2022b}}] \label{cor:geometrically finite Fuchsian groups} If $X = \Hb^2_{\Rb}$ is real hyperbolic 2-space, $\Gamma \leq \Isom_0(X)$ is geometrically finite, and $\rho \colon \Gamma \to \SL(d,\Kb)$ is $\Psf_k$-Anosov relative to $\peripherals(\Gamma)$, then $\rho$ is uniformly $\Psf_k$-Anosov relative to $\Cc_X(\Gamma)$.
\end{corollary}

\subsection{Visible subgroups in real projective geometry} 
\label{subsec:visible subgroups in intro}

We will also apply our general results to the setting of convex real projective geometry.

Given a properly convex domain $\Omega \subset \proj(\Rb^d)$, the \emph{automorphism group of $\Omega$}, denoted $\Aut(\Omega)$, is the subgroup of $\PGL(d,\Rb)$ which preserves $\Omega$. The \emph{limit set} of a subgroup $\Gamma \leq \Aut(\Omega)$ is defined to be 
$$
\Lambda_\Omega(\Gamma) := \partial \Omega \cap \bigcup_{p \in \Omega} \overline{\Gamma(p)}.
$$
Following~\cite{CZZ2022}, we say that $\Gamma$ is a \emph{visible subgroup of $\Aut(\Omega)$} if 
\begin{enumerate}
\item for all $p,q \in \Lambda_\Omega(\Gamma)$ distinct, the open line segment in $\overline{\Omega}$ joining $p$ to $q$ is contained in $\Omega$.
\item every point in $\Lambda_\Omega(\Gamma)$ is a $\Cc^1$-smooth point of $\partial \Omega$.
\end{enumerate} 
A visible subgroup acts as a convergence group on its limit set and if, in addition, the action on the limit set is geometrically finite then the inclusion representation is relatively $\Psf_1$-Anosov (these assertions follow from~\cite[Prop. 3.5]{CZZ2022}). 

Using the methods in~\cite{DGK_convex_cocopt_realproj} and ~\cite{Z2017}, we will construct the following examples. 

\begin{proposition}[{\cite[Prop.\ 1.16]{ZZ2022b}}] \label{prop:examples of visible groups} Suppose that $X$ is a negatively-curved symmetric space which is not isometric to real hyperbolic 2-space and $\mathsf{G} := \Isom_0(X)$. If $\tau \colon \mathsf{G} \to \PGL(d,\Rb)$ is $\Psf_1$-proximal, then there exists a properly convex domain $\Omega \subset \proj(\Rb^d)$ where $\tau(\mathsf{G}) \leq \Aut(\Omega)$ and if $\Gamma \leq \mathsf{G}$ is geometrically finite, then $\tau(\Gamma) \leq \Aut(\Omega)$ is a visible subgroup which acts geometrically finitely on its limit set.
\end{proposition} 

As an application of Theorem~\ref{thm:stability}, we will prove the following stability result. 

\begin{theorem}[{\cite[Cor.\ 1.15]{ZZ2022b}}] \label{thm:stability of visible groups} Suppose that $\Gamma \leq \Aut(\Omega)$ is a visible subgroup acting geometrically finitely on its limit set and $\iota \colon \Gamma \into \PGL(d,\Rb)$ is the inclusion representation. Then there is an open neighborhood $\Oc \subset \Hom_{\iota}(\Gamma, \PGL(d,\Rb))$ of $\iota$ such that: if $\rho \in \Oc$, then there exists a properly convex domain $\Omega_\rho \subset \proj(\Rb^d)$ where $\rho(\Gamma) \leq \Aut(\Omega_\rho)$ is a visible subgroup acting geometrically finitely on its limit set.
\end{theorem} 

\begin{remark} For other stability results in the context of convex real projective geometry, see~\cite{Koszul1968,BenoistIII,Marquis2010,CLT2018,Choi}. \end{remark}

\section{Preliminaries}\label{sec:prelims}

\subsection{Ambiguous notation} Here we fix any possibly ambiguous notation. 
\begin{itemize}
\item We let $\norm{\cdot}_2$ denote the standard Euclidean norm on $\Kb^d$ and let $e_1,\dots,e_d$ denote the standard basis of $\Kb^d$. 
\item A \emph{metric} $\norm{\cdot}$ on a vector bundle $V \to B$ is a continuous varying family of norms on the fibers each of which is induced by an inner product. 
\item Given a metric space $X$, we will use $\Bc_X(p,r)$ to denote the open ball of radius $r$ centered at $p \in X$ and $\Nc_X(A,r)$ to denote the $r$-neighborhood of a subset $A \subset X$. 
\item Given functions $f,g \colon S \to \Rb$ we write $f \lesssim g$ or equivalently $g \gtrsim f$ if there exists a constant $C > 0$ such that $f(s) \leq C g(s)$ for all $s \in S$. If $f \lesssim g$ and $g \lesssim f$, then we write $f \asymp g$. 

\item Except where otherwise specified, all logarithms are taken to base $e$.
\item Note that constants often carry over between statements in the same section, but not across sections.
\end{itemize}

\subsection{Convergence groups} 

When $M$ is a compact perfect metrizable space, a subgroup $\Gamma \leq {\rm Homeo}(M)$ is called a \emph{convergence group} if for every sequence $(\gamma_n)_{n \geq 1}$ of distinct elements in $\Gamma$, there exist $x,y \in M$ and a subsequence $(\gamma_{n_j})_{j \geq 1}$ such that $\gamma_{n_j}|_{M \smallsetminus \{y\}}$ converges locally uniformly to the constant map $x$. In this case, an element of $\Gamma$ is either 
\begin{itemize}
\item \emph{elliptic}, that is it has finite order, 
\item \emph{parabolic}, that is it has infinite order and fixes exactly one point in $M$, or 
\item \emph{loxodromic}, that is it has infinite order and fixes exactly two points in $M$.
\end{itemize} 
Parabolic and loxodromic elements have the following behavior:
\begin{enumerate}
\item If $g \in \Gamma$ is parabolic and $x$ is the unique fixed point of $g$, then 
$$
\lim_{n \to \pm \infty} g^n(y) =x
$$ 
for all $y \in M \smallsetminus \{x\}$. 
\item If $g \in \Gamma$ is loxodromic, then it is possible to label the fixed points of $g$ as $x^+, x^-$ so that 
$$
\lim_{n \to \pm \infty} g^n(y) =x^\pm
$$
for all $y \in M \smallsetminus \{x^\mp\}$. 
\end{enumerate}
In both cases, the limits are locally uniform.

\subsection{Relatively hyperbolic groups}\label{sec: defining rel hyp groups} We now recall the definition of relatively hyperbolic groups. There are a number of equivalent definitions, here we give one based on the action of the group on a suitable boundary space. For more background and other definitions see \cite{Bowditch_relhyp,GrovesManning,Osin,Yaman,DS2005}.

Suppose that $M$ is a compact perfect metrizable space and $\Gamma \leq {\rm Homeo}(M)$ is a convergence group, then: 
\begin{itemize}
 \item A point $x \in M$ is a \emph{conical limit point} if there exist $a, b \in M$ distinct and a sequence $(\gamma_n)_{n \geq 1}$ in $\Gamma$ such that $\gamma_n(x) \to a$ and $\gamma_n(y) \to b$ for any $y \in M \smallsetminus \{x\}$. 
 \item A infinite order subgroup $H \leq \Gamma$ is \emph{parabolic} if it fixes some point of $M$ and each infinite order element in $H$ is parabolic. The fixed point of a parabolic subgroup is called a \emph{parabolic point}.
 \item A parabolic point $x \in M$ is \emph{bounded} if the quotient $\Stab_\Gamma(x) \backslash (M \smallsetminus \{x\})$ is compact.

\end{itemize}
Finally, $\Gamma$ is called a \emph{geometrically finite convergence group} if every point in $M$ is either a conical limit point or a bounded parabolic point.

\begin{definition}\label{defn:RH}
Given a finitely generated group $\Gamma$ and a collection $\peripherals$ of finitely generated infinite subgroups, we say that $\Gamma$ is \emph{hyperbolic relative to $\peripherals$}, or that $(\Gamma,\peripherals)$ is \emph{relatively hyperbolic}, if $\Gamma$ acts on a compact perfect metrizable space $M$ as a geometrically finite convergence group and the maximal parabolic subgroups are exactly the set
$$
\peripherals^\Gamma := \{ \gamma P \gamma^{-1} : P \in \peripherals, \gamma \in \Gamma\}.
$$
\end{definition}

\begin{remark}
Notice that by definition we assume that a relatively hyperbolic group is non-elementary (i.e.\  $M$ is infinite) and finitely generated. \end{remark}

By a theorem of Bowditch \cite{Bowditch_relhyp}, given a relatively hyperbolic group $(\Gamma,\peripherals)$, any two compact perfect metrizable spaces satisfying Definition~\ref{defn:RH} are $\Gamma$-equivariantly homeomorphic. This unique topological space is then denoted by $\partial(\Gamma,\peripherals)$ and called the \emph{Bowditch boundary of $(\Gamma, \peripherals)$}. 

If a group $\Gamma$ acts properly discontinuously and by isometries on a proper geodesic Gromov-hyperbolic metric space $X$, then the action of $\Gamma$ on the Gromov boundary $\partial_\infty X$ is a convergence group action~\cite{Bowditch_convergence_grps}. 
As the next definition and theorem make precise, one can always assume that the space $M$ in Definition~\ref{defn:RH} is the boundary of such a metric space. 

\begin{definition}\label{defn:weak cusp space} Suppose that  $(\Gamma,\peripherals)$ is relatively hyperbolic and $\Gamma$ acts properly discontinuously and by isometries on a proper geodesic Gromov-hyperbolic metric space $X$. If 
\begin{enumerate}
\item $\partial_\infty X$ satisfies Definition~\ref{defn:RH} and
\item every point in $X$ is within a uniformly bounded distance of a geodesic line, 
\end{enumerate}
then $X$ is a \emph{weak cusp space for $(\Gamma,\peripherals)$}. 
\end{definition} 

By work of Bowditch \cite{Bowditch_relhyp} (also see the exposition in \cite[Section 3]{HealyHruska}), one can alternatively define weak cusp spaces in terms of the action of $\Gamma$ on $X$. 

The main result in~\cite{Yaman}   implies the following. 

\begin{theorem}Any relatively hyperbolic group has a weak cusp space. 
\end{theorem}

For future use, we note that condition (2) in Definition~\ref{defn:weak cusp space} implies the following stronger density result for geodesic lines. 

\begin{proposition}\label{prop:abundance of geodesic lines} Suppose that $(\Gamma,\peripherals)$ is relatively hyperbolic and $X$ is a weak cusp space for $(\Gamma,\peripherals)$. Then there exists $R > 0$ such that: for any $p,q \in X$ there is a geodesic line $\sigma : \Rb \rightarrow X$ with 
$$
p,q \in \Nc_X(\sigma, R).
$$
\end{proposition}

\begin{proof} This follows from Lemma~\ref{lem:finding_axes}. \end{proof} 

The Bowditch boundary $\partial(\Gamma,\peripherals)$ can be used to compactify $\Gamma$ by saying that a sequence $(\gamma_n)_{n \geq 1}$ in $\Gamma$ converges to $x \in \partial(\Gamma,\peripherals)$ if for every subsequence $(\gamma_{n_j})_{j \geq 1}$ there exist $y \in \partial(\Gamma,\peripherals)$ and a further subsequence $(\gamma_{n_{j_k}})_{k \geq 1}$ such that $\gamma_{n_{j_k}}|_{M \smallsetminus \{y\}}$ converges locally uniformly to the constant map $x$. In this case we write $\gamma_n \to x$. If we identify $\partial(\Gamma,\peripherals)$ with the Gromov boundary $\partial_\infty X$ of a weak cusp space $X$, then $\gamma_n \to x$ if and only if $\gamma_n(p) \to x$ for some (any) $p \in X$. 

\subsection{The Groves--Manning cusp space}  A relatively hyperbolic group can have non-quasi-isometric weak cusp spaces, see~\cite{Healy2020}, but perhaps the most canonical  is a construction due to Groves--Manning. As we describe below, this is obtained by attaching combinatorial horoballs to the standard Cayley graph. 

\begin{definition} Suppose $Y$ is a graph with the simplicial distance $\dist_Y$. The \emph{combinatorial horoball} $\Hc(Y)$ is the graph, also equipped with the simplicial distance, that has vertex set $Y^{(0)} \times \Nb$ and two types of edges:
	\begin{itemize}
		\item \emph{vertical edges} joining vertices $(v,n)$ and $(v,n+1)$, 
		\item \emph{horizontal edges} joining vertices $(v,n)$ and $(w,n)$ when $d_Y(v,w) \leq 2^{n-1}$. 
	\end{itemize}
\end{definition} 

\begin{definition} \label{def:cusp spaces} \label{defn: cusped Cayley graph}
	Suppose that $(\Gamma,\peripherals)$ is relatively hyperbolic. A finite symmetric generating set $S \subset \Gamma$ is \emph{adapted} if $S \cap P$ is a generating set of $P$ for every $P \in \peripherals$. Given such an $S$, we let $\Cc(\Gamma, S)$ and $\Cc(P, S \cap P)$ denote the associated Cayley graphs. Then the associated \emph{Groves--Manning cusp space}, denoted $\Cc_{GM}(\Gamma, \peripherals, S)$, is obtained from the Cayley graph $\Cc(\Gamma, S)$ by attaching, for each $P \in \peripherals$ and $\gamma \in \Gamma$, a copy of the combinatorial horoball $\Hc( \gamma\Cc(P, S \cap P))$  by identifying $\gamma\Cc(P, S \cap P)$ with the $n=1$ level of $\Hc( \gamma\Cc(P, S \cap P))$.

\end{definition}

\begin{theorem}[{\cite[Th.\ 3.25]{GrovesManning}}] If $(\Gamma, \peripherals)$ is relatively hyperbolic and $S$ is an adapted finite generating set, then $\Cc_{GM}(\Gamma, \peripherals, S)$ is a weak cusp space for $(\Gamma, \peripherals)$. 
\end{theorem}

We will use the following well-known distance estimate in the Groves--Manning cusp space.

\begin{proposition}\label{prop:cusp_space_est} Suppose that $(\Gamma, \peripherals)$ is relatively hyperbolic, $S$ is an adapted finite generating set, and   $X:=\Cc_{GM}(\Gamma, \peripherals, S)$. For any $x_0 \in X$ there exists $\beta > 0$ such that: if $P \in \peripherals$ and  $g \in P \smallsetminus \{\id\}$, then
	$$
	-\beta+2\log_2 \abs{g}_{S \cap P} \leq  \dist_{X}(g(x_0), x_0)  \leq \beta+ 2\log_2 \abs{g}_{S \cap P}.
	$$

\end{proposition} 

\begin{proof} For $P \in  \peripherals$ and $L \geq 1$, let $\Hc_P(L)\subset X$ denote the induced subgraph of the associated combinatorial horoball with vertex set
$$
\{ (g,n) : g \in P, n \geq L\}. 
$$
By~\cite[Lem.\ 3.26]{GrovesManning} there exists $\delta \geq 1$ such that each $\Hc_P(\delta)$ is geodesically convex in $X$. 

It suffices to consider the case when $x_0 = \id$. Fix $P \in \peripherals$ and  $g \in P \smallsetminus \{\id\}$. For the upper bound, let $n :=  1+\lceil \log_2 \abs{g}_{S \cap P}\rceil$. Then 
$$
\dist_X(g, \id) \leq 2(n-1) + \dist_X\big((g,n), (\id, n)\big) \leq 2n-1 \leq 3+2\log_2\abs{g}_{S \cap P}. 
$$
To prove the lower bound we use~\cite[Lem.\ 3.10]{GrovesManning}, which implies that there exists a geodesic in $\Hc_P(\delta)$ joining $(\id, \delta)$ to $(g,\delta)$ which consists of $m$ vertical edges, followed by no more than three horizontal edges, followed by $m$ vertical edges. Then $\abs{g}_{S \cap P} \leq 3 \cdot 2^{m-1}$ and since $\Hc_P(\delta)$ is geodesically convex
\begin{align*}
\dist_X(g,\id) & \geq  \dist_X\big( (g,\delta), (\id, \delta)\big)-2\delta \geq 2m-2\delta   \\
& \geq -2 \log_2(3) + 2 - 2\delta+2\log_2\abs{g}_{S \cap P}.
\end{align*}
Since $\delta$ is independent of $P$ and $g$, this completes the proof. 
\end{proof}

\subsection{The geometry of the Grassmanians} Throughout the paper, we will let $\dist_{\proj(\Kb^d)}$ denote the \emph{angle distance} on $\proj(\Kb^d)$, that is: if $\ip{\cdot,\cdot}$ is the standard Euclidean inner product on $\Kb^d$, then 
$$
\dist_{\proj(\Kb^d)}([v], [w]) = \cos^{-1} \left( \frac{\abs{\ip{v,w}}}{\sqrt{\ip{v,v}}\sqrt{\ip{w,w}}} \right)
$$
for all non-zero $v,w \in \Kb^d$. 

Using the Pl\"ucker embedding, we can view $\Gr_k(\Kb^d)$ as a subset of $\proj(\wedge^k \Kb^d)$. Let $\dist_{\proj(\wedge^k \Kb^d)}$ denote the angle distance associated to the inner product on $\wedge^k \Kb^d$ which makes 
$$
\{ e_{i_1} \wedge \dots \wedge e_{i_k} : i_1 < \dots < i_k\}
$$
an orthonormal basis. We then let $\dist_{\Gr_k(\Kb^d)}$ denote the distance on $\Gr_k(\Kb^d)$ obtained by restricting $\dist_{\proj(\wedge^k \Kb^d)}$.

\subsection{The singular value decomposition}\label{sec:SVD background} By the singular value decomposition, any element $g \in \SL(d,\Kb)$ can be written as $g = m a \ell$ where $m,\ell \in \SU(d,\Kb)$ and $a$ is a diagonal matrix with 
$$
\mu_1(g) \geq \dots \geq \mu_d(g)
$$
down the diagonal. In general this decomposition is not unique, but when $\mu_k(g) > \mu_{k+1}(g)$ the subspace 
$$
U_k(g) :=m \ip{e_1,\dots, e_k}
$$
is well defined. Geometrically, $U_k(g)$ is the subspace spanned by the $k$ largest axes of the ellipse $g\cdot\{ x \in \Kb^d : \norm{x}_2=1\}$. 

We will frequently use the following observation. 

\begin{observation}\label{obs:strongly_dynamics_pres_div_cartan}
Suppose that $(g_n)_{n \geq 1}$ is a sequence in $\SL(d,\Kb)$, $V_0 \in \Gr_k(\Kb^d)$, and $W_0 \in \Gr_{d-k}(\Kb^d)$. Then the following are equivalent: 
\begin{enumerate}
\item  $g_n(V) \to V_0$ uniformly on compact subsets of 
$$
\left\{ V \in \Gr_k(\Kb^d) : V \text{ transverse to } W_0\right\},
$$ 
\item $\frac{\mu_k}{\mu_{k+1}}(g_n) \to \infty$, $U_k(g_n) \to V_0$, and $U_{d-k}(g_n^{-1}) \to W_0$. 
\end{enumerate}
\end{observation}

\begin{proof} We provide a proof in Appendix~\ref{sec: proofs for section SVD background}. \end{proof} 

We will also use the following estimates for distances between the spaces $U_k(g)$ when considering products; for proofs see~\cite[Lem.\ A.4, A.5]{BPS}. 

\begin{lemma}\label{lem: U_k under product} Suppose that $g,h \in \GL(d,\Kb)$. 
\begin{enumerate}
\item If $\mu_k(g) > \mu_{k+1}(g)$ and $\mu_k(gh) > \mu_{k+1}(gh)$,  then 
$$
\dist_{\Gr_k(\Kb^d)}\big( U_k(gh), U_k(g)\big) \leq \frac{\mu_1}{\mu_d}(h) \frac{\mu_{k+1}}{\mu_k}(g).
$$
\item If $\mu_k(g) > \mu_{k+1}(g)$ and $\mu_k(hg) > \mu_{k+1}(hg)$,  then 
$$
\dist_{\Gr_k(\Kb^d)}\big( U_k(hg), hU_k(g) \big) \leq \frac{\mu_1}{\mu_d}(h) \frac{\mu_{k+1}}{\mu_k}(g).
$$
\end{enumerate}
\end{lemma} 

\subsection{Proximal and weakly unipotent elements}\label{sec: background on proximal and weakly unipotent elements} An element $g \in \SL(d,\Kb)$ is called \emph{$\Psf_k$-proximal} if $\lambda_k(g) > \lambda_{k+1}(g)$ (recall that $\lambda_j(g)$ denote the absolute values of the eigenvalues of $g$ listed in decreasing order).  In this case, there exists an $g$-invariant decomposition $\Kb^d = V^+_g \oplus W^-_g$ where $\dim_{\Kb} V_g^+= k$, $\dim_{\Kb} W_g^-=d-k$,
$$
\lambda_j(g|_{V^+_g}) = \lambda_j(g) \quad \text{for} \quad j=1,\dots,k,
$$
and 
$$
 \lambda_j( g|_{W^-_g}) = \lambda_{k+j}(g) \quad \text{for} \quad j=1,\dots, d-k.
$$
Further, 
$$
g^n(V) \to V^+_g
$$
for all $V \in \Gr_k(\Kb^d)$ transverse to $W^-_g$. In fact, as the next observation states, this dynamical behavior characterizes proximality (see Observation~\ref{obs:strongly_dynamics_pres_div_cartan}).

\begin{observation}\label{obs:dynamics of Pk proximal} If $g \in \SL(d,\Kb)$, then the following are equivalent: 
\begin{enumerate}
\item $g$ is $\mathsf{P}_k$-proximal,
\item there exist $V_0 \in \Gr_k(\Kb^d)$, $W_0 \in \Gr_{d-k}(\Kb^d)$ such that $V_0 \oplus W_0 = \Kb^d$ and
$$g^n(V) \to V_0$$ 
uniformly on compact subsets of $\left\{ V \in \Gr_k(\Kb^d) : V \text{ transverse to } W_0\right\}$.
\end{enumerate}
Moreover, if the above conditions are satisfied, then $V_0=V_g^+$ and $W_0 = W_g^-$. 
\end{observation}

\begin{proof} We provide a proof in Appendix~\ref{sec: proofs for section SVD background}. \end{proof}

Recall that an element $g \in \SL(d,\Kb)$ is called  \emph{weakly unipotent} if 
\begin{align*}
\lambda_1(g) =  \dots = \lambda_d(g)=1.
\end{align*}
We also have a dynamical characterization of certain weakly unipotent elements.

\begin{observation}\label{obs:dynamics of weakly unipotent} Suppose that $g \in \SL(d,\Kb)$, $V_0^\pm \in \Gr_k(\Kb^d)$, $W_0^\pm \in \Gr_{d-k}(\Kb^d)$, and 
$$
g^{\pm n} V \to V_0^\pm
$$
uniformly on compact subsets of $\left\{ V \in \Gr_k(\Kb^d) : V \text{ transverse to } W_0^{\pm}\right\}$. Then  $g$ is weakly unipotent if and only if $V_0^+=V_0^-$. 
\end{observation}

\begin{proof} We provide a proof in Appendix~\ref{sec: proofs for section SVD background}. \end{proof}

\subsection{The symmetric space associated to the special linear group} We will consider the symmetric space $N:=\GL(d,\Kb) / \mathsf{U}(d,\Kb)$  normalized so that the distance is given by 
\begin{equation}
\label{eqn:symmetric distance in prelims}
\dist_N\left( g \mathsf{U}(d,\Kb), h \mathsf{U}(d,\Kb) \right) = \sqrt{ \sum_{j=1}^d (\log \mu_j(g^{-1} h) )^2 },
\end{equation}
see~\cite[Chap.\ II.10]{BH1999} for more details. We will also consider the symmetric space $M := \SL(d,\Kb) / \SU(d,\Kb)$ which can be viewed as a totally geodesic subspace of $N$. 

Recall that $N$ identifies with the space of inner products on $\Kb^d$ via 
$$
g \mapsto \ip{g^{-1} \cdot, g^{-1}\cdot}
$$ 
(where $\ip{\cdot,\cdot}$ is the standard Euclidean inner product). The next proposition provides an elementary description of the geodesic segment in $N$ joining two inner products and is used in the proof of Theorem~\ref{thm:main}. 

\begin{proposition}\label{prop:paths between inner products} Suppose that $Q_0$ and $Q_1$ are inner products on $\Kb^d$. Then
	\begin{enumerate}
		\item There exists a basis $v_1,\dots, v_d$ of $\Kb^d$ which is orthogonal with respect to $Q_0$ and $Q_1$.
		\item There exists a smooth path 
		$$
		t \in [0,1] \mapsto f(Q_0,Q_1)(t)
		$$
		of inner products joining $Q_0$ and $Q_1$ such that: if $v_1,\dots, v_d$ is an orthogonal basis with respect to both $Q_0$ and $Q_1$, then $v_1,\dots, v_d$ is an orthogonal basis with respect to every $f(Q_0,Q_1)(t)$ and 
		\begin{equation}
		\label{eqn:path_of_inner_products}
		f(Q_0,Q_1)(t)(v_j, v_j) = Q_0(v_j,v_j)^{1-t} Q_1(v_j, v_j)^t
		\end{equation}
		for every $1 \leq j \leq d$.
	\end{enumerate}
\end{proposition}

\begin{remark} Notice that Equation~\eqref{eqn:path_of_inner_products} implies that the inner product $f(Q_0,Q_1)(t)$ depends smoothly on $Q_0$, $Q_1$, and $t$. \end{remark}

\begin{proof} 
	(1): Pick $g_0 \in \GL(d,\Kb)$ so that $Q_0 \circ g_0$ is the standard Euclidean inner product $\ip{\cdot,\cdot}$. By the spectral theory of Hermitian matrices, there exists a basis $w_1,\dots, w_d$ of $\Kb^d$ which is orthonormal relative to $\ip{\cdot,\cdot}$ and orthogonal relative to $Q_1 \circ g_0$. Then $g_0w_1,\dots, g_0w_d$  is an orthogonal basis with respect to both $Q_0$ and $Q_1$.
	
	(2): Fix a basis $v_1,\dots, v_d$ of $\Rb^d$ which is orthonormal with respect to $Q_0$ and orthogonal with respect to $Q_1$. By relabelling, we may assume that 
	$$
	Q_1(v_1,v_1) \geq Q_1(v_2, v_2) \geq \dots \geq Q_1(v_d,v_d).
	$$
	Then define an inner product $Q_t$ by 
	$$
	Q_t\left( \sum_{j=1}^d \alpha_j v_j,  \sum_{j=1}^d \beta_j v_j \right) = \sum_{j=1}^d \alpha_j \bar{\beta}_j Q_0(v_j,v_j)^{1-t} Q_1(v_j, v_j)^t =  \sum_{j=1}^d \alpha_j \bar{\beta}_j Q_1(v_j, v_j)^t.
	$$
	Suppose $w_1,\dots, w_d$ is a basis of $\Kb^d$ which is  orthogonal with respect to $Q_0$ and $Q_1$. We claim that Equation~\eqref{eqn:path_of_inner_products} holds. By scaling we can assume that $w_1,\dots, w_d$ is orthonormal with respect to $Q_0$ and by relabelling we may assume that  
	$$
	Q_1(w_1,w_1) \geq Q_1(w_2, w_2) \geq \dots \geq Q_1(w_d,w_d).
	$$
	Then $Q_1(v_j, v_j) = Q_1(w_j,w_j)$ for all $j$. Also, 
	$$
	w_j = \sum_{k=1}^d \alpha_{j,k} v_k
	$$
	where $\sum_{k=1}^d \abs{\alpha_{j,k}}^2 = 1$ and $\alpha_{j,k} \neq 0$ implies that $Q_1(w_j,w_j) = Q_1(v_k,v_k)$. Hence 
	$$
	Q_t(w_j,w_j) = \sum_{k=1}^d \abs{\alpha_{j,k}}^2 Q_1(v_k,v_k)^{t} = \sum_{k=1}^d \abs{\alpha_{j,k}}^2 Q_1(w_j,w_j)^t = Q_1(w_j,w_j)^t. 
	$$
	So $f(Q_0,Q_1)(t) := Q_t$ satisfies part (2). 
\end{proof}

\section{Definition~\ref{defn:Pk Anosov}: remarks and variations} \label{sec:rem_var}

In this section we record some basic properties of the representations introduced in  Definition~\ref{defn:Pk Anosov} and their connections to previous relative notions of Anosov representations. Then we explain how exponential contraction on the Hom bundle is equivalent to a dominated splitting of the vector bundle. 

\subsection{Basic properties} 

The symmetry in Observation~\ref{obs:strongly_dynamics_pres_div_cartan} implies the following symmetry in the definition of relatively Anosov representations. 

\begin{observation}\label{obs: k and d-k duality} Suppose that $(\Gamma,\peripherals)$ is relatively hyperbolic and $\rho\colon \Gamma \to \SL(d,\Kb)$ is a representation. Then $\rho$ is $\Psf_k$-Anosov relative to $\peripherals$ if and only if $\rho$ is $\Psf_{d-k}$-Anosov relative to $\peripherals$. 
\end{observation} 

Observation~\ref{obs:strongly_dynamics_pres_div_cartan} also gives information about the eigenvalues of peripheral and non-peripheral infinite order elements. We say subgroup $G \leq \SL(d,\Cb)$ is \emph{weakly unipotent} if every element of $G$ is weakly unipotent. 

\begin{proposition}\label{prop:eigenvalue data in rel Anosov repn} Suppose that $(\Gamma,\peripherals)$ is relatively hyperbolic and $\rho\colon \Gamma \to \SL(d,\Kb)$ is $\Psf_k$-Anosov relative to $\peripherals$. 
\begin{enumerate}
\item If $P \in \peripherals$, then $\rho(P)$ is weakly unipotent. 
\item If $\gamma \in \Gamma$ is non-peripheral and has infinite order, then $\rho(\gamma)$ is $\mathsf{P}_k$-proximal. 
\end{enumerate}
\end{proposition} 

\begin{proof} This follows immediately from the strongly dynamics-preserving property and Observations~\ref{obs:strongly_dynamics_pres_div_cartan}, ~\ref{obs:dynamics of Pk proximal}, and ~\ref{obs:dynamics of weakly unipotent}. \end{proof} 

\subsection{Relatively asymptotically embedded in the sense of Kapovich--Leeb}
\label{sec:relAE_KL}

In \cite{KL}, Kapovich and Leeb study a number of notions that provide relative versions of Anosov representations. In this subsection we recall one of their definitions (formulated in the language of this paper) and observe that it is equivalent to Definition~\ref{defn:Pk Anosov}. Later, in Section~\ref{sec:relMorse_relunifAnosov}, we will consider another one of their definitions and relate it to the uniformly relatively Anosov representations introduced in Definition~\ref{defn:uniformly Anosov}.

A subgroup $\Gamma \leq \SL(d,\Kb)$ is \emph{$\Psf_k$-divergent} if $\lim_{n \to \infty} \frac{\mu_k}{\mu_{k+1}}(\gamma_n)= \infty$ for every escaping sequence $(\gamma_n)_{n \geq 1}$  in $\Gamma$. Notice that a subgroup is $\Psf_k$-divergent if and only if it is $\Psf_{d-k}$-divergent. 

Let $\Fc_{k,d-k}(\Kb^d)$ denote the space of partial flags of the form $F=(F^k, F^{d-k})$ where $\dim F^j = j$ (with a slight abuse of notation we have $F^k \supset F^{d-k}$ when $k > d/2$ and $F^k = F^{d-k}$ when $k=d/2$). 

A $\Psf_k$-divergent group $\Gamma \leq \SL(d,\Kb)$ has a well-defined limit set in $\Fc_{k,d-k}(\Kb^d)$ defined by
$$
\Lambda_{k,d-k}(\Gamma) := \{ F : \exists (\gamma_n)_{n \geq 1} \text{ in }  \Gamma \text{ with } \gamma_n \to \infty \text{ and } F = \lim (U_k, U_{d-k})(\gamma_n)\}.
$$
Such a group is called \emph{$\Psf_k$-transverse} if every pair of distinct elements in $\Lambda_{k,d-k}(\Gamma)$ are transverse, that is
$$
F_1^k \oplus F_2^{d-k} = \Kb^d
$$
for all distinct $F_1, F_2 \in \Lambda_{k,d-k}(\Gamma)$.

\begin{definition}\cite[Def.\ 7.1]{KL}
A discrete subgroup $\Gamma \leq \SL(d,\Kb)$ is said to be \emph{$\Psf_k$-asymptotically embedded} relative to a finite collection of subgroups $\peripherals$ if $\Gamma$ is $\Psf_k$-transverse, $(\Gamma, \peripherals)$ is relatively hyperbolic, and there is a continuous $\Gamma$-equivariant map
$$
\xi \colon \partial(\Gamma,\peripherals) \to \Gr_k(\Kb^d) \times \Gr_{d-k}(\Kb^d)
$$
which is a homeomorphism onto $\Lambda_{k,d-k}(\Gamma)$. 
\end{definition}

\begin{proposition} \label{prop:relAnosov_relae} Suppose that $\rho \colon \Gamma \to \SL(d,\Kb)$ is a representation and $\peripherals$ is a collection of subgroups of $\Gamma$. Then the following are equivalent: 
\begin{enumerate}
\item $(\Gamma, \peripherals)$ is relatively hyperbolic and $\rho$ is $\Psf_k$-Anosov relative to $\peripherals$.
\item $\rho$ has finite kernel and $\rho(\Gamma)$ is $\Psf_k$-asymptotically embedded relative to $\rho(\mathcal{P})$.
\end{enumerate}
\end{proposition} 

\begin{proof}

(1)$\implies$(2): Let $\xi \colon \partial(\Gamma,\peripherals) \to \Gr_k(\Kb^d) \times \Gr_{d-k}(\Kb^d)$ denote the Anosov boundary map. By the strongly dynamics-preserving property, $\ker \rho$ is finite and hence $(\rho(\Gamma), \rho(\peripherals))$ is relatively hyperbolic and there is a $\rho$-equivariant homeomorphism $\partial(\Gamma, \peripherals) \to \partial(\rho(\Gamma), \rho(\peripherals))$ of the Bowditch boundaries. The strongly dynamics-preserving property, see Observation~\ref{obs:strongly_dynamics_pres_div_cartan}, also implies that 
$$
\xi(\partial(\Gamma, \peripherals)) = \Lambda_{k,d-k}(\Gamma)
$$
and the transversality property implies $\xi$ is injective. So by compactness, $\xi$ is a homeomorphism. Thus $\rho(\Gamma)$ is $\Psf_k$-asymptotically embedded relative to $\rho(\mathcal{P})$.

(2)$\implies$(1): Since $\ker \rho$ is finite, $(\Gamma, \peripherals)$ is relatively hyperbolic and there is a $\rho$-equivariant homeomorphism $\partial(\Gamma, \peripherals) \to \partial(\rho(\Gamma), \rho(\peripherals))$ of the Bowditch boundaries. So by hypothesis, there exists a continuous $\rho$-equivariant map
$$
\xi \colon \partial(\Gamma,\peripherals) \to \Gr_k(\Kb^d) \times \Gr_{d-k}(\Kb^d)
$$
which is a homeomorphism onto $\Lambda_{k,d-k}(\Gamma)$. By definition $\xi$ is transverse. To verify the strongly dynamics-preserving property, fix a sequence $(\gamma_n)_{n \geq 1}$ with $\gamma_n \to x \in \partial(\Gamma, \peripherals)$ and $\gamma_n^{-1} \to y \in \partial(\Gamma, \peripherals)$. By Observation~\ref{obs:strongly_dynamics_pres_div_cartan}, we need to show that $\frac{\mu_k}{\mu_{k+1}}(\rho(\gamma_n))$ goes to infinity, $U_k(\rho(\gamma_n))$ converges to $\xi^k(x)$, and $U_{d-k}(\rho(\gamma)^{-1})$ converges to $\xi^{d-k}(y)$. 

By hypothesis, $\lim_{n \to \infty} \frac{\mu_k}{\mu_{k+1}}(\rho(\gamma_n))= \infty$ and by compactness, it suffices to consider the case where 
$$
F^+ := \lim_{n \to \infty} (U_k,U_{d-k})(\rho(\gamma_n)) \quad \text{and} \quad F^-:= \lim_{n \to \infty} (U_k,U_{d-k})(\rho(\gamma_n)^{-1})
$$
exist. Since $\xi$ is a homeomorphism onto $\Lambda_{k,d-k}(\Gamma)$, there exists $x^\prime, y^\prime \in \partial(\Gamma, \peripherals)$ such that $\xi(x^\prime) = F^+$ and $\xi(y^\prime) = F^-$. Fix $z \in \partial(\Gamma, \peripherals) \smallsetminus \{x,y,x^\prime,y^\prime\}$. Then by equivariance, transversality, and Observation~\ref{obs:strongly_dynamics_pres_div_cartan} we have 
$$
\xi(x) = \lim_{n \to \infty} \xi(\gamma_n(z)) = \lim_{n\to \infty} \rho(\gamma_n)\xi(z) = F^+
$$
and likewise $\xi(y) = F^-$. So by Observation~\ref{obs:strongly_dynamics_pres_div_cartan}
$$
\rho(\gamma_n) V \to \xi^k(x)
$$
uniformly on compact subsets of $\left\{ V \in \Gr_k(\Kb^d) : V \text{ transverse to } W_0\right\}$. Thus $\xi$ is strongly dynamics-preserving. 
\end{proof} 

\subsection{Relatively dominated representations} In this section we explain how Theorem~\ref{thm:singular value and eigenvalue estimates} implies Corollary \ref{cor:relAnosov = reldom}. 

Instead of recalling the definition of $\Psf_k$-relatively dominated representations from \cite{reldomreps}, we will use the following characterization.

\begin{definition} [{\cite[Th.\ C]{rdr2}}] \label{thm:rdr2_thmC}
Suppose that $(\Gamma,\peripherals)$ is relatively hyperbolic and $X$ is a Groves--Manning cusp space for $(\Gamma, \peripherals)$. Then a representation $\rho\colon \Gamma \to \SL(d,\Kb)$ is \emph{$\Psf_k$-dominated relative to $\peripherals$} if there exists a continuous, $\rho$-equivariant, transverse, strongly dynamics-preserving map
$$ 
\xi \colon \partial(\Gamma,\peripherals) \to \Gr_k(\Kb^d) \times \Gr_{d-k}(\Kb^d)
$$
(i.e.\ $\rho$ is relatively $\mathsf{P}_k$-Anosov in the sense of Definition~\ref{defn:Pk Anosov}) and for any $x_0 \in X$ there exist constants $\alpha > 1, \beta > 0$ such that 
\begin{equation*} 
-\beta + \frac{1}{\alpha}  \dist_X(x_0, \gamma(x_0)) \leq \log \frac{\mu_k}{\mu_{k+1}}(\rho(\gamma)) \leq \log \frac{\mu_1}{\mu_d}(\rho(\gamma)) \leq \beta + \alpha  \dist_X(x_0, \gamma(x_0))
\end{equation*}
 for all $\gamma \in \Gamma$.
\end{definition} 

\begin{remark} \cite[Th.\ C]{rdr2} assumes that $\xi$ satisfies a weaker condition than strongly dynamics-preserving, called dynamics-preserving in~\cite{rdr2}. However, by \cite[Prop.\ 6.14]{reldomreps} the boundary maps are indeed strongly dynamics-preserving. 
\end{remark}

\begin{corollary}[to Theorem~\ref{thm:singular value and eigenvalue estimates}]\label{cor:relAnosov = reldom in body}
Suppose that $(\Gamma, \peripherals)$ is relatively hyperbolic and $\rho\colon \Gamma \to \SL(d,\Kb)$ is a representation. Then the following are equivalent: 
\begin{enumerate}
\item $\rho$ is $\Psf_k$-Anosov relative to $\peripherals$,
\item $\rho$ is $\Psf_k$-dominated relative to $\peripherals$. 
\end{enumerate}
\end{corollary}

\begin{proof} By definition, (2) implies (1). For the other direction, suppose that  $\rho$ is $\Psf_k$-Anosov relative to $\peripherals$. Fix a Groves--Manning cusp space $X$ of $(\Gamma, \peripherals)$ and $x_0 \in X$. By Theorem~\ref{thm:singular value and eigenvalue estimates} there exist $\alpha_0>1, \beta_0 > 0$ such that 
\begin{equation*} 
-\beta_0 + \frac{1}{\alpha_0}\dist_X(x_0, \gamma(x_0)) \leq \log \frac{\mu_k}{\mu_{k+1}}(\rho(\gamma))
\end{equation*}
for all $\gamma \in \Gamma$. Let $p_0 := \SU(d,\Kb)$. Since the orbits $\Gamma(x_0)$ and $\rho(\Gamma)(p_0)$ are quasi-isometric, Equation~\eqref{eqn:symmetric distance in prelims} implies that there exist $\alpha_1>1, \beta_1 > 0$ such that 
\begin{equation*} 
 \log \frac{\mu_1}{\mu_d}(\rho(\gamma)) \leq \beta_1 + \alpha_1  \dist_X(x_0, \gamma(x_0))
\end{equation*}
 for all $\gamma \in \Gamma$. So $\rho$ is $\Psf_k$-dominated relative to $\peripherals$. 
\end{proof}

\subsection{Irreducible representations}

We observe, as in the classical word-hyperbolic case \cite[Prop.\ 4.10]{GW}, that if a representation is sufficiently irreducible, then the strongly dynamics-preserving property can be dropped from the definition. 

\begin{proposition}\label{prop:sufficiently irreducible}
Suppose that $(\Gamma,\peripherals)$ is relatively hyperbolic, $\rho\colon \Gamma \to \SL(d,\Kb)$ is a representation, and there exists a continuous $\rho$-equivariant transverse map
\begin{align*}
\xi \colon \partial(\Gamma,\peripherals) \to \Gr_k(\Kb^d) \times \Gr_{d-k}(\Kb^d).
\end{align*}
 If $\bigwedge^k \rho \colon \Gamma \to \SL(\bigwedge^k \Kb^d)$ is irreducible (e.g.\ $\rho$ has Zariski-dense image), then $\rho$ is $\Psf_k$-Anosov relative to $\peripherals$ with Anosov boundary map $\xi$. 
\end{proposition}

\begin{proof} The argument is standard, see for instance the proof of~\cite[Cor.\ 6.3]{CZZ2021}. \end{proof} 

A version of this result for representations into general semisimple Lie groups and with the irreducibility assumption replaced with a Zariski-density assumption may be found in \cite[Th.\ 7.5]{KL}.

\subsection{Dominated splitting and contraction/expansion on Hom bundles} In this section we observe that the exponential contraction of the flow on the Hom bundle in the definition of relatively Anosov representations can be recast in terms of a dominated splitting condition. This is well known in the word-hyperbolic case~\cite{BCLS2015, BPS} and the same arguments work in the relative case as well. 

Suppose, for the rest of this section, that $(\Gamma, \peripherals)$ is a relatively hyperbolic group, $\rho \colon \Gamma \to \SL(d,\Kb)$ is a representation, $X$ is a weak cusp space for $(\Gamma, \peripherals)$, and $\norm{\cdot}$ is a metric on the vector bundle $\wh{E}_\rho(X) \to \wh{\Gc}(X)$. 

If $V,W \subset \wh{E}_\rho(X)$ are subbundles, we can consider the bundle $\Hom(V,W) \to \wh{\Gc}(X)$ with the associated family of operator norms defined by
$$
\norm{f}_\sigma : = \max \left\{\norm{f(Y)}_{\sigma} : Y \in V|_\sigma, \ \norm{Y}_{\sigma} = 1\right\}
$$
when $f \in \Hom(V,W)|_\sigma$. In particular, given a continuous $\rho$-equivariant transverse map 
$$
\xi = (\xi^k, \xi^{d-k}) \colon \partial(\Gamma, \peripherals) \to \Gr_k(\Kb^d) \times \Gr_{d-k}(\Kb^d)
$$
let $\wh{\Theta}^k, \wh{\Xi}^{d-k} \subset \wh{E}_\rho(X)$ denote the subbundles defined in Section~\ref{sec:intro to flow}. Then consider the vector bundles
$$
\Hom\left( \wh{\Theta}^k, \wh{\Xi}^{d-k}\right), \Hom\left( \wh{\Xi}^{d-k}, \wh{\Theta}^{k}\right) \to \wh{\Gc}(X)
$$
with the operator norm. Since $\wh{\Theta}^k$ and $\wh{\Xi}^{d-k}$ are invariant under the flow $\flatflow^t$, 
$$
\homflow^t(f) := \flatflow^t \circ f \circ \flatflow^{-t}.
$$
defines a flow on both Hom bundles. 

We have the following connection between the dynamics on these bundles. 

\begin{proposition}\label{prop: Hom bundles contraction/expansions} With the notation above and $c,C > 0$ fixed, the following are equivalent: 
\begin{enumerate}
\item For all $t \geq 0$, $\sigma \in \wh{\Gc}(X)$, $Y \in \wh{\Theta}^k|_{\sigma}$, and non-zero $Z\in \wh{\Xi}^{d-k}|_{\sigma}$,
$$
\frac{\norm{\flatflow^t(Y) }_{\geodflow^t(\sigma)}}{\norm{\flatflow^t(Z) }_{\geodflow^t(\sigma)}} \leq Ce^{-ct} \frac{\norm{Y}_{\sigma} }{\norm{Z}_{\sigma} }. 
$$
\item For all $t \geq 0$, $\sigma \in \wh{\Gc}(X)$, and $f \in \Hom\left( \wh{\Xi}^{d-k}, \wh{\Theta}^{k}\right)|_{\sigma}$, 
$$
\norm{\homflow^t(f)}_{\geodflow^t(\sigma)} \leq C e^{-ct} \norm{f}_\sigma.
$$
\item For all $t \geq 0$, $\sigma \in \wh{\Gc}(X)$, and $f \in \Hom\left( \wh{\Theta}^k, \wh{\Xi}^{d-k}\right)|_{\sigma}$, 
$$
\norm{\homflow^t(f)}_{\geodflow^t(\sigma)} \geq \frac{1}{C} e^{ct} \norm{f}_\sigma.
$$
\end{enumerate}
\end{proposition} 

\begin{proof} 
One can argue exactly as in Proposition 2.3 in~\cite{BCLS2015}.
\end{proof} 

\section{Singular value growth of type-preserving representations}\label{sec:char of type-preserving representations}  

In this section we use singular values to characterize the representations of a relatively hyperbolic group  that are type-preserving in the sense that they map peripheral subgroups to weakly unipotent subgroups. 

\begin{proposition} \label{prop:type-preserving<=>good upper bound on singular values}
Suppose that $(\Gamma, \peripherals)$ is relatively hyperbolic,  $X := \Cc_{GM}(\Gamma, \peripherals, S)$ is a Groves--Manning cusp space, $x_0 \in X$, and $\rho \colon \Gamma \to \SL(d,\Kb)$ is a representation. Then the following are equivalent: 
\begin{enumerate}
\item $\rho(P)$ is weakly unipotent for every $P \in \peripherals$.
\item There exist $\alpha, \beta > 0$ such that
$$
\log  \frac{\mu_1}{\mu_d}(\rho(\gamma)) \leq \alpha \dist_X(\gamma(x_0), x_0) +\beta
$$
for all $\gamma \in \Gamma$. 
\end{enumerate}
\end{proposition} 

One direction is straightforward.

\begin{lemma} With the notation in Proposition~\ref{prop:type-preserving<=>good upper bound on singular values}, $(2) \implies (1)$. \end{lemma}

\begin{proof}  Fix $P \in \peripherals$. By Proposition~\ref{prop:cusp_space_est}  there exists $\beta_1 > 0$ such that: if $g \in P$, then 
\begin{equation*}
-\beta_1+ 2 \log_2 \abs{g}_{S \cap P}  \leq \dist_X(g(x_0), x_0) \leq \beta_1+ 2 \log_2 \abs{g}_{S \cap P}. 
\end{equation*}
Now fix $g \in P$. Then 
\begin{align*}
\log \frac{\lambda_1}{\lambda_d}(\rho(g)) &= \lim_{n \to \infty} \frac{1}{n} \log \frac{\mu_1}{\mu_d}(\rho(g^n)) \leq \liminf_{n \to \infty} \frac{1}{n} \left( \alpha \dist_X(g^n(x_0), x_0) +\beta\right) \\
& \leq \liminf_{n \to \infty} \frac{2 \alpha}{n} \log_2 \abs{g^n}_{S \cap P} \leq \liminf_{n \to \infty} \frac{2 \alpha}{n} \log_2 \left( n \abs{g}_{S \cap P} \right) = 0. 
\end{align*}
So $\lambda_1(\rho(g)) = \lambda_d(\rho(g))$ which implies that $g$ is weakly unipotent. Since $P \in \peripherals$ and $g \in P$ were arbitrary, this completes the proof. 
\end{proof} 

The other direction is more involved and we start with some general lemmas about weakly unipotent subgroups. 

\begin{lemma}\label{lem:structure of weakly unipotent} If $H \leq \GL(d,\Rb)$ is weakly unipotent and $\mathsf{G} := \overline{H}^{Zar}$ is the Zariski closure of $H$, then $\mathsf{G} = \mathsf{L} \ltimes \mathsf{U}$ where $\mathsf{L}$ is compact and $\mathsf{U}$ is the unipotent radical of $\mathsf{G}$.  
\end{lemma}

This lemma is well known (see~\cite[Th.\ 5.12]{KL} or~\cite[Prop.\ 5.5]{AZimmer_rigid_Ccvx}) and follows easily from a result of Prasad, but since the proof is short we include it. 

\begin{proof} Choose a Levi decomposition $\mathsf{G} = \mathsf{L} \ltimes \mathsf{U}$, where $\mathsf{U}$ is the unipotent radical of $\mathsf{G}$ and let $\tau \colon \mathsf{G} \to \mathsf{L}$ denote the projection. Notice that $\tau(H)$ is Zariski-dense in $\mathsf{L}$ and 
$$
\lambda_j(g) = \lambda_j(\tau(g))
$$
for all $g \in \mathsf{G}$ and $1 \leq j \leq d$. So $\mathsf{L}$ has a Zariski-dense weakly unipotent subgroup. Since $\mathsf{L}$ is reductive, then $\mathsf{L}$ is compact by a result of Prasad~\cite{P1994}. 
\end{proof} 

\begin{lemma}\label{lem: growth estimates in weakly unipotent groups} Suppose that $\mathsf{U} \leq \SL(d,\Rb)$ is unipotent and $\mathsf{L} \leq \SL(d,\Rb)$ is compact and normalizes $\mathsf{U}$. For any $g_1, \dots, g_m \in \mathsf{L} \ltimes \mathsf{U}$ there exists $C > 0$ such that: if $N \geq 1$ and $i_1, \dots, i_N \in \{1,\dots, m\}$, then 
\begin{align*}
\frac{\mu_1}{\mu_d}\left(  g_{i_1} \cdots g_{i_N}\right) \leq C N^{2(d-1)}.
\end{align*}
 \end{lemma} 
 
 \begin{proof} By conjugating we may assume that $\mathsf{L} \leq \SU(d,\Rb)$. Let 
 $$
 R := 1+\max\left\{ \mu_1(g_1), \dots, \mu_1(g_m), \mu_1(g_1^{-1}), \dots, \mu_1(g_m^{-1}) \right\}.
  $$
 By definition we can decompose $g_i = \ell_i u_i$ where $\ell_i \in \mathsf{L}$ and $u_i \in \mathsf{U}$. Then 
 $$
g_{i_1} \cdots g_{i_N}= \left( \hat{u}_1 \cdots \hat{u}_N \right) (\ell_{i_1} \cdots \ell_{i_N}) 
 $$
 where $\hat{u}_j := (\ell_{i_1} \cdots \ell_{i_j}) u_{i_j}  (\ell_{i_1} \cdots \ell_{i_j})^{-1}$. Notice that $\hat{u}_j \in \mathsf{U}$ since $\mathsf{L}$ normalizes $\mathsf{U}$. Next let $T_j := \hat{u}_j - \id$. Then 
 $$
 \mu_1(T_j) \leq 1 + \mu_1(\hat{u}_j)=1 + \mu_1(g_{i_j}) \leq R
 $$
 and, since $\mathsf{U}$ is unipotent, the product of any $d$ elements in $\{ T_1, \dots, T_N\}$ is zero. So 
 \begin{align*}
\hat{u}_1 \cdots \hat{u}_N  = \left(\id + T_1 \right)\cdots \left(\id + T_N \right)= \id + \sum_{k=1}^{d-1} \sum_{1 \leq \alpha_1 < \dots < \alpha_k \leq N} T_{\alpha_1} \cdots T_{\alpha_k}
\end{align*}
and thus 
\begin{align*}
\mu_1 \left(  g_{i_1} \cdots g_{i_N} \right) = \mu_1 \left(   \hat{u}_1 \cdots \hat{u}_N \right) \leq 1 + \sum_{k=1}^{d-1} \binom{N}{k} R^k  \leq \left(  \sum_{k=0}^{d-1} R^k \right) N^{d-1}.
\end{align*}

Since $\frac{1}{\mu_d}(g_i) = \mu_1(g_i^{-1})$, the same argument implies that 
$$
\frac{1}{\mu_d}\left(  g_{i_1} \cdots g_{i_N} \right)  = \mu_1 \left(  g_{i_{N}}^{-1} \cdots g_{i_1}^{-1} \right)\leq  \left(  \sum_{k=0}^{d-1} R^k \right) N^{d-1}.
$$
So $C :=  \left(  \sum_{k=0}^{d-1} R^k \right) ^2$ suffices. 
 
 \end{proof}

Now we are ready to prove that $(1) \implies (2)$ in Proposition~\ref{prop:type-preserving<=>good upper bound on singular values}. 

\begin{lemma}\label{lem:GM upper bound on mu1/mud} 
With the notation in Proposition~\ref{prop:type-preserving<=>good upper bound on singular values},
if $\rho\colon \Gamma \to \SL(d,\Kb)$ is a representation where $\rho(P)$ is weakly unipotent for every $P \in \peripherals$, then there exists a constant $C\geq 1$ such that 
\begin{equation*}
\log \frac{\mu_1}{\mu_d}\left( \rho(\gamma) \right) \leq C \dist_X(\gamma(x_0), x_0) + 2 \dist_X(x_0,\id) 
\end{equation*}
for all $x_0 \in X$ and $\gamma \in \Gamma$.
\end{lemma} 

\begin{proof} 
Using the inclusion $\SL(d,\Cb) \leq \SL(2d,\Rb)$, we may assume that $\Kb=\Rb$. 
 
 By Proposition~\ref{prop:cusp_space_est}  there exists $\beta > 0$ such that: if $g \in P$, then 
\begin{equation*}
-\beta + 2 \log_2\abs{g}_{S \cap P}  \leq \dist_X(g(x_0), x_0) \leq \beta + 2 \log_2 \abs{g}_{S \cap P}.
\end{equation*}
By Lemmas~\ref{lem:structure of weakly unipotent} and~\ref{lem: growth estimates in weakly unipotent groups} for each $P \in \peripherals$ there exists $\hat{\beta}_P > 0$ such that: if $g \in P$, then 
\begin{equation*}
\log \frac{\mu_1}{\mu_d}\left( \rho(g) \right) \leq 2(d-1) \log \abs{g}_{S \cap P} + \hat{\beta}_P. 
\end{equation*}
Finally let 
\begin{equation}
\label{eqn:constant in lem:GM upper bound on mu1/mud}
C := \max \left\{ \max_{s \in S} \log \frac{\mu_1}{\mu_d}(\rho(s)), (d-1)(1+\beta)(\log 2) +  \max_{P \in \peripherals} \,\hat{\beta}_P \right\}.
\end{equation}

Fix $\gamma \in \Gamma$. Let $T := \dist_X(\gamma, \id)$ and let $\sigma \colon [0,T] \to X$ be a geodesic in $X$ joining $\id$ to $\gamma$. Notice that $\sigma(0), \sigma(1), \dots, \sigma(T)$ are vertices of $X$. Then let 
$$
\{ 1=t_1 < t_2 < \dots < t_m=T\} = \{ j : \sigma(j) \in \Gamma\}.
$$ 
Then $s_j := \sigma(t_j)^{-1} \sigma(t_{j+1})$ is an element of $\Gamma$ and by construction 
$$
s_j \in S \cup \bigcup_{P \in \peripherals} P.
$$
If $s_j \in S$, then 
$$
\log \frac{\mu_1}{\mu_d}(\rho(s_j)) \leq C = C \dist_X( \sigma(t_j), \sigma(t_{j+1})) 
$$
and if $s_j \in P$, then 
\begin{align*}
\log \frac{\mu_1}{\mu_d}(\rho(s_j)) & \leq 2(d-1) \log \abs{s_j}_{S \cap P} + \hat{\beta}_P \\
& \leq (d-1) (\log 2) \dist_X(s_j, \id)+(d-1)\beta(\log 2) + \hat{\beta}_P \\
& \leq \Big((d-1)(1+\beta)(\log 2) + \hat{\beta}_P\Big) \dist_X( s_j, \id) \leq C \dist_X( \sigma(t_j), \sigma(t_{j+1})). 
\end{align*} 
So 
\begin{align*}
\log \frac{\mu_1}{\mu_d}\left( \rho(\gamma) \right) & = \log \frac{\mu_1}{\mu_d}\left(  \rho(s_1) \dots \rho(s_m) \right) \leq \sum_{j=1}^m \log \frac{\mu_1}{\mu_d}(\rho(s_j)) \\
& \leq C \sum_{j=1}^m \dist_X( \sigma(t_j), \sigma(t_{j+1})) = C \dist_X(\gamma,\id) \\
& \leq C \dist_X(\gamma(x_0), x_0) + 2\dist_X(x_0, \id).  
\end{align*}

\end{proof}

\section{Consequences of a contracting flow}  \label{sec:contraction_Anosov}

In this section we establish some consequences of having a contracting flow on the Hom bundle associated to a representation with a transverse boundary map. These results show that (2) $\implies$ (1) in Theorem \ref{thm:main} and will also be used in Section~\ref{sec:proof of singular value and eigenvalue estimates} to complete the proof of Theorem~\ref{thm:singular value and eigenvalue estimates}. 

\begin{theorem}\label{thm:2 implies 1 main theorem}
If $(\Gamma, \peripherals)$ is relatively hyperbolic, $X$ is a weak cusp space for $(\Gamma, \peripherals)$, $x_0 \in X$, and $\rho \colon \Gamma \to \SL(d,\Kb)$ is $\Psf_k$-Anosov relative to $X$, then: 
\begin{enumerate}
\item $\rho$ is $\Psf_k$-Anosov relative to $\peripherals$.
\item There exist $\alpha, \beta > 0$ such that: if $\gamma \in \Gamma$, then 
$$
-\beta+\alpha \dist_X(\gamma(x_0), x_0) \leq \log \frac{\mu_k}{\mu_{k+1}}(\rho(\gamma)) \quad \text{and} \quad \alpha \ell_X(\gamma) \leq \log \frac{\lambda_k}{\lambda_{k+1}}(\rho(\gamma)).
$$
\item If $X$ is a Groves--Manning cusp space for $(\Gamma, \peripherals)$, then for any $p_0$ in the symmetric space $\SL(d,\Kb)/ \SU(d,\Kb)$ the orbits $\Gamma(x_0)$ and $\rho(\Gamma)(p_0)$ are quasi-isometric.
\end{enumerate} 
\end{theorem} 

The rest of the section is devoted to the proof of Theorem~\ref{thm:2 implies 1 main theorem}. So fix $\Gamma$, $\peripherals$, $X$, and $\rho$ as in the statement of the theorem. Then there exists a continuous $\rho$-equivariant transverse map 
$$
\xi = (\xi^k, \xi^{d-k}) \colon \partial(\Gamma, \peripherals) \to \Gr_k(\Kb^d) \times \Gr_{d-k}(\Kb^d).
$$
By hypothesis and Proposition~\ref{prop: Hom bundles contraction/expansions}, there exists a family of norms $\norm{\cdot}$ on the fibers of $\Gc(X) \times \Kb^d \to \Gc(X)$ such that:
\begin{itemize}
\item Each $\norm{\cdot}_\sigma$ is induced by an inner product on $\Kb^d$. 
\item $\norm{\rho(\gamma)(\cdot)}_{\gamma \sigma} = \norm{\cdot}_\sigma$ for all $\gamma \in \Gamma$ and $\sigma \in \Gc(X)$.
\item There are $c, C > 0$ such that
\begin{equation}
\label{eqn:dom splitting in consequences section} 
\frac{\norm{Y }_{\geodflow^t(\sigma)}}{\norm{Z }_{\geodflow^t(\sigma)}} \leq Ce^{-ct} \frac{\norm{Y}_{\sigma} }{\norm{Z}_{\sigma} }
\end{equation}
for all $t \geq 0$, $\sigma \in \Gc(X)$, $Y \in \xi^k(\sigma^+)$, and non-zero $Z\in \xi^{d-k}(\sigma^-)$.
\end{itemize} 

Since each norm is induced by an inner product, for every $\sigma \in \Gc(X)$  there exists a matrix $A_\sigma \in \GL(d,\Kb)$ such that 
$$
\norm{\cdot}_2 = \norm{A_\sigma (\cdot)}_{\sigma}. 
$$

\begin{lemma}\label{lem:decay of ratio of singular values} If $\sigma \in \Gc(X)$ and $t \geq 0$, then
$$
\frac{\mu_{k+1}}{\mu_{k}}\left(A_\sigma^{-1} A_{\geodflow^{t}(\sigma)}\right) \leq Ce^{-c t}.
$$
\end{lemma}

\begin{proof} Fix $\sigma \in \Gc(X)$ and $t \geq 0$. By Equation~\eqref{eqn:dom splitting in consequences section} 
 \begin{align*}
    \max_{Y \in \xi^k(\sigma^+) \smallsetminus \{0\}}  \frac{\norm{Y}_{\geodflow^t(\sigma)}}{\norm{Y}_{\sigma}} \leq Ce^{-ct}  \min_{Z \in \xi^{d-k}(\sigma^-) \smallsetminus \{0\}}  \frac{\norm{Z}_{\geodflow^t(\sigma)}}{\norm{Z}_{\sigma}}.
    \end{align*}
Hence
 \begin{align*}
    \max_{Y \in A_{\sigma}^{-1}\xi^k(\sigma^+) \smallsetminus \{0\}} & \frac{\norm{A_{\geodflow^t(\sigma)}^{-1} A_{\sigma} Y}_2}{\norm{Y}_2} \leq Ce^{-ct}  \min_{Z \in A_{\sigma}^{-1}\xi^{d-k}(\sigma^-) \smallsetminus \{0\}} & \frac{\norm{A_{\geodflow^{t}(\sigma)}^{-1} A_{\sigma}Z}_{2}}{\norm{Z}_{2}}.
    \end{align*}
    So by the max-min/min-max theorem for singular values 
    $$
    \mu_{d-k+1}\left(A_{\geodflow^t(\sigma)}^{-1} A_{\sigma}\right) \leq  C e^{-c t} \mu_{d-k}\left(A_{\geodflow^t(\sigma)}^{-1} A_{\sigma}\right)
$$
or equivalently 
    $$
    \mu_{k+1}\left(A_\sigma^{-1} A_{\geodflow^{t}(\sigma)}\right) \leq C e^{-c t}  \mu_{k}\left(A_\sigma^{-1} A_{\geodflow^{t}(\sigma)}\right)
$$
which establishes the lemma.
\end{proof}

\begin{lemma}\label{lem:convergence of uK}
$$
\lim_{t \to \infty} \sup_{\sigma \in \Gc(X)} \dist_{\Gr_k(\Kb^d)}\left( U_k\left(A_\sigma^{-1}A_{\geodflow^t(\sigma)}\right), A_\sigma^{-1}\xi^k(\sigma^+) \right) = 0. 
$$
\end{lemma} 

\begin{proof} Suppose not. Then there exist $t_n \to \infty$ and a sequence $(\sigma_n)_{n \geq 1}$ in $\Gc(X)$ such that
$$
\liminf_{n \to \infty} \dist_{\Gr_k(\Kb^d)} \left( U_k\left(A_{\sigma_n}^{-1}A_{\geodflow^{t_n}(\sigma_n)}\right), A_{\sigma_n}^{-1}\xi^k(\sigma_n^+) \right) >0.  
$$
Passing to a subsequence, we can suppose that $U_k\left(A_{\sigma_n}^{-1}A_{\geodflow^{t_n}(\sigma_n)}\right) \to V$ and $A_{\sigma_n}^{-1}\xi^k(\sigma_n^+) \to W$ where $V \neq W$. Fix some $Y \in W \smallsetminus V$. We can find $Y_n \in A_{\sigma_n}^{-1}\xi^k(\sigma_n^+)$ such that $Y_n \to Y$. Then 
$$
\liminf_{n \to \infty} \dist_{\proj(\Kb^d)}\left(U_k\left(A_{\sigma_n}^{-1}A_{\geodflow^{t_n}(\sigma_n)}\right), Y_n \right) > 0
$$
(where the distance denotes the minimum of $\dist_{\proj(\Kb^d)}(Z,Y_n)$ over all $Z$ representing lines in the $k$-plane $U_k\left(A_{\sigma_n}^{-1}A_{\geodflow^{t_n}(\sigma_n)}\right)$) and so 
$$
\norm{A_{\geodflow^{t_n}(\sigma_n)}^{-1} A_{\sigma_n} Y_n}_2 \gtrsim \frac{1}{\mu_{k+1}\left(A_{\sigma_n}^{-1}A_{\geodflow^{t_n}(\sigma_n)}\right)} \norm{Y_n}_2 = \mu_{d-k}\left(A_{\geodflow^{t_n}(\sigma_n)}^{-1}A_{\sigma_n}\right)\norm{Y_n}_2.
$$
On the other hand, by the max-min/min-max theorem for singular values, there exists $Z_n \in A_{\sigma_n}^{-1}\xi^{d-k}(\sigma_n^-) \smallsetminus\{0\}$ such that 
$$
\norm{A_{\sigma_n(t_n)}^{-1} A_{\sigma_n}Z_n}_2 \leq \mu_{d-k}\left(A_{\geodflow^{t_n}(\sigma_n)}^{-1}A_{\sigma_n}\right)\norm{Z_n}_2.
$$
Let $\hat{Y}_n := A_{\sigma_n} Y_n$ and $\hat{Z}_n := A_{\sigma_n} Z_n$. Then
$$
\frac{\norm{\hat{Y}_n }_{\geodflow^{t_n}(\sigma_n)}}{\norm{\hat{Z}_n }_{\geodflow^{t_n}(\sigma_n)}}  =  \frac{ \norm{A_{\geodflow^{t_n}(\sigma_n)}^{-1} A_{\sigma_n} Y_n}_2}{\norm{A_{\geodflow^{t_n}(\sigma_n)}^{-1} A_{\sigma_n} Z_n}_2} \gtrsim \frac{\norm{Y_n}_2}{\norm{Z_n}_2} = \frac{\norm{\hat{Y}_n}_{\sigma_n}}{\norm{\hat{Z}_n}_{\sigma_n}}
$$
which contradicts Equation~\eqref{eqn:dom splitting in consequences section}. 
\end{proof}

Fix $x_0 \in X$. By Proposition~\ref{prop:abundance of geodesic lines} there is some $R > 0$ such that: if $\gamma \in \Gamma$, then there exist $\sigma_\gamma \in \Gc(X)$ and $T_\gamma \geq 0$ such that 
$$
\max\left\{ \dist_X\big(x_0, \sigma_\gamma(0)\big), \dist_X\big(\gamma(x_0), \sigma_\gamma(T_\gamma)\big) \right\} \leq R. 
$$
Let 
$$
K : = \{ \sigma \in \Gc(X) : \dist_X(x_0,\sigma(0)) \leq R\}.
$$
By continuity, there exists $C_{K}>1$ so that if $\sigma \in K$, then $\norm{\cdot}_\sigma$ is $C_{K}$-bilipschitz to the standard Euclidean norm $\norm{\cdot}_2$ on $\mathbb K^d$.

\begin{lemma} \label{lem:D-}
If $\gamma \in \Gamma$, then 
$$
\log \frac{\mu_{k}}{\mu_{k+1}}(\rho(\gamma))\geq - \log \left( C C_K^4e^{2cR} \right) + c \dist_X(x_0, \gamma(x_0))
$$
and 
$$
\log \frac{\lambda_{k}}{\lambda_{k+1}}(\rho(\gamma))\geq  c \ell_X(\gamma). 
$$
\end{lemma} 

\begin{proof} Fix $\gamma \in \Gamma$. Then let $\sigma := \sigma_\gamma$ and $T := T_\gamma$. Notice that 
$$
\dist_X(x_0, \gamma(x_0)) \leq T+2R
$$
and $\sigma, \gamma^{-1} \geodflow^{T}(\sigma) \in K$. So 
\begin{align*}
\norm{\rho(\gamma)^{-1}A_{\geodflow^{T}(\sigma)}( \cdot)}_2 & \leq C_K \norm{\rho(\gamma)^{-1}A_{\geodflow^{T}(\sigma)}( \cdot)}_{\gamma^{-1} \geodflow^{T}(\sigma)} =  C_K \norm{A_{\geodflow^{T}(\sigma)}( \cdot)}_{\geodflow^{T}(\sigma)} =C_K  \norm{\cdot}_2
\end{align*}
and likewise 
$$
\norm{\rho(\gamma)^{-1}A_{\geodflow^{T}(\sigma)}( \cdot)}_2  \geq \frac{1}{C_K} \norm{\cdot}_2. 
$$
Thus 
$$
\frac{1}{C_K}\norm{A_{\geodflow^{T}(\sigma)}^{-1}(\cdot)}_2  \leq \norm{\rho(\gamma)^{-1}( \cdot)}_2\leq C_K \norm{A_{\geodflow^{T}(\sigma)}^{-1}(\cdot)}_2 
$$
which implies that 
$$
\frac{1}{C_K} \mu_j\left( A_{\geodflow^{T}(\sigma)}\right) \leq \mu_j\left( \rho(\gamma) \right) \leq C_K \mu_j\left( A_{\geodflow^{T}(\sigma)}\right) \quad \text{for} \quad j=1,\dots, d.
$$
Similar reasoning shows that 
$$
\frac{1}{C_K}  \leq \mu_j\left( A_{\sigma} \right) \leq C_K  \quad \text{for} \quad j=1,\dots, d.
$$

So by Lemma~\ref{lem:decay of ratio of singular values}
\begin{align*}
\frac{\mu_{k}}{\mu_{k+1}}(\rho(\gamma))&\geq \frac{1}{C_K^2} \frac{\mu_{k}}{\mu_{k+1}}\left(A_{\geodflow^{T}(\sigma)}\right) \geq \frac{1}{C_K^2}\frac{\mu_{d}}{\mu_{1}}\left(A_\sigma\right) \frac{\mu_{k}}{\mu_{k+1}}\left(A_\sigma^{-1}A_{\geodflow^{T}(\sigma)}\right)\\
& \geq \frac{1}{C_K^4C} e^{c T} \geq \frac{1}{C_K^4Ce^{2cR}} e^{c \dist_{X}(x_0,\gamma(x_0))}.
\end{align*}
This proves the first assertion. For the second note that 
\begin{equation*}
\log \frac{\lambda_{k}}{\lambda_{k+1}}(\rho(\gamma)) = \lim_{n \to \infty} \frac{1}{n}\log \frac{\mu_{k}}{\mu_{k+1}}(\rho(\gamma)^n) \geq \lim_{n \to \infty} \frac{c}{n} \dist_X(x_0, \gamma(x_0)) = c \ell_X(\gamma). \qedhere
\end{equation*}

\end{proof}

\begin{lemma} $\xi$ is strongly dynamics-preserving, i.e.\ $\rho \colon \Gamma \to \SL(d,\Kb)$ is $\Psf_k$-Anosov relative to $\peripherals$. \end{lemma} 

\begin{proof} Fix an escaping  sequence $(\gamma_n)_{n \geq 1}$ in  $\Gamma$ with $\gamma_n \to x \in \partial_\infty X$ and $\gamma_n^{-1} \to y \in \partial_\infty X$.  Lemma~\ref{lem:D-} implies that $\frac{\mu_k}{\mu_{k+1}}(\rho(\gamma_n)) \to \infty$. So by Observation~\ref{obs:strongly_dynamics_pres_div_cartan}  it suffices to show that $U_k(\rho(\gamma_n))$ converges to $\xi^k(x)$ and $U_{d-k}(\rho(\gamma_n)^{-1})$ converges to $\xi^{d-k}(y)$.

Let $\sigma_n := \sigma_{\gamma_n}$ and $T_n := T_{\gamma_n}$. Then $T_n \to \infty$ and $\sigma_n^+ \to x$. Arguing as in the proof of the last lemma, if 
$$
g_n : =A^{-1}_{\geodflow^{T_n}(\sigma_n)} \rho(\gamma_n),
$$
then $\{ g_n : n \in \Nb\} \subset \GL(d,\Kb)$ is relatively compact.  

Then by Lemma~\ref{lem: U_k under product} (twice) and Lemma~\ref{lem:convergence of uK}
\begin{align*}
\lim_{n \to \infty} U_k(\rho(\gamma_n)) & = \lim_{n \to \infty} U_k(A_{\geodflow^{T_n}(\sigma_n)}g_n)=   \lim_{n \to \infty} U_k(A_{\geodflow^{T_n}(\sigma_n)}) \\
& =  \lim_{n \to \infty} A_{\sigma_n}U_k(A_{\sigma_n}^{-1} A_{\geodflow^{T_n}(\sigma_n)}) = \lim_{n \to \infty} \xi^k( \sigma_n^+) = \xi^k(x). 
\end{align*} 
Applying the same argument to $\rho(\gamma_n^{-1})$ we have 
\begin{align*}
\lim_{n \to \infty} U_{d-k}(\rho(\gamma_n)^{-1}) = \xi^{d-k}(y)
\end{align*} 
which completes the proof of the lemma.
\end{proof} 

\begin{lemma} If $X$ is a Groves--Manning cusp space for $(\Gamma, \peripherals)$, then for any $p_0$ in the symmetric space $\SL(d,\Kb)/ \SU(d,\Kb)$ the orbits $\Gamma(x_0)$ and $\rho(\Gamma)(p_0)$ are quasi-isometric. \end{lemma} 

\begin{proof} It suffices to consider the case $p_0 = \SU(d,\Kb)$. Equation~\eqref{eqn:symmetric distance in prelims} implies that
$$
\dist_M( \rho(\gamma)(p_0), p_0) \asymp \log \frac{\mu_1}{\mu_d}(\rho(\gamma))
$$
for all $\gamma \in \Gamma$. By Propositions~\ref{prop:eigenvalue data in rel Anosov repn} and~\ref{prop:type-preserving<=>good upper bound on singular values} there exist $\alpha, \beta > 0$ such that 
$$
 \log \frac{\mu_1}{\mu_d}(\rho(\gamma)) \leq \beta + \alpha \dist_X(\gamma(x_0), x_0)
 $$
 for all $\gamma \in \Gamma$. Using Lemma~\ref{lem:D-} and possibly increasing $\alpha,\beta$ we may also assume that 
 $$
 \log \frac{\mu_1}{\mu_d}(\rho(\gamma)) \geq  \log \frac{\mu_k}{\mu_{k+1}}(\rho(\gamma)) \geq-\beta + \frac{1}{\alpha} \dist_X(\gamma(x_0), x_0)
 $$
 for all $\gamma \in \Gamma$. Thus the orbits are quasi-isometric. 
\end{proof}

\section{Growth rates for positive proper rational functions}\label{sec: growth rates for rational functions}

In this section we prove a quantitative lower bound on any positive proper rational function. This will be used in the next section to prove part (4) in Theorem~\ref{thm:structure of weakly unipotent discrete groups}. 

\begin{theorem}\label{thm:growth rate of positive proper rational functions} If $R \colon \Rb^d \to \Rb$ is rational, positive, everywhere defined, and 
$$
\lim_{x \to \infty} R(x) = \infty,
$$
then there exist $C, \delta >0$ such that $R(x) \geq C\norm{x}_2^{\delta}$ for all $x \in \Rb^d$. 
\end{theorem}

We will deduce the result from the following lemma. 

\begin{lemma}\label{lem:estimate on 1/R} Suppose that $f \colon \Rb^d \to \Rb$ is rational, extends to a continuous function $\hat{f}\colon \Rb^d \to \Rb$, and $\hat{f} \equiv 0$ on the set $\{x_1 = 0\}$. Then for any compact subset $K \subset \Rb^d$ there exist $C, \delta > 0$ such that: if $x \in K$, then 
$$
\abs{\hat{f}(x)} \leq C \abs{x_1}^{\delta}.
$$
\end{lemma}

Delaying the proof of the lemma we prove the theorem.

\begin{proof}[Proof of Theorem~\ref{thm:growth rate of positive proper rational functions}] 

We identify $\Rb^d$ with  the affine chart 
$$
\{ [1:x_1:\dots:x_d] : x_1,\dots, x_d \in \Rb\}
$$ 
in $\proj(\Rb^{d+1})$. Then $1/R$ extends to a continuous function $f \colon \proj(\Rb^{d+1}) \to \Rb$ where $f \equiv 0$ on $\proj(\Rb^{d+1}) \smallsetminus \Rb^d$. 

For $j=1, \dots, d$ let $\phi_j \colon \Rb^d \to \proj(\Rb^{d+1})$ be the map 
$$
\phi_j(y_1,\dots, y_d) = \left[ y_1 : \dots : y_{j} : 1 : y_{j+1} : \dots : y_d\right].
$$
Then 
$$
\proj(\Rb^{d+1}) = \Rb^d \cup \bigcup_{j=1}^{d} \phi_j\left(  [-1,1]^d \right). 
$$

Each $f \circ \phi_j$ satisfies Lemma~\ref{lem:estimate on 1/R} and so there exist 
$C_0, \delta > 0$ such that 
$$
f \circ \phi_j(y) \leq C_0 \abs{y_1}^{\delta}
$$
when $y \in[-1,1]^d$. By continuity and the positivity of $R$, there exists $C_1 > 0$ such that 
$$
R(x) \geq C_1 \norm{x}_2^{\delta}
$$
when $x \in [-1,1]^d$. 

We claim that $\delta$ and $C := \max \left\{ C_1, \frac{d^{\delta/2}}{C_0} \right\}$ satisfy the theorem. If $x \in [-1,1]^d$, this follows from the definition of $C_1$. So suppose that 
$x \notin [-1,1]^d$. Fix $1 \leq j \leq d$ such that $\abs{x_j}$ is maximal. Let 
$$
y := \left(\frac{1}{x_j}, \frac{x_1}{x_j}, \dots, \frac{x_{j-1}}{x_j}, \frac{x_{j+1}}{x_j}, \dots, \frac{x_{d}}{x_j}  \right).
$$
Then $y \in [-1,1]^d$ and $x = \phi_{j}(y)$. So 
$$
R(x) = \frac{1}{f \circ \phi_{j}(y)} \geq \frac{1}{C_0\abs{y_1}^\delta}= \frac{1}{C_0} \abs{x_j}^{\delta} \geq \frac{d^{\delta/2}}{C_0} \norm{x}_2^\delta
$$
(where in the last inequality we used the maximality of $\abs{x_j}$). 
\end{proof} 

\subsection{Proof of Lemma~\ref{lem:estimate on 1/R}} To prove the lemma we need some terminology and a result from~\cite{FHMM2016}. 

Following~\cite{FHMM2016},  for $k \in \Zb_{\geq 0}$ a function $f \colon \Rb^d \to \Rb$ is called \emph{$k$-regulous} if $f$ is $\Cc^k$-smooth and coincides with a rational function on a Zariski open subset of $\Rb^d$. The set of $k$-regulous functions is denoted by $\Rc^k(\Rb^d)$, which we can either view as a subring of the rational functions $\Rb(x_1,\dots, x_d)$ on $\Rb^d$ or as a subring of the $\Cc^k$-smooth functions on $\Rb^d$. 

Recall that an ideal $J$ in the ring of polynomials $\Rb[x_1,\dots, x_d]$ is called \emph{real} if whenever $f_1^2 + \dots + f_m^2 \in J$ then $f_1,\dots, f_m \in J$. Also given an ideal $J \subset \Rb[x_1,\dots, x_d]$, let 
$$
\mathcal{Z}(J) := \{ x \in \Rb^d : f(x) = 0 \text{ for all } f \in J\}.
$$
Finally, given a subset $A \subset \Rb^d$, let 
$$
 \mathcal{I}_{\Rc^k}(A) := \{ f \in \Rc^k(\Rb^d) : f(x) = 0 \text{ for all } x \in A\}.
 $$

We will use the following version of the Nullstellensatz.

\begin{theorem}[{\cite[Th.\ 5.11]{FHMM2016}}] If $k \in \Zb_{\geq 0}$ and  $J \subset \Rb[x_1,\dots, x_d]$ is a real ideal, then 
$$
{\rm Rad}\left( \Rc^k(\Rb^d) \cdot J \right) = \mathcal{I}_{\Rc^k}( \mathcal{Z}(J) ).
$$
\end{theorem} 

Now we are ready to prove the lemma. 

\begin{proof}[Proof of Lemma~\ref{lem:estimate on 1/R}] Consider the ideal $J = (x_1)$ in $\Rb[x_1,\dots, x_d]$. Then $J$ is a real ideal and $f \in  \mathcal{I}_{\Rc^k}( \mathcal{Z}(J) )$. So there exist $N \in \Nb$ and $h \in \Rc^k(\Rb^d)$ such that $f^N = h \cdot x_1$. So if $K \subset \Rb^d$ is compact and $C := \max \left\{ \abs{h(x)}^{1/N} : x \in K\right\}$, then 
$$
\abs{\hat{f}(x)} \leq C \abs{x_1}^{1/N} 
$$
for all $x \in K$. 
\end{proof}

\section{The structure of weakly unipotent discrete groups}\label{sec: structure of weakly unipotent groups}

Recall, from Proposition~\ref{prop:eigenvalue data in rel Anosov repn}, that the image of a peripheral subgroup under a relatively Anosov representation is weakly unipotent. In this section we prove a structure theorem for weakly unipotent discrete groups which will be fundamental in the arguments that follow. 

Given a Lie group $\mathsf{G}$, we let $\mathsf{G}^0 \leq \mathsf{G}$ denote the connected component of the identity. 

\begin{theorem} \label{thm:structure of weakly unipotent discrete groups}
Suppose that $\Gamma \leq \SL(d,\Rb)$ is a weakly unipotent discrete group. 
\begin{enumerate}
\item $\Gamma$ is virtually nilpotent. 
\item $\Gamma$ is a cocompact lattice in its Zariski closure $\mathsf{G}:=\overline{\Gamma}^{Zar}$. Moreover 
\begin{enumerate} 
\item $\mathsf{G} = \mathsf{L} \ltimes \mathsf{U}$ where $\mathsf{L}$ is compact and $\mathsf{U}$ is unipotent. 
\item $\mathsf{G}^0 = \mathsf{L}^0 \times \mathsf{U}$ and $\mathsf{L}^0$ is Abelian. 
\end{enumerate}
\item If $S$ is a finite symmetric generating set  of $\Gamma$, then there exist $\alpha, \beta > 0$ such that 
$$
\log  \frac{\mu_1}{\mu_d}(\gamma) \leq \alpha \log \abs{\gamma}_S +\beta
$$
for all $\gamma \in \Gamma$. 
\item If $\Gamma$ is $\mathsf{P}_k$-divergent and $S$ is a finite symmetric generating set of $\Gamma$, then there exist $\alpha, \beta > 0$ such that 
$$
\log  \frac{\mu_k}{\mu_{k+1}}(\gamma) \geq \alpha \log \abs{\gamma}_S +\beta
$$
for all $\gamma \in \Gamma$. 
\end{enumerate}
\end{theorem} 

For the rest of this section suppose that $\Gamma \leq \SL(d,\Rb)$ is a weakly unipotent discrete group with Zariski closure $\mathsf{G}$. By Lemma~\ref{lem:structure of weakly unipotent}, $\mathsf{G} = \mathsf{L} \ltimes \mathsf{U}$, where $\mathsf{L}$ is compact and $\mathsf{U}$ is the unipotent radical of $\mathsf{G}$. Thus (2)(a) is true, and then (3) follows from Lemma~\ref{lem: growth estimates in weakly unipotent groups}.

The next lemma will be used in the proof of (1). 

\begin{lemma} There exists a flag $\{0\} \subset V_1 \subset \dots \subset V_m =\Rb^d$ such that the projection of $\mathsf{G}$ to each $\GL(V_{j+1}/V_j)$ is compact. \end{lemma} 

\begin{proof} If $\mathsf{U} = 1$, then $\mathsf{G}=\mathsf{L}$ is compact and the trivial flag $\{0\} \subset \Rb^d$ suffices. 

If $\mathsf{U}$ is non-trivial, then the subspace 
$$
W := \{ w \in \Rb^d : u(w) = w \text{ for all } u \in \mathsf{U}\}
$$
 is proper. Then, since $\mathsf{L}$ normalizes $\mathsf{U}$, $\mathsf{G}$ preserves the flag $\{0\} \subset W \subset \Rb^d$. Let $\Gamma_1, \mathsf{G}_1\subset \GL(W)$ and $\Gamma_2, \mathsf{G}_2 \subset \GL(\Rb^d / W)$ denote the projections of $\Gamma$ and $\mathsf{G}$. Then $\Gamma_j$ is weakly unipotent and Zariski-dense in $\mathsf{G}_j$ for $j=1,2$. Notice that we can apply Lemma~\ref{lem:structure of weakly unipotent} to both $\Gamma_1$ and $\Gamma_2$. So by induction on dimension, there exists a flag
$$
\{0\}=V_0  \subset \dots \subset W \subset \dots  \subset V_m=\Rb^d
$$
with the desired properties. 
\end{proof} 

\begin{lemma} $\Gamma$ is virtually nilpotent. \end{lemma} 

\begin{proof} Let $\dist_M$ denote the standard symmetric space metric on $M := \SL(d,\Rb)/\SU(d,\Rb)$ defined in Equation~\eqref{eqn:symmetric distance in prelims}. Fix a finite symmetric set $S \subset \Gamma$ such that the group $\Gamma_S$ generated by $S$ has the same Zariski closure as $\Gamma$. 

We claim that $\Gamma_S$ is virtually nilpotent. Using the Margulis lemma, see~\cite[Th.\ 9.5]{BGS1985}, it suffices to show that 
$$
\inf_{p \in M} \max_{s \in S} \dist_M( s(p), p) = 0.
$$
Let $d_j := \dim V_{j}-\dim V_{j-1}$. Using the last lemma and conjugating, we can assume that 
$$
\Gamma \leq \left\{ \begin{pmatrix} A_1 & *  & \dots & *\\ 0 & \ddots & \ddots & \vdots \\ \vdots & \ddots & \ddots & * \\ 0 & \dots & 0 & A_m \end{pmatrix} : A_j \in \mathsf{U}(d_j, \Rb)\right\}.
$$
Fix real numbers $\lambda_1 >  \dots > \lambda_m$ with $\sum_{j=1}^m \lambda_j d_j = 0$ and let
$$
a_t := \bigoplus_{j=1}^m e^{\lambda_j t} \id_{d_j} \in \SL(d,\Rb). 
$$
Then, by choosing $t$ sufficiently large, we can make
$$
\max_{s \in S} \dist_M\left( s a_t \SU(d,\Rb), a_t \SU(d,\Rb) \right) 
= \max_{s \in S} \dist_M\left( a_{-t} s a_t \SU(d,\Rb), \SU(d,\Rb) \right)
$$
arbitrarily small. So $\Gamma_S$ is virtually nilpotent. 

Then the connected component of the identity in $\overline{\Gamma}_S^{Zar}=\overline{\Gamma}^{Zar}$ is nilpotent which implies that $\Gamma$ is virtually nilpotent. 
\end{proof} 

\begin{lemma} $\Gamma$ is a cocompact lattice in $\mathsf{G}$. Moreover, $\mathsf{G}^0 = \mathsf{L}^0 \times \mathsf{U}$ and $\mathsf{L}^0$ is Abelian. 
\end{lemma}

\begin{proof} 
First notice that $\mathsf{G}^0 = \mathsf{L}^0 \ltimes \mathsf{U}$ since $\mathsf{U}$ is the unipotent radical and hence by definition is connected.

Let $\mathsf{A} \subset \mathsf{G}^0$ denote the set of semisimple elements in $\mathsf{G}^0$. By \cite[Th.\ III.10.6]{Borel_linalg_groups}, $\mathsf{A}$ is an Abelian subgroup and $\mathsf{G}^0 = \mathsf{A} \times \mathsf{U}$. Since $\mathsf{G}^0$ is weakly unipotent, $\mathsf{A}$ must be compact. Finally, since $\mathsf{L}^0$ is compact, $\mathsf{L}^0$ consists of semisimple elements and hence is a subgroup of $\mathsf{A}$. So $\mathsf{L}^0$ is Abelian and commutes with $\mathsf{U}$. 

Fix a finite-index nilpotent subgroup $\Gamma_0 \leq \Gamma$ with  $\overline{\Gamma_0}^{Zar} = \mathsf{G}^0$. Let $\Gamma_{0}^\prime$ denote the projection of $\Gamma_0$ to $\mathsf{U}$ with respect to the decomposition $\mathsf{G}^0 = \mathsf{A} \times \mathsf{U}$. Then $\Gamma_{0}^\prime$ is discrete and Zariski-dense in $\mathsf{U}$. Further $\mathsf{U}$, being unipotent and connected, is simply connected. So by a theorem of Malcev (see e.g.\ \cite[Th.\ 2.3]{Raghunathan}), $\Gamma_{0}^\prime$ is a cocompact lattice in $\mathsf{U}$. Then, since $\mathsf{A}$ is compact, $\Gamma_0 \leq \mathsf{G}^0$ is a cocompact lattice. Finally, since  $\Gamma_0 \leq \Gamma$ and $\mathsf{G}^0 \leq \mathsf{G}$ are finite-index subgroups, we see that $\Gamma$ is a cocompact lattice of $\mathsf{G}$.
\end{proof}

Finally, to prove (4) we will use Theorem~\ref{thm:growth rate of positive proper rational functions}. In the lemmas that follow let $\mathfrak{u}$ denote the Lie algebra of $\mathsf{U}$ and fix a norm $\norm{\cdot}$ on $\mathfrak{u}$. 

\begin{lemma}\label{lem: rational function approximating singular values} For any $k \in \{1,\dots, d-1\}$, there exists a (real) rational function $R \colon \mathfrak{u} \to \Rb$ such that:
\begin{enumerate} 
\item $R$ is positive and defined everywhere. 
\item There exists $C > 0$ such that: if $Y \in \mathfrak{u}$, then 
$$
\frac{1}{C} \sqrt{R(Y)} \leq \frac{\mu_k}{\mu_{k+1}}(e^Y) \leq  C \sqrt{R(Y)}.
$$
\item If $\Gamma$ is $\mathsf{P}_k$-divergent, then $\displaystyle \lim_{Y \in \mathfrak{u}, Y \to \infty} R(Y) = \infty$.
\end{enumerate}
\end{lemma} 

\begin{proof} We start by introducing some notation. For a $d$-by-$d$ real matrix $A$ let 
$$
\norm{A}_2 := \sqrt{ \sum_{i,j=1}^d \abs{A_{i,j}}^2}.
$$
Then there exists $C_d > 1$ such that 
\begin{equation}
\label{eqn:matrix L2 norm versus singular values}
\frac{1}{C_d} \norm{A}_2 \leq \mu_1(A) \leq C_d \norm{A}_2.
\end{equation} 
Also, for $2 \leq \ell \leq d$ and $g \in \GL(d,\Rb)$, let $\wedge^\ell(g) \in \GL(\bigwedge^\ell \Rb^d)$ denote the linear isomorphism defined by 
$$
\wedge^\ell(g)(v_1 \wedge \dots \wedge v_\ell) = (g v_1) \wedge \dots \wedge (gv_\ell).
$$
If $D_\ell := \dim \bigwedge^\ell \Rb^d$ and we identify $\bigwedge^\ell \Rb^d$ with $\Rb^{D_\ell}$ via the standard basis
$$
 \{ e_{i_1} \wedge \dots \wedge e_{i_\ell} : i_1 < \dots < i_\ell\},
$$
then 
\begin{equation}
\label{eqn:singular values of wedge}
\mu_1(\wedge^\ell(g)) = \mu_1(g) \cdots \mu_\ell(g)
\end{equation}
for all $g \in \GL(d,\Rb)$. 

Since $\mathsf{U}$ is unipotent, 
$$
e^Y = \id + Y + \frac{1}{2!} Y^2 + \dots + \frac{1}{(d-1)!} Y^{d-1}  
$$
for all $Y \in \mathfrak{u}$. Then Equations~\eqref{eqn:matrix L2 norm versus singular values} and~\eqref{eqn:singular values of wedge} imply that the rational function $R \colon \mathfrak{u} \to \Rb$ defined by 
$$
R(Y) = \frac{\norm{ \wedge^k e^Y}_2^4 }{\norm{ \wedge^{k+1} e^Y}_2^2 \cdot \norm{ \wedge^{k-1} e^Y}_2^2}
$$
satisfies (1) and (2). 

To prove (3), fix an escaping sequence $(Y_n)_{n \geq 1}$ in $\mathfrak{u}$. Since $\mathsf{U}$ is unipotent and connected, $\exp\colon \mathfrak{u} \to \mathsf{U}$ is a diffeomorphism and so $(e^{Y_n})_{n \geq 1}$ is an escaping sequence in $\mathsf{G}$. Since $\Gamma \leq \mathsf{G}$ is a cocompact lattice, there exists an escaping sequence $(\gamma_n)_{n \geq 1}$ in $\Gamma$ such that $\left\{ \gamma_n^{-1} e^{Y_n} : n \geq 1\right\}$ is relatively compact. Then, since $\Gamma$ is $\mathsf{P}_k$-divergent, 
\begin{equation*}
\lim_{n \to \infty} R(Y_n) \asymp \lim_{n \to \infty} \left( \frac{\mu_k}{\mu_{k+1}}\left(e^{Y_n}\right) \right)^2 \asymp \lim_{n \to \infty} \left( \frac{\mu_k}{\mu_{k+1}}(\gamma_n) \right)^2 = \infty. \qedhere
\end{equation*}
\end{proof}

\begin{lemma}\label{lem:word metric and exponential function} For any finite symmetric generating set $S \subset \Gamma$ there exist $\alpha_1,\beta_1 > 0$ such that: if $\gamma \in \Gamma$ and $\gamma = \ell e^Y$ where $\ell \in \mathsf{L}$ and $Y \in \mathfrak{u}$, then 
$$
\alpha_1 \abs{\gamma}_S - \beta_1 \leq \norm{Y}. 
$$
\end{lemma} 

\begin{proof} Fix a distance $\dist_\mathsf{G}$ on $\mathsf{G}$ generated by a $\mathsf{G}$-invariant Riemannian metric. Since $\Gamma \leq \mathsf{G}$ is a cocompact lattice, by the fundamental lemma of geometric group theory there exist $\alpha_0 > 1$, $\beta_0> 0$ such that 
$$
\frac{1}{\alpha_0} \abs{\gamma}_S -\beta_0  \leq \dist_\mathsf{G}(\gamma, \id) \leq \alpha_0 \abs{\gamma}_S +\beta_0 
$$
for all $\gamma \in \Gamma$. Also, let 
$$
R_1 := \max\{ \dist_\mathsf{G}(\ell,\id) : \ell \in \mathsf{L}\} \quad \text{and} \quad R_2 : =\max\{ \dist_\mathsf{G}(e^Y, \id) : \norm{Y} \leq 1\}.
$$

Now suppose that $\gamma = \ell e^Y \in \Gamma$ where $\ell \in \mathsf{L}$ and $Y \in \mathfrak{u}$. Let $n = \lfloor \norm{Y} \rfloor$. Then 
\begin{align*}
\frac{1}{\alpha_0} & \abs{\gamma}_S -\beta_0  \leq \dist_\mathsf{G}(\gamma, \id) \leq R_1+ \dist_\mathsf{G}(e^Y, \id) \\
&  \leq R_1 + \dist_\mathsf{G}\left(e^Y, e^{ \frac{n}{\norm{Y}}Y}\right) + \sum_{j=0}^{n-1} \dist_\mathsf{G}\left(e^{ \frac{j+1}{\norm{Y}}Y}, e^{ \frac{j}{\norm{Y}}Y}\right) \\
& \leq R_1 + R_2(n+1) \leq R_1+R_2 + R_2 \norm{Y}. 
\end{align*}
\end{proof}

\begin{lemma} 
If $\Gamma$ is $\mathsf{P}_k$-divergent and $S$ is a finite symmetric generating set of $\Gamma$, then there exist $\alpha_2, \beta_2 > 0$ such that 
$$
\log  \frac{\mu_k}{\mu_{k+1}}(\gamma) \geq \alpha_2 \log \abs{\gamma}_S +\beta_2
$$
for all $\gamma \in \Gamma$. 
\end{lemma} 

\begin{proof} By Lemma~\ref{lem: rational function approximating singular values} and Theorem~\ref{thm:growth rate of positive proper rational functions}, there exist $C_2, \epsilon > 0$ such that 
$$
\frac{\mu_k}{\mu_{k+1}}(e^Y) \geq C_2 \norm{Y}^{\epsilon}
$$
for all $Y \in \mathfrak{u}$. 

Fix $\gamma \in \Gamma$. Then $\gamma = \ell e^Y$ for some $\ell \in \mathsf{L}$ and $Y \in \mathfrak{u}$. Then by Lemma~\ref{lem:word metric and exponential function}
\begin{equation*}
\frac{\mu_k}{\mu_{k+1}}(\gamma) \gtrsim  \frac{\mu_k}{\mu_{k+1}}(e^Y)\gtrsim  \norm{Y}^{\epsilon}\gtrsim \abs{\gamma}_S^{\epsilon} - 1. \qedhere
\end{equation*}

\end{proof}

\section{Relatively Anosov implies the existence of a contracting flow} \label{sec:building_norms}

In this section we prove that (1)$\implies$(3) in Theorem \ref{thm:main}. Since the implication (3)$\implies$(2) is by definition and the implication (2)$\implies$(1) was established in Theorem~\ref{thm:2 implies 1 main theorem} this will complete the proof of Theorem~\ref{thm:main}.

This implication, when  combined with Theorem~\ref{thm:2 implies 1 main theorem}, also proves the claims in Theorem~\ref{thm:singular value and eigenvalue estimates} for a single representation. 

\begin{theorem}\label{thm: building a norm} Suppose that $(\Gamma, \peripherals)$ is relatively hyperbolic and  $\rho \colon \Gamma \to \SL(d,\Kb)$ is $\Psf_k$-Anosov relative to $\peripherals$. If $X = \Cc_{GM}(\Gamma,\peripherals,S)$ is a Groves--Manning cusped space for $(\Gamma,\peripherals)$, then $\rho$ is $\Psf_k$-Anosov relative to $X$. 
\end{theorem} 

The rest of the section is devoted to the proof of Theorem~\ref{thm: building a norm}. So fix $\Gamma$, $\peripherals$, $\rho$, and $X= \Cc_{GM}(\Gamma,\peripherals,S)$ as in the statement of the theorem. Let $\xi$ denote the Anosov boundary map. Since $X$ is fixed for the entire section, we will let 
$$\Gc:=\Gc(X) \quad \text{and} \quad  E := E(X)=\Gc(X) \times \Kb^d.$$
Also let $I \colon \Gc \to \Gc$ denote the involution 
$$
I(\sigma)(t) = \sigma(-t). 
$$

Observation \ref{obs:strongly_dynamics_pres_div_cartan} implies that $\rho(\Gamma)$ is $\Psf_k$-divergent and Proposition~\ref{prop:eigenvalue data in rel Anosov repn} implies that if $P \in \peripherals$, then $\rho(P)$ is weakly unipotent. So by Theorem \ref{thm:structure of weakly unipotent discrete groups}(4) and Proposition~\ref{prop:cusp_space_est} there exist constants $\alpha,\beta > 0$ such that: if $P \in \peripherals$ and $\gamma \in P$, then 
\begin{equation}
\label{eqn:growth on parabolics in construction of norms} 
\log \frac{\mu_k}{\mu_{k+1}}(\rho(\gamma)) \geq -\beta + \alpha \dist_X(\gamma, \id). 
\end{equation}

\subsection{Thick-thin-like decomposition} We begin the construction of the norms by dividing the flow space $\Gc$ into a ``thick'' and ``thin'' part. 

For $P \in \peripherals$, let $H_P^\prime \subset X$ denote the induced subgraph of the associated combinatorial horoball with vertex set $\{ (\gamma, n) : \gamma \in P, n \geq 2\}$, let $H_P^{\prime\prime} \subset X$ denote the induced subgraph of the associated combinatorial horoball with vertex set $\{ (\gamma, 2) : \gamma \in P\}$, and let
$H_P := H_P^\prime \smallsetminus H_P^{\prime\prime}$.  Next, for 
$$
\gamma P \gamma^{-1} \in \peripherals^\Gamma:=\{ \gamma P \gamma^{-1} : P \in \peripherals, \gamma \in \Gamma\}
$$ 
let $H_{\gamma P \gamma^{-1}} := \gamma H_P$. 

The equivariant family of sets $\{H_P\}_{P \in \peripherals^\Gamma}$ are open in $X$, have disjoint closures, and each $\partial H_{\gamma P\gamma^{-1}}$ (with $P \in \peripherals$ and $\gamma \in \Gamma$) consists of the vertex set $\gamma\{ (g, 2) : g \in P\}$.  Further, $\Gamma$ acts cocompactly on the set
$$
X \smallsetminus \bigcup_{P \in \peripherals^\Gamma} H_P. 
$$

For $P \in \peripherals^\Gamma$, let 
\begin{align*}
\Gc_P & :=\{ \sigma \in \Gc : \sigma(0) \in H_P\}, \\
\partial \Gc_P & :=\{ \sigma \in \Gc : \sigma(0) \in \partial H_P\}, \\
\partial^+\Gc_P & := \{ \sigma \in \partial \Gc_P : \sigma(t) \in H_P \text{ for } t > 0 \text{ sufficiently small}\}, \text{and} \\
\partial^-\Gc_P &:= \{ \sigma \in \partial \Gc_P : \sigma(t) \in  H_P \text{ for } t < 0 \text{ sufficiently small}\}.
\end{align*}
Notice that $\sigma \in \partial^+ \Gc_P$ if and only if $I(\sigma) \in \partial^- \Gc_P$. Also, by definition,  $\partial^+\Gc_P \cap \partial^-\Gc_P = \varnothing$.

Next, for $\sigma \in \partial^+\Gc_P$, define
$$
T_{\sigma}^+ := \min \{ t \in (0,\infty] : \sigma(t) \notin H_P \}
$$
and for $\sigma \in \partial^-\Gc_P$, define
$$
T_{\sigma}^- := \max \{ t \in [-\infty, 0) : \sigma(t) \notin H_P\} = -T_{I(\sigma)}^+
$$
(where $\sigma(\pm \infty) = \sigma^\pm$). Then 
$$
\Gc_P=\left( \bigcup_{\sigma \in \partial^+\Gc_P}\, \bigcup_{t \in (0,T_\sigma^+)} \geodflow^t(\sigma)\right) \cup\left( \bigcup_{\sigma \in \partial^-\Gc_P}\, \bigcup_{t \in (T_\sigma^-,0)} \geodflow^t(\sigma) \right).
$$
Finally let 
\begin{align*}
\Gc_{thin} & := \bigcup_{P \in \peripherals^\Gamma} \Gc_P, \quad  \Gc_{thick}  := \Gc \smallsetminus \Gc_{thin},\\
E_{thin}& :=\bigcup_{\sigma \in \Gc_{thin}} E|_{\sigma}, \quad \text{and} \quad  E_{ thick}  :=\bigcup_{\sigma \in \Gc_{thick}} E|_{\sigma}.
\end{align*}

\subsection{Building the norm} \label{subsec:norm_defn}

Since $\xi$ is transverse, we can define a vector bundle decomposition $E = E_1 \oplus E_2 \oplus E_3$ by setting 
$$
E_1|_\sigma = \xi^k(\sigma^+), \quad \quad E_2|_\sigma = \xi^{d-k}(\sigma^+) \cap \xi^{d-k}(\sigma^-), \quad \text{and} \quad E_3|_\sigma= \xi^{k}(\sigma^-).
$$
For $\sigma \in \Gc(X)$, let $\pi^\sigma_1, \pi^\sigma_2, \pi^\sigma_3$ denote the projections induced by the decomposition 
$$
E|_\sigma = E_1|_\sigma \oplus E_2|_\sigma \oplus E_3|_{\sigma}.
$$

Fix a continuous $\rho$-equivariant family of inner products $Q_{\sigma}$ on the fibers of $E_{ thick}$ 
such that 
\begin{equation*}
Q_{\sigma} = Q_{I(\sigma)} \quad \text{for all} \quad \sigma \in \Gc_{thick}
\end{equation*}
and $E=E_1\oplus E_2 \oplus E_3$ is an orthogonal decomposition, that is
$$
Q_{\sigma} (Y,Y)= \sum_{j=1}^3 Q_\sigma(\pi^\sigma_j(Y), \pi^\sigma_j(Y))
$$
for all $\sigma \in \Gc_{thick}$ and $Y \in \Kb^d$.

Let $\alpha$ be the constant in Equation~\eqref{eqn:growth on parabolics in construction of norms}, then extend the family of inner products to $\Gc_{thin}$ as follows: 
\begin{enumerate}
\item If $\sigma = \geodflow^t(\sigma_0)$ for some $\sigma_0 \in \partial^+ \Gc_P$ and $t \in (0, T_{\sigma_0}^+)$, write $T = T_{\sigma_0}^+$ to lighten the notation, then:\begin{itemize}
\item If $t \in (0, \frac13 T]$, define
$$
Q_{\sigma} (Y,Y):= \sum_{j=1}^3 e^{\alpha(j-2)t} Q_{\sigma_0}(\pi^\sigma_j(Y), \pi^\sigma_j(Y)).
$$
\item If $t \in [\frac23 T, T)$, define
$$
Q_{\sigma} (Y,Y):= \sum_{j=1}^3 e^{\alpha(2-j)(T - t)} Q_{\phi^T(\sigma_0)}(\pi^\sigma_j(Y), \pi^\sigma_j(Y)).
$$
\item If $t \in (\frac13 T, \frac23 T)$, define 
$$ 
Q_\sigma := f\left(Q_{\geodflow^{\frac13 T }(\sigma_0)}, Q_{\geodflow^{\frac23 T }(\sigma_0)} \right) \left( \frac{3}{T} t - 1 \right)
$$
where $f$ is the path defined in Proposition~\ref{prop:paths between inner products}. 
\end{itemize}

\item If $\sigma = \geodflow^t(\sigma_0)$ for some $\sigma_0 \in \partial^- \Gc_P$ and $t \in (T_{\sigma_0}^-,0)$, define $Q_{\sigma}:= Q_{I(\sigma)}$.
\end{enumerate}
Finally, let $\norm{\cdot}_\sigma$ denote the norm induced by $Q_\sigma$.

\begin{lemma} The family of norms $\norm{\cdot}_{\sigma}$ is $\rho$-equivariant and continuous. \end{lemma} 

\begin{proof}By construction, $\norm{\cdot}_{\sigma}$ is $\rho$-equivariant. To verify that $\norm{\cdot}_{\sigma}$ is continuous, it suffices to fix $P \in \peripherals^\Gamma$ and show that 
$$
\sigma \in \overline{\Gc_P} \mapsto Q_\sigma
$$
is continuous. Suppose $\sigma_n \to \sigma$  in $\overline{\Gc_P}$. Since $X$ is a metric graph, this implies that there exist sequences $(\epsilon_n)_{n \geq 1}$ and $(S_n)_{n \geq 1}$ such that: $\epsilon_n \to 0$, $S_n \to \infty$, and 
$$
\sigma_n(t+\epsilon_n) = \sigma(t)
$$
for all $t \in [-S_n, S_n]$. Then it is straightforward to check directly from the definition that $Q_{\sigma_n} \to Q_{\sigma}$. 
\end{proof} 

\begin{remark} One naive way of extending the inner products from the thick part to the thin part is to identify the space of inner products with the symmetric space $ \GL(d,\Kb) / \mathsf{U}(d,\Kb)$, then use the symmetric space geodesics to extend to the fibers above the thin part. However, since a given peripheral subgroup may limit onto many points in the geodesic boundary of the symmetric space, this extension may fail to be well-defined or continuous at geodesics asymptotic to a bounded parabolic point. Our piecewise definition can be viewed as refinement of this naive extension.

\end{remark} 

Given $\sigma \in \Gc$ and $t \geq 0$, define 
\begin{align*}
\kappa_t(\sigma) &:= \max\left\{ \frac{\norm{Y}_{\geodflow^{t}(\sigma)}}{\norm{Z}_{\geodflow^{t}(\sigma)}} : Y \in \xi^k(\sigma^+), Z \in \xi^{d-k}(\sigma^-), \norm{Y}_{\sigma}= \norm{Z}_{\sigma}=1 \right\} \\
& = \frac{ \max\{ \norm{Y}_{\geodflow^{t}(\sigma)} : Y \in \xi^k(\sigma^+), \norm{Y}_\sigma = 1 \}}{\min\{ \norm{Z}_{\geodflow^{t}(\sigma)} : Z \in \xi^{d-k}(\sigma^-), \norm{Z}_\sigma = 1\}}.
\end{align*}
Notice that if $s,t \geq 0$, then 
\begin{equation}
\label{eqn:kappa_is_sub_multiplicative}
\kappa_{t+s}(\sigma) \leq \kappa_s(\geodflow^t(\sigma))\kappa_t(\sigma).
\end{equation}

By Proposition~\ref{prop: Hom bundles contraction/expansions}, to prove that $\rho$ is $\Psf_k$-Anosov relative to $X$ it suffices to show that $\kappa_t(\sigma)$ decays to zero exponentially fast in $t$.

\subsection{Estimates on the thick part}\label{sec:estimates on the thick part}

\begin{lemma}\label{lem:growth_on_thick_part} For any compact set $K \subset X$, there exists $C(K) > 1$ such that: if $\sigma \in \Gc$, $t \geq 0$, $\gamma \in \Gamma$, and $\sigma(0), \gamma\sigma(t) \in K$, then 
$$
\frac{1}{C(K)} \frac{\mu_{d-k+1}}{\mu_{d-k}}(\rho(\gamma)) \leq \kappa_{t}(\sigma) \leq C(K) \frac{\mu_{d-k+1}}{\mu_{d-k}}(\rho(\gamma)).
$$
\end{lemma}

The following proof is inspired by arguments of Tsouvalas \cite[Th.\ 1.1]{Kostas2020} (also see \cite[Prop.\ 6.5]{CZZ2021}).

\begin{proof} Suppose not. Then we can find sequences $(\sigma_n)_{n \geq 1}$ in $\Gc$, $(t_n)_{n \geq 1}$ in $[0,\infty)$, and $(\gamma_n)_{n \geq 1}$ in $\Gamma$ such that $\sigma_n(0), \gamma_n \sigma_n(t_n) \in K$ for all $n$ and 
\begin{equation}
\label{eqn:cont_hypothesis}
\lim_{n \to \infty} \abs{ \log\left(  \kappa_{t_n}(\sigma_n) \frac{\mu_{d-k}}{\mu_{d-k+1}}(\rho(\gamma_n))\right)} = \infty.
\end{equation}
Notice that we must have $t_n \to \infty$. 

Let $\wh{K} := \{ \sigma \in \Gc : \sigma(0) \in K\}$. Since $K$ is compact, we have 
$$
\norm{\cdot}_\sigma \asymp \norm{\cdot}_2
$$ 
for all $\sigma \in \wh{K}$. Then 
$$
\norm{\rho(\gamma_n)Y}_2 \asymp \norm{ \rho(\gamma_n) Y}_{\gamma_n \geodflow^{t_n}(\sigma_n)} =\norm{ Y}_{\geodflow^{t_n}(\sigma_n)} 
$$
for all $n \in \Nb$ and $Y \in \Rb^d$. So 
$$
\kappa_{t_n}(\sigma_n) \asymp \frac{ \max\{ \norm{\rho(\gamma_n)Y}_2 : Y \in \xi^k(\sigma_n^+), \norm{Y}_2 = 1 \}}{\min\{ \norm{\rho(\gamma_n)Z}_2 : Z \in \xi^{d-k}(\sigma_n^-), \norm{Z}_2 = 1\}}
$$
for all $n \in \Nb$. Thus by the max-min/min-max theorem for singular values, 
\begin{equation}
\label{eqn:lower_bd}
\kappa_{t_n}(\sigma_n) \gtrsim \frac{\mu_{d-k+1}}{\mu_{d-k}}(\rho(\gamma_n)).
\end{equation}

Passing to a subsequence we can suppose that $\sigma_n \to \eta_1$ and $\gamma_n\geodflow^{t_n}(\sigma_n) \to \eta_2$ in $\Gc$. Then $\gamma_n \to \eta_2^-$ and $\gamma_n^{-1} \to \eta_1^+$. Let $\rho(\gamma_n) = m_{n} a_n \ell_{n}$ be a singular value decomposition of $\rho(\gamma_n)$. Passing to a subsequence we can suppose that $m_{n} \to m$ and $\ell_n \to \ell$. Then, since the limit maps are strongly dynamics-preserving, Observation~\ref{obs:strongly_dynamics_pres_div_cartan} implies that
\begin{align}
\ell^{-1}\ip{ e_{d-k-1}, \dots, e_d} & = \xi^k(\eta_1^+) \quad \text{and} \label{eqn:k1_subspace}\\
m\ip{ e_{1}, \dots, e_{d-k}} & = \xi^{d-k}(\eta_2^-) \label{eqn:k2_subspace}.
\end{align}

Since $\xi^{d-k}(\sigma_n^-) \to \xi^{d-k}(\eta_1^{-})$ and $\xi^{d-k}(\eta_1^-)$ is transverse to $\xi^{k}(\eta_1^+)$, Equation~\eqref{eqn:k1_subspace} implies that 
\begin{equation}
\label{eqn:Z subspace estimate}
\min\{ \norm{\rho(\gamma_n)Z}_2 : Z \in \xi^{d-k}(\sigma_n^-), \norm{Z}_2 = 1\} \gtrsim \mu_{d-k}(\rho(\gamma_n)). 
\end{equation}

For each $n$, fix $Y_n \in \xi^k(\sigma_n^+)$ with $\norm{Y_n}_2 = 1$ and 
\begin{equation}
\label{eqn:Y subspace estimate}
 \norm{\rho(\gamma_n)Y_n}_2 = \max\{ \norm{\rho(\gamma_n)Y}_2 : Y \in \xi^k(\sigma_n^+), \norm{Y}_2 = 1 \}.
 \end{equation}
Then we can write $Y_n = Y_{1,n} + Y_{2,n}$ where $Y_{1,n} \in \ell_n^{-1}\ip{e_1,\dots, e_{d-k}}$ and $Y_{2,n} \in \ell_n^{-1} \ip{e_{d-k+1},\dots, e_d}$. We claim that 
$$
\norm{\rho(\gamma_n)Y_{1,n}}_2 \lesssim \norm{\rho(\gamma_n)Y_{2,n}}_2.
$$
If not we can pass to a subsequence so that 
$$
\lim_{n \to \infty} \frac{\norm{\rho(\gamma_n)Y_{1,n}}_2}{\norm{\rho(\gamma_n)Y_{2,n}}_2}=\infty.
$$
Passing to a further subsequence we can suppose that 
$$
V:=\lim_{n \to \infty} \frac{\rho(\gamma_n)Y_{n}}{\norm{\rho(\gamma_n)Y_{n}}_2}
$$
exists. Then by Equation~\eqref{eqn:k2_subspace}
$$
V = \lim_{n \to \infty} \frac{\rho(\gamma_n)Y_{n}}{\norm{\rho(\gamma_n)Y_{n}}_2} = \lim_{n \to \infty} \frac{\rho(\gamma_n)Y_{1,n}}{\norm{\rho(\gamma_n)Y_{1,n}}_2} \in m\ip{e_{1},\dots, e_{d-k}} = \xi^{d-k}(\eta_2^-). 
$$
However, 
$$
\frac{\rho(\gamma_n)Y_{n}}{\norm{\rho(\gamma_n)Y_{n}}_2} \in \rho(\gamma_n) \xi^k(\sigma_n^+) = \xi^k( (\gamma_n \sigma_n)^+)
$$
and so $V \in \xi^k(\eta_2^+)$. Thus we have a contradiction and thus 
$$
\norm{\rho(\gamma_n)Y_{1,n}}_2 \lesssim \norm{\rho(\gamma_n)Y_{2,n}}_2.
$$
Then 
$$
\norm{\rho(\gamma_n)Y_n}_2 \leq \norm{\rho(\gamma_n)Y_{n,1}}_2+\norm{\rho(\gamma_n)Y_{n,2}}_2 \lesssim \norm{\rho(\gamma_n)Y_{n,2}}_2 \leq \mu_{d-k+1}(\rho(\gamma_n)).
$$
Thus by Equations~\eqref{eqn:Z subspace estimate} and~\eqref{eqn:Y subspace estimate} we have
\begin{equation}
\label{eqn:upper_bd}
\kappa_{t_n}(\sigma_n) \lesssim \frac{\mu_{d-k+1}}{\mu_{d-k}}(\rho(\gamma_n)).
\end{equation}

Combining Equations~\eqref{eqn:cont_hypothesis}, \eqref{eqn:lower_bd}, and \eqref{eqn:upper_bd} gives a contradiction. 

\end{proof}

\begin{lemma}\label{lem:growth_between_boundaries} There exists $C_0 > 0$ such that: if $P \in \peripherals^\Gamma$, $\sigma \in \partial^+\Gc_P$, and $T_\sigma^+ < \infty$, then 
$$
\kappa_{T_{\sigma}^+}(\sigma) \leq C_0 e^{-\alpha T_{\sigma}^+}.
$$
\end{lemma}

\begin{proof} Fix $C(K) > 1$ satisfying Lemma~\ref{lem:growth_on_thick_part} for the compact set $K := \overline{ \Bc_X(\id, 1)}$. 

Fix $P \in \peripherals^\Gamma$ and  $\sigma \in \partial^+\Gc_P$ with $T_\sigma^+ < \infty$. By translating we can assume that $P \in \peripherals$, $\sigma(0)=(\id,2) \in \partial H_P$, and $\sigma(T_\sigma^+) = (\gamma, 2) \in \partial H_P$ for some $\gamma \in P$. Then $\sigma(0), \gamma^{-1}\sigma(T_\sigma^+) \in K$ and 
$$
\dist_X(\gamma, \id) \geq \dist_X \big( (\gamma, 2) ,(\id, 2) \big) -2 = T_\sigma^+ - 2. 
$$
So by Lemma~\ref{lem:growth_on_thick_part} and Equation~\eqref{eqn:growth on parabolics in construction of norms} 
\begin{align*}
\kappa_{T_{\sigma}^+}(\sigma)& \leq C(K) \frac{\mu_{d-k+1}}{\mu_{d-k}}(\rho(\gamma)^{-1}) = C(K) \frac{\mu_{k+1}}{\mu_{k}}(\rho(\gamma)) \\
& \leq C(K)e^{\beta} e^{-\alpha \dist_X(\gamma, \id)} \leq C(K)e^{\beta+2\alpha} e^{-\alpha T_{\sigma}^+}. 
\end{align*}

\end{proof}

\subsection{Contraction in the thin part}

\begin{lemma} \label{lem:thin_contraction}
There exists $C_1 > 0$ such that: if $P \in \peripherals^\Gamma$, $t \geq 0$, and $\geodflow^s(\sigma) \in \Gc_P$ for all $0 \leq s \leq t$, then 
$$
\kappa_t(\sigma) \leq C_1 e^{-\alpha t}. 
$$
\end{lemma} 

\begin{proof} We claim that $C_1 := \max\{1, C_0\}$ suffices where  $C_0$ is the constant from Lemma \ref{lem:growth_between_boundaries}. 

Fix $P \in \peripherals^\Gamma$, $t > 0$, and $\sigma \in \Gc_P$ where $\geodflow^s(\sigma) \in \Gc_P$ for all $0 \leq s \leq t$. We break the proof into a number of cases.

\medskip

\noindent \textbf{Case 1:} Assume $\sigma([0,\infty)) \subset \Gc_P$. Then $\sigma = \geodflow^s(\sigma_0)$ for some $s > 0$ and  $\sigma_0 \in \partial^+ \Gc_P$ with  $T_{\sigma_0}^+ = \infty$. Fix $Y \in \xi^k(\sigma^+)$ and non-zero $Z\in \xi^{d-k}(\sigma^-)$. Then
$$
\norm{Y}_{\geodflow^t(\sigma)}  =\norm{Y}_{\geodflow^{t+s}(\sigma_0)} =  e^{-\alpha (t+s)} \norm{Y}_{\sigma_0}= e^{-\alpha t} \norm{Y}_{\sigma}.
$$
We can decompose $Z = Z_2 + Z_3$ where $Z_2 \in  E_2|_\sigma = \xi^{d-k}(\sigma^+) \cap \xi^{d-k}(\sigma^-)$ and $Z_3 \in E_3|_\sigma = \xi^k(\sigma^-)$. Then
$$
\norm{Z}_{\geodflow^t(\sigma)}^2  = \norm{Z_2}_{\sigma_0}^2  +e^{\alpha (t+s)}\norm{Z_3}_{\sigma_0}^2   \geq \norm{Z_2}^2_{\sigma_0}  +e^{\alpha s}\norm{Z_3}_{\sigma_0}^2  = \norm{Z}_{\sigma}^2.
$$
So
$$
\frac{\norm{Y}_{\geodflow^t(\sigma)}}{\norm{Z }_{\geodflow^t(\sigma)}} \leq e^{-\alpha t} \frac{\norm{Y}_{\sigma} }{\norm{Z}_{\sigma} }.
$$
Since $Y$ and $Z$ were arbitrary, 
$$
\kappa_t(\sigma) \leq e^{-\alpha t} \leq C_1e^{-\alpha t}. 
$$

\noindent \textbf{Case 2:}  Assume $\sigma((-\infty,0]) \subset \Gc_P$. Arguing as in Case 1, one can show that 
$$
\kappa_t(\sigma) \leq e^{-\alpha t} \leq C_1e^{-\alpha t}. 
$$

\noindent \textbf{Case 3:}  Assume $\sigma([0,\infty))$ and $\sigma((-\infty,0])$ both intersect $\partial_P \Gc$. Then there exist $\sigma_0 \in \partial^+ \Gc_P$ and $s \in [0,T_{\sigma_0}^+]$ such that $T_{\sigma_0}^+ < \infty$ and $\sigma = \geodflow^s(\sigma_0)$. Let $T := T_{\sigma_0}^+$ and $\sigma_1 := \geodflow^{T}(\sigma_0)$.

\medskip

\noindent \textbf{Case 3(a):} Assume $s,t+s \in [0,T/3]$ or $s,t+s \in [2T/3,T]$. Then arguing as in Case 1, one can show that 
\begin{equation}
\label{eqn:caseA_conclusion}
\kappa_t(\sigma) \leq e^{-\alpha t}.
\end{equation}

 \noindent\textbf{Case 3(b):} Assume $s,t+s \in [T/3, 2T/3]$. Let $Q_0 := Q_{\geodflow^{T/3}(\sigma_0)}$ and $Q_1 := Q_{\geodflow^{2T/3}(\sigma_0)}$. Then 
$$
\Kb^d = E_1|_{\sigma} \oplus E_2|_{\sigma} \oplus E_3|_{\sigma}
$$
is an orthogonal decomposition with respect to $Q_0$ and $Q_1$. So by Proposition~\ref{prop:paths between inner products} we can fix a basis $v_1,\dots, v_d$ such that 
\begin{itemize}
\item $\ip{v_1,\dots, v_k} = E_1|_{\sigma}$,
\item $\ip{v_{k+1},\dots, v_{d-k}} = E_2|_{\sigma}$,
\item $\ip{v_{d-k+1},\dots, v_d} = E_3|_{\sigma}$,
\item $v_1,\dots, v_d$ is orthonormal with respect to $Q_0$,
\item $v_1,\dots, v_d$ is orthogonal with respect to $Q_1$, and
\item if $r \in [T/3, 2T/3]$, then $v_1,\dots, v_d$ is orthogonal with respect to $Q_{\geodflow^r(\sigma_0)}$ and
\begin{align*}
\norm{v_j}_{\geodflow^r(\sigma_0)}^2 & = Q_1(v_j,v_j)^{ \frac{3}{T} r - 1} 
\end{align*}
\end{itemize}

\medskip 
\noindent \textbf{Claim:} $\kappa_t(\sigma)^2 = \left(\frac{ \max_{1 \leq j \leq k} Q_1(v_j,v_j)}{\min_{k+1 \leq j \leq d} Q_1(v_j,v_j)} \right)^{ \frac{3}{T}t }$.
\medskip 

Since $v_1,\dots, v_k \in E_1|_\sigma= \xi^k(\sigma^+)$ and $v_{k+1},\dots, v_d \in E_2|_\sigma \oplus E_3|_\sigma=\xi^{d-k}(\sigma^-)$, we have 
$$
\kappa_t(\sigma)^2 \geq \left(\frac{ \max_{1 \leq j \leq k} Q_1(v_j,v_j)}{\min_{k+1 \leq j \leq d} Q_1(v_j,v_j)} \right)^{ \frac{3}{T}t }.
$$
For the  other inequality, fix $Y \in \xi^k(\sigma^+)$ and $Z \in \xi^{d-k}(\sigma^-)$ with $\norm{Y}_\sigma =\norm{Z}_\sigma= 1$. Writing $Y = \sum_{j=1}^k c_j v_j$, we have 
$$
1 = \norm{Y}_\sigma^2= \norm{Y}_{\geodflow^s(\sigma_0)}^2= \sum_{j=1}^k c_j^2 Q_1(v_j,v_j)^{ \frac{3}{T} s - 1}.
$$
Then 
$$
\norm{Y}_{\geodflow^t(\sigma)}^2 = \norm{Y}_{\geodflow^{s+t}(\sigma_0)}^2= \sum_{j=1}^k c_j^2 Q_1(v_j,v_j)^{ \frac{3}{T} (s+t) - 1} \leq \max_{1 \leq j \leq k} Q_1(v_j,v_j)^{ \frac{3}{T}t }.
$$
Likewise,  
$$
\norm{Z}_{\geodflow^t(\sigma)}^2 \geq \min_{k+1 \leq j \leq d} Q_1(v_j,v_j)^{ \frac{3}{T} t}.
$$
Hence 
$$
\kappa_t(\sigma)^2 \leq \left(\frac{ \max_{1 \leq j \leq k} Q_1(v_j,v_j)}{\min_{k+1 \leq j \leq d} Q_1(v_j,v_j)} \right)^{ \frac{3}{T}t }
$$
and the claim is established. 

Notice that this argument also implies that 
$$
\kappa_{T/3}( \geodflow^{T/3}(\sigma_0))^2 = \frac{ \max_{1 \leq j \leq k} Q_1(v_j,v_j)}{\min_{k+1 \leq j \leq d} Q_1(v_j,v_j)}
$$
and so
$$
\kappa_t(\sigma) =\kappa_{T/3}( \geodflow^{T/3}(\sigma_0)) ^{ \frac{3}{T}t}.
$$

By definition, 
$$
\frac{Q_1(v_j, v_j)}{Q_{\sigma_1}(v_j,v_j)}= 
\begin{cases} 
e^{\alpha T/3} & \text{ if } 1 \leq j \leq k \\
1 & \text{ if } k+1 \leq j \leq d-k \\
e^{-\alpha T/3} & \text{ if } d-k+1 \leq j \leq d
\end{cases}
$$
and 
$$
Q_{\sigma_0}(v_j, v_j) = \frac{Q_{\sigma_0}(v_j,v_j)}{Q_0(v_j, v_j)}= 
\begin{cases} 
e^{\alpha T/3} & \text{ if } 1 \leq j \leq k \\
1 & \text{ if } k+1 \leq j \leq d-k \\
e^{-\alpha T/3} & \text{ if } d-k+1 \leq j \leq d
\end{cases}.
$$
So 
\begin{align*}
\kappa_{T/3}( \geodflow^{T/3}(\sigma_0))^2 & = \frac{ \max_{1 \leq j \leq k} Q_1(v_j,v_j)}{\min_{k+1 \leq j \leq d} Q_1(v_j,v_j)} \leq e^{2\alpha T/3} \frac{ \max_{1 \leq j \leq k} Q_{\sigma_1}(v_j,v_j)}{\min_{k+1 \leq j \leq d} Q_{\sigma_1}(v_j,v_j)} \\
& \leq \kappa_T(\sigma_0)^2 e^{2\alpha T/3} \frac{ \max_{1 \leq j \leq k} Q_{\sigma_0}(v_j,v_j)}{\min_{k+1 \leq j \leq d} Q_{\sigma_0}(v_j,v_j)} = \kappa_T(\sigma_0)^2 e^{4\alpha T/3}.
\end{align*}
Then by Lemma~\ref{lem:growth_between_boundaries} 
$$
\kappa_{T/3}( \geodflow^{T/3}(\sigma_0))  \leq C_0 e^{-\alpha T} e^{2\alpha T/3} = C_0 e^{-\alpha\frac{T}{3}}.
$$
So 
\begin{equation}
\label{eqn:caseB_conclusion} 
\kappa_t(\sigma) =\kappa_{T/3}( \geodflow^{T/3}(\sigma_0)) ^{ \frac{3}{T}t} \leq C_0^{\frac{3t}{T}} e^{-\alpha t } \leq C_0 e^{-\alpha t} \leq C_1 e^{-\alpha t}
\end{equation}
(notice that we used the fact that $t \leq \frac{1}{3}T$ in the second inequality). 

\medskip

\noindent \textbf{Case 3(c):} Assume $s,t+s \in [0,T]$. We can divide the interval $[s,t+s]$ into at most three pieces so each piece is contained in one of $[0,T/3]$, $[T/3, 2T/3]$, or $[2T/3,T]$. Then Equations~\eqref{eqn:kappa_is_sub_multiplicative}, \eqref{eqn:caseA_conclusion}, and \eqref{eqn:caseB_conclusion} imply
\begin{equation*}
\kappa_t(\sigma) \leq C_1 e^{-\alpha t}. \qedhere
\end{equation*}

\end{proof} 

\subsection{Contraction everywhere} 

Now we combine our estimates on the thick and thin parts to show that $\rho$ is $\Psf_k$-Anosov relative to $X$. This part of the argument is similar to an analogous argument for geometrically finite Fuchsian groups in Section 6 in~\cite{CZZ2021}. 

Since $\rho$ is $\Psf_k$-Anosov relative to $\peripherals$, Observation~\ref{obs:strongly_dynamics_pres_div_cartan} implies that
$$
\lim_{\gamma \to \infty}  \frac{\mu_{d-k+1}}{\mu_{d-k}}(\rho(\gamma)) = \lim_{\gamma \to \infty}  \frac{\mu_{k+1}}{\mu_{k}}(\rho(\gamma)^{-1})=0. 
$$
Then, since $\Gamma$ acts cocompactly on $\Gc_{thick}$,  by Lemma~\ref{lem:growth_on_thick_part}  there exists $T_0 > 1$ such that: if $\sigma \in \Gc_{thick}$, $t \geq T_0$, and $\geodflow^t(\sigma) \in \Gc_{thick}$, then 
$$
\kappa_t(\sigma) \leq \frac{1}{2C_1^2}
$$
(where $C_1$ is the constant in Lemma~\ref{lem:thin_contraction}). Next, since $\Gamma$ acts cocompactly on  $\bigcup_{P \in \peripherals^\Gamma} \partial \Gc_P$, there exists $C_2 > 1$ such that 
\begin{equation}
\kappa_t(\sigma) \leq C_2
\end{equation} 
for all $\sigma \in \bigcup_{P \in \peripherals^\Gamma} \partial \Gc_P$ and $t \in [0,T_0]$. Fix $T > T_0$ sufficiently large so that 
$$
C_2 C_1^2 e^{-\alpha(T-T_0)} \leq \frac{1}{2} \quad \text{and} \quad C_1 e^{-\alpha T} \leq \frac{1}{2}.
$$

\begin{lemma}\label{lem:time_T_contraction} If $\sigma \in \Gc$ and $t \geq T$, then $\kappa_t(\sigma) \leq \frac{1}{2}$.
\end{lemma} 

\begin{proof} Fix $\sigma \in \Gc$ and $t \geq T$. If $\phi^s(\sigma) \in \Gc_{thin}$ for all $s \in [0,t]$, then Lemma~\ref{lem:thin_contraction} implies that
$$
\kappa_t(\sigma) \leq C_1 e^{-\alpha t} \leq C_1 e^{-\alpha T} \leq \frac{1}{2}.
$$
So we can suppose that $\phi^s(\sigma) \in \Gc_{thick}$ for some $s \in [0,t]$. Then define 
$$
s_1 := \min\{ s \in [0,t] :\phi^s(\sigma) \in \Gc_{thick}\} \quad \text{and} \quad s_2 := \max\{ s \in [0,t] : \phi^s(\sigma) \in \Gc_{thick}\}.
$$
If $s_2 -s_1 \geq T_0$, then 
$$
\kappa_t(\sigma) \leq \kappa_{t-s_2}(\geodflow^{s_2}(\sigma)) \cdot \kappa_{s_2-s_1}(\geodflow^{s_1}(\sigma)) \cdot \kappa_{s_1}(\sigma) \leq C_1 e^{-\alpha(t-s_2)}\cdot \frac{1}{2C_1^2} \cdot C_1 e^{-\alpha s_1} \leq \frac{1}{2}.
$$
If $s_2 - s_1 \leq T_0$, then 
\begin{align*}
\kappa_t(\sigma) &\leq \kappa_{t-s_2}(\geodflow^{s_2}(\sigma)) \cdot \kappa_{s_2-s_1}(\geodflow^{s_1}(\sigma)) \cdot \kappa_{s_1}(\sigma) \leq C_1 e^{-\alpha(t-s_2)}\cdot C_2 \cdot C_1 e^{-\alpha s_1} \\
& \leq C_2 C_1^2e^{-\alpha(s_1+t-s_2)} \leq C_2 C_1^2 e^{-\alpha(T-T_0)} \leq \frac{1}{2}.
\end{align*}
So in all cases $\kappa_t(\sigma) \leq \frac{1}{2}$ and the proof is complete. 
\end{proof}

\begin{lemma} There exists $C_3 \geq 1$ such that: if $\sigma \in \Gc$ and $t \in [0,T]$, then $\kappa_t(\sigma) \leq C_3$. \end{lemma}

\begin{proof} Let 
$$
\Gc_{thicker} :=\{ \sigma \in \Gc: \geodflow^t(\sigma) \in \Gc_{thick} \text{ for some } t \in [0,T]\}.
$$
Then $\Gamma$ acts cocompactly on $\Gc_{thicker}$ and so 
$$
C_3:= \max \left\{ 1, C_1, \max\{ \kappa_t(\sigma) : t \in [0,T] \text{ and } \sigma \in \Gc_{thicker}\} \right\}
$$
is finite. 

If $\sigma \in \Gc$, then either $\geodflow^t(\sigma) \in \Gc_{thin}$ for all $t \in [0,T]$, in which case Lemma~\ref{lem:thin_contraction} implies that
$$
\kappa_t(\sigma) \leq C_1 e^{-\alpha t} \leq C_1 \leq C_3
$$
for all $t \in [0,T]$, or $\sigma \in \Gc_{thicker}$ in which case $\kappa_t(\sigma) \leq C_3$ for all $t \in [0,T]$. 
\end{proof}

\begin{lemma} There exists $c > 0$ such that: if $\sigma \in \Gc$ and $t \geq 0$, then $\kappa_t(\sigma) \leq 2C_3 e^{-c t}$. \end{lemma}

\begin{proof} We claim that $c: = \frac{\log(2)}{T}$ suffices. 

If $t \geq T$, then we can break the interval $[0,t]$ into $\lfloor t/T \rfloor$ subintervals each with length at least $T$. Then using Equation~\eqref{eqn:kappa_is_sub_multiplicative} and Lemma~\ref{lem:time_T_contraction} we have 
$$
\kappa_t(\sigma) \leq \left(\frac{1}{2} \right)^{\lfloor t/T \rfloor} \leq 2\left(\frac{1}{2} \right)^{t/T} = 2 e^{-\frac{\log(2)}{T} t} \leq 2C_3 e^{-c t}. 
$$
If $t < T$, then 
\begin{equation*}
\kappa_t(\sigma) \leq C_3 = 2C_3 e^{-\log(2)} \leq 2C_3 e^{-c t}.  \qedhere
\end{equation*}
\end{proof} 

Then by Proposition~\ref{prop: Hom bundles contraction/expansions}  and the definition of $\kappa_t$, we see that $\rho$ is $\Psf_k$-Anosov relative to $X$.

\section{Uniformly Anosov representations}\label{sec: uniformly Anosov} 

In this section we prove the claims in Theorem~\ref{thm:uniform_stability_case} for a single representation. Later, in Section~\ref{sec:uniform stability}, we will complete the proof of  Theorem~\ref{thm:uniform_stability_case}  by observing that one can obtain uniform estimates over a small neighborhood in the constrained representation variety. The proofs in this section are slightly inefficient due to the need to carefully track constants for this later work.

\begin{theorem}\label{thm:estimates for uniform representations} Suppose that $(\Gamma,\peripherals)$ is relatively hyperbolic, $X$ is a weak cusp space for $(\Gamma, \peripherals)$, and $\rho \colon \Gamma \to \SL(d,\Kb)$ is uniformly $\Psf_k$-Anosov relative to $X$. Then:
\begin{enumerate}
\item There exists a $\rho$-equivariant quasi-isometric embedding 
$$
X \to \SL(d,\Kb) / \SU(d,\Kb).
$$
\item The Anosov boundary map 
$$
\xi \colon \partial_\infty X \to \Gr_k(\Kb^d) \times \Gr_{d-k}(\Kb^d)
$$
is H\"older relative to any visual metric on $\partial_\infty X$ and any Riemannian distance on $\Gr_k(\Kb^d) \times \Gr_{d-k}(\Kb^d)$. 

\end{enumerate}
\end{theorem}

The rest of the section is devoted to the proof of Theorem~\ref{thm:estimates for uniform representations}. Suppose $(\Gamma,\peripherals)$ is relatively hyperbolic, $X$ is a weak cusp space for $(\Gamma, \peripherals)$, and $\rho \colon \Gamma \to \SL(d,\Kb)$ is uniformly $\Psf_k$-Anosov relative to $X$. Let $\xi \colon \partial_\infty X \to \Gr_k(\Kb^d) \times \Gr_{d-k}(\Kb^d)$ denote the Anosov boundary map. 

Let $\dist_M$ denote the symmetric space distance on $M:=\SL(d,\Kb) / \SU(d,\Kb)$ defined by Equation~\eqref{eqn:symmetric distance in prelims} and let $\mathsf{K}:=\SU(d,\Kb)$. Then there exists $\alpha_0>1$, which only depends on $d$, such that 
\begin{equation}
\label{eqn:bounds in symmetric space}
\frac{1}{\alpha_0} \log \frac{\mu_{1}}{\mu_{d}}\left(  g^{-1}h \right) \leq \dist_M(g \mathsf{K},h \mathsf{K}) \leq \alpha_0 \log \frac{\mu_{1}}{\mu_{d}}\left(g^{-1}h \right)
\end{equation}
for all $g,h \in \SL(d,\Kb)$. 

By hypothesis and Proposition~\ref{prop: Hom bundles contraction/expansions}, there exists a family of norms $\norm{\cdot}$ on the fibers of $\Gc(X) \times \Kb^d \to \Gc(X)$ such that:
\begin{itemize}
\item Each $\norm{\cdot}_\sigma$ is induced by an inner product on $\Kb^d$. 
\item $\norm{\rho(\gamma)(\cdot)}_{\gamma \sigma} = \norm{\cdot}_\sigma$ for all $\gamma \in \Gamma$ and $\sigma \in \Gc(X)$.
\item For any $r \geq 0$, there is some $L_r \geq 1$ such that:
\begin{align}
\label{eqn:definition of Lr in Section uniform}
\frac{1}{L_r} \norm{\cdot}_{\sigma_1} \leq \norm{\cdot}_{\sigma_2} \leq L_r \norm{\cdot}_{\sigma_1}
\end{align}
for all $\sigma_1, \sigma_2 \in \Gc(X)$ with $\dist_X(\sigma_1(0), \sigma_2(0)) \leq r$. 
\item There are $c, C > 0$ such that
$$
\frac{\norm{Y }_{\geodflow^t(\sigma)}}{\norm{Z }_{\geodflow^t(\sigma)}} \leq Ce^{-ct} \frac{\norm{Y}_{\sigma} }{\norm{Z}_{\sigma} }
$$
for all $t \geq 0$, $\sigma \in \Gc(X)$, $Y \in \xi^k(\sigma^+)$, and non-zero $Z\in \xi^{d-k}(\sigma^-)$.
\end{itemize} 

As in Section~\ref{sec:contraction_Anosov}, since each $\norm{\cdot}_\sigma$ is induced by an inner product, for every $\sigma \in \Gc(X)$ there exists a matrix $A_\sigma \in \GL(d,\Kb)$ such that 
\begin{equation}
\label{eqn:definition of Asigma in Section uniform}
\norm{A_\sigma(\cdot)}_{\sigma} = \norm{\cdot}_2.
\end{equation}
It is convenient to make the following normalization: by possibly replacing each $A_\sigma$ by an element in the coset $A_\sigma \mathsf{U}(d,\Kb)$ we may assume that 
 \begin{equation}
 \label{eqn:positivity of determinant} 
 \det(A_\sigma) > 0
 \end{equation} 
for all $\sigma \in \Gc(X)$. 

We start by observing some estimates on the singular values of the matrices $A_\sigma$. By Lemma~\ref{lem:decay of ratio of singular values}
\begin{equation}
\label{eqn: repeating lem:decay of ratio of singular values}
\frac{\mu_{k+1}}{\mu_{k}}\left(A_\sigma^{-1} A_{\geodflow^{t}(\sigma)}\right) \leq C e^{-c t}
\end{equation}
for any $\sigma \in \Gc(X)$ and $t \geq 0$.

\begin{observation}\label{obs:the A_sigma have local bounds} If $\sigma_1, \sigma_2 \in \Gc(X)$ and $\dist_X(\sigma_1(0), \sigma_2(0)) \leq r$, then 
$$
\frac{\mu_1}{\mu_d}(A_{\sigma_1}^{-1}A_{\sigma_2}) \leq L_r^2.
$$
\end{observation} 

\begin{proof} Notice that 
\begin{align*}
\norm{A_{\sigma_1}^{-1} A_{\sigma_2}(\cdot)}_2 = \norm{A_{\sigma_2}(\cdot)}_{\sigma_1} \leq L_r  \norm{A_{\sigma_2}(\cdot)}_{\sigma_2} = L_r \norm{\cdot}_2. 
\end{align*}
So $\mu_1( A_{\sigma_1}^{-1} A_{\sigma_2}) \leq L_r$. Also, by symmetry, 
\begin{equation*}
\frac{1}{\mu_{d}}\left(A_{\sigma_1}^{-1} A_{\sigma_2}\right)  = \mu_1\left(A_{\sigma_2}^{-1} A_{\sigma_1}\right)  \leq L_r
\end{equation*}
which completes the proof of the observation. 
\end{proof} 

\subsection{Quasi-isometric embedding of the entire weak cusp space} \label{sec:QI of entire weak cusp space}

By Proposition~\ref{prop:abundance of geodesic lines}  there is some $R_1 > 0$ with the following property: for all $p,q \in X$, there exists a geodesic line $\sigma \colon \Rb \to X$ such that  
$$
p,q \in \Nc_X(\sigma, R_1).
$$

Let $C_1 := L_{R_1}^2$ and $c_1 := \frac{1}{R_1}\log(L_{R_1}^2)$. Then repeated applications of Observation~\ref{obs:the A_sigma have local bounds} implies that 
\begin{equation}
\label{eqn: upper bd singular value ratio in uniform proof}
\frac{\mu_{1}}{\mu_{d}}\left(A_\sigma^{-1} A_{\geodflow^{t}(\sigma)}\right) \leq \left( L_{R_1}^2 \right)^{ \left\lceil \frac{\abs{t}}{R_1} \right\rceil } \leq C_1 e^{c_1 \abs{t}}
\end{equation}
for any $\sigma \in \Gc(X)$ and $t \in \Rb$.

Fix a subset $\Fc \subset X$ such that 
$$
X = \bigsqcup_{p \in \Fc} \Gamma(p)
$$
is a disjoint union. We define an equivariant map from $X$ into the set of finite subsets of  $\Gc(X)$ as follows: 
\begin{itemize}
\item If $p \in \Fc$, let $\sigma_p \in \Gc(X)$ be any geodesic line with $\dist_X(p, \sigma_p(0)) \leq R_1$. Then let 
$$
S_p := \Stab_{\Gamma}(p) \cdot \sigma_p \subset \Gc(X). 
$$ 
\item If $p = \gamma(q)$ for some $\gamma \in \Gamma$ and $q \in \Fc$, then define $S_p := \gamma S_q$. 
\end{itemize}
Since $\Gamma$ acts properly discontinuously on $X$, each $S_p$ is a finite set. Further, if $p \in X$ and $\sigma \in S_p$, then 
\begin{equation}
\label{eqn:diam of S_p}
\dist_X(p,\sigma(0)) \leq R_1.
\end{equation}

Since the symmetric space $M=\SL(d,\Kb) /\SU(d,\Kb)$ is simply connected and non-positively curved, for any finite set $S \subset M$ the function 
$$
E_S(x) = \max_{s \in S} \dist_M(s,x)
$$
has a unique minimum point in $M$ (see ~\cite[Chap.\ 6.2.2]{P2016}) which we denote by ${\rm CoM}(S)$. By construction 
$$
g\; {\rm CoM}(S) = {\rm CoM}(gS) \quad \text{for all} \quad g \in \SL(d,\Kb)
$$
and 
\begin{equation}
\label{eqn:diam of COM}
\max_{s \in S} \dist_M(s,  {\rm CoM}(S) ) \leq \max_{s_1,s_2\in S} \dist_M( s_1, s_2).
\end{equation}

For $\sigma \in \Gc(X)$, let $\bar{A}_\sigma := \det(A_{\sigma})^{-1/d}A_{\sigma}  \in \SL(d,\Kb)$. Then define $F \colon X \to M$ by 
$$
F(p)= {\rm CoM} \left\{ \bar{A}_\sigma \mathsf{K} : \sigma \in S_p\right\}
$$
(recall that $\mathsf{K}=\SU(d,\Kb)$). 

\begin{lemma}\label{lem: qi embedded approximated by one} If $\sigma \in \Gc(X)$, $p \in X$, and $\dist_X(p,\sigma(0)) \leq R_1$, then 
$$
\dist_M\left(F(p), \bar{A}_\sigma \mathsf{K} \right) \leq 2\alpha_0 \log L_{2R_1}^2.
$$
\end{lemma} 

\begin{proof} By Equation~\eqref{eqn:diam of S_p} and Observation~\ref{obs:the A_sigma have local bounds}
$$
\max_{\sigma_1, \sigma_2 \in S_p} \frac{\mu_{1}}{\mu_{d}}\left(A_{\sigma_1}^{-1} A_{\sigma_2} \right) \leq L_{2R_1}^2.
$$
So by Equations~\eqref{eqn:bounds in symmetric space} and ~\eqref{eqn:diam of COM}
$$
\max_{\sigma_1 \in S_p}\dist_M\left( F(p),\bar{A}_{\sigma_1} \mathsf{K} \right) \leq \alpha_0 \log L_{2R_1}^2.
$$
Similar reasoning shows that 
$$
\max_{\sigma_1 \in S_p}\dist_M\left( \bar{A}_{\sigma} \mathsf{K} , \bar{A}_{\sigma_1} \mathsf{K} \right) \leq \alpha_0 \log L_{2R_1}^2
$$
which completes the proof. 
\end{proof}

\begin{lemma}\label{lem:unform qiembed entire space} $F$ is a $\rho$-equivariant quasi-isometric embedding with constants only depending on $d$, $L_{2R_1}$, $c$, $C$, and $R_1$. \end{lemma} 

\begin{proof} We first verify that $F$ is  $\rho$-equivariant. If $\sigma \in \Gc(X)$ and $\gamma \in \Gamma$, then by definition 
$$
\norm{\rho(\gamma)A_{\sigma}(\cdot) }_{\gamma  \sigma} = \norm{A_{\sigma}(\cdot)}_{\sigma} = \norm{\cdot}_2
$$
and so $\rho(\gamma)A_{\sigma}=A_{\gamma \sigma}g_{\gamma, \sigma}$ for some $g_{\gamma, \sigma} \in \mathsf{U}(d,\Kb)$. By Equation~\eqref{eqn:positivity of determinant} we must have $g_{\gamma, \sigma} \in \mathsf{SU}(d,\Kb) = \mathsf{K}$. Then if $p \in X$ and $\gamma \in \Gamma$, we have 
\begin{align*}
\rho(\gamma) \left\{ \bar{A}_{\sigma} \mathsf{K} : \sigma \in S_p\right\} & = \left\{ \bar{A}_{\gamma \sigma} \mathsf{K} : \sigma \in S_p\right\}  = \left\{ \bar{A}_{\sigma} \mathsf{K} : \sigma \in S_{\gamma(p)} \right\}.
\end{align*}
So $\rho(\gamma) F(p) = F(\gamma(p))$ and thus $F$ is  $\rho$-equivariant. 

To show that $F$ is a quasi-isometric embedding, fix $p,q \in X$. Then fix a geodesic line $\sigma \in \Gc(X)$ and $T \geq 0$ such that 
$$
\dist_X(p,\sigma(0)) \leq R_1 \quad \text{and} \quad \dist_X(q, \sigma(T)) \leq R_1. 
$$
Notice that $\abs{T-\dist_X(p,q)} \leq 2R_1$, and Lemma~\ref{lem: qi embedded approximated by one} implies that 
$$
\abs{ \dist_M(F(p), F(q)) - \dist_M\left( \bar{A}_{\sigma} \mathsf{K}, \bar{A}_{\geodflow^T(\sigma)} \mathsf{K}\right)} \leq 4\alpha_0 \log L_{2R_1}^2.
$$

Then Equations~\eqref{eqn:bounds in symmetric space}, \eqref{eqn: repeating lem:decay of ratio of singular values}, and \eqref{eqn: upper bd singular value ratio in uniform proof}, imply that $F$ is a $(\alpha, \beta)$-quasi-isometric embedding where $\alpha := \max\left\{ \frac{1}{c\alpha_0}, c_1\alpha_0 \right\} $ and 
\begin{equation*}
\beta :=2\alpha_0 R_1\max\{c,c_1\}+ 4\alpha_0 \log \left( L_{2R_1}^2\right)+\alpha_0 \max\{ \log(C), \log(C_1)\}
\end{equation*}
Recall, $c_1$ and $C_1$ only depend on $L_{R_1}$ and $R_1$. So we can choose the quasi-isometric constants to depend only on $d$, $L_{2R_1}$, $c$, $C$, and $R_1$.
\end{proof}

\subsection{H\"older regularity of the boundary maps} \label{subsec:uniform Holder reg}

The key step in the proof of H\"older regularity is to make Lemma~\ref{lem:convergence of uK} quantitative. 

\begin{lemma}\label{lem:quantitative convergence}There exist $C_2, T_0 > 0$ (which only depend on $c$, $C$, and $L_1$) such that: if $\sigma_1, \sigma_2 \in \Gc(X)$, $r > 0$, $\dist_X(\sigma_1(0), \sigma_2(0)) \leq r$, and $t > T_0 + \frac{2}{c} \log L_r$, then
$$
\dist_{\Gr_k(\Kb^d)}\left( U_k\left(A_{\sigma_1}^{-1}A_{\geodflow^t(\sigma_2)}\right), A_{\sigma_1}^{-1}\xi^k(\sigma_2^+) \right) \leq  C_2L_r^2e^{-c t}. 
$$
\end{lemma}

\begin{proof} For ease of notation, let $B_t = A_{\geodflow^t(\sigma_2)}$.

Let $T_0 := \frac{1}{c} \log C$. Observation~\ref{obs:the A_sigma have local bounds} and Equation~\eqref{eqn: repeating lem:decay of ratio of singular values} imply that
\begin{equation}
\label{eqn:singular value gap of AB t} 
\frac{\mu_{k+1}}{\mu_{k}}\left(A_{\sigma_1}^{-1}B_t\right) \leq \frac{\mu_1}{\mu_d}\left(A_{\sigma_1}^{-1}B_0\right) \frac{\mu_{k+1}}{\mu_{k}}\left(B_0^{-1}B_t\right) \leq L_r^2 C e^{-c t}
\end{equation}
and so $U_k\left(A_{\sigma_1}^{-1}B_t\right)$ is well defined when $t > T_0 + \frac{2}{c} \log L_r$. Further, Lemma~\ref{lem: U_k under product}  and Lemma~\ref{lem:convergence of uK} imply that 
$$
\lim_{t \to \infty} U_k\left(A_{\sigma_1}^{-1}B_t\right) = \lim_{t \to \infty} A_{\sigma_1}^{-1} B_0U_k\left(B^{-1}_0B_t\right) =A_{\sigma_1}^{-1}\xi^k(\sigma_2^+).
$$
Then 
$$
 \dist_{\Gr_k(\Kb^d)}\left( U_k\left(A_{\sigma_1}^{-1}B_{t}\right), A_{\sigma_1}^{-1}\xi^k(\sigma_2^+) \right) \leq \sum_{j=0}^\infty  \dist_{\Gr_k(\Kb^d)}\left( U_k\left(A_{\sigma_1}^{-1}B_{t+j}\right), U_k\left(A_{\sigma_1}^{-1}B_{t+j+1}\right) \right). 
 $$
Then by Lemma~\ref{lem: U_k under product}(1), Observation~\ref{obs:the A_sigma have local bounds}, and Equation~\eqref{eqn:singular value gap of AB t}, we have 
\begin{align*}
 \dist_{\Gr_k(\Kb^d)}&\left( U_k\left(A_{\sigma_1}^{-1}B_{t}\right), A_{\sigma_1}^{-1}\xi^k(\sigma_2^+) \right) \leq \sum_{j=0}^\infty \frac{\mu_1}{\mu_d}\left( B_{t+j+1}^{-1} B_{t+j}\right)\frac{\mu_{k+1}}{\mu_k}\left(A_{\sigma_1}^{-1}B_{t+j+1} \right) \\
 & \leq  \sum_{j=0}^\infty L_1^2\cdot L_r^2C e^{-c (t+j+1)} = \frac{L_1^2Ce^{-c}}{1-e^{-c}} L_r^2e^{-ct}. 
\end{align*}
So $C_2:= L_1^2Ce^{-c}(1-e^{-c})^{-1}$ suffices. 
\end{proof} 

Fix $\delta > 1$ such that every (possibly ideal) geodesic triangle in $X$ is $\delta$-slim (i.e.\ each side is contained in the $\delta$-neighborhood of the union of the two other sides).

\begin{lemma} 
There exist $C_3 > 0$ (which only depends on $\delta$, $c$, $C$, and $L_{2\delta}$) such that: if $\sigma \in \Gc(X)$ and $y \in \partial_\infty X \smallsetminus \{\sigma^+\}$,  then
$$
\dist_{\Gr_k(\Kb^d)}\left( A_\sigma^{-1} \xi^k(\sigma^+), A_{\sigma}^{-1} \xi^k(y) \right) \leq C_3e^{-c \dist_X(\sigma(0), \eta)}
$$
where $\eta \in \Gc(X)$ is any geodesic line with $\eta^- = \sigma^+$ and $\eta^+ = y$. 
\end{lemma}

\begin{proof} Let $T_0$ be as in Lemma~\ref{lem:quantitative convergence} and let $T :=T_0 + \frac{2}{c} \log L_{\delta}$. 

\medskip
\noindent \textbf{Case 1:} Assume $\dist_X(\sigma(0), \eta) \leq 6\delta + T$. Then 
\begin{align*}
\dist_{\Gr_k(\Kb^d)}&\left( A_\sigma^{-1} \xi^k(\sigma^+), A_{\sigma}^{-1} \xi^k(y) \right) \leq {\rm diam} \Gr_k(\Kb^d) = \frac{\pi}{2} \\
& \leq  \left( \frac{\pi}{2e^{6\delta c+Tc}} \right)e^{-c \dist_X(\sigma(0), \eta)}.
\end{align*}

\noindent \textbf{Case 2:} Assume $\dist_X(\sigma(0), \eta) > 6\delta + T$. Let $\hat{\sigma} \in \Gc(X)$ be geodesic line with $\hat{\sigma}^+ = y$ and $\hat{\sigma}^- = \sigma^-$. 
If $\sigma^- = \eta^+$, then $\sigma \cup \eta$ is a degenerate ideal triangle and hence $\delta$-slim. So 
$$
 \dist_X(\sigma(0), \eta) \leq \delta,
$$
which is impossible in Case 2. Hence $\sigma^- \neq \eta^+$.

Since the ideal geodesic triangle $\sigma \cup \hat{\sigma} \cup \eta$ is $\delta$-slim and $\dist_X(\sigma(0), \eta) > \delta$, we can parametrize $\hat{\sigma}$ so that $\dist_X(\sigma(0), \hat{\sigma}(0)) \leq \delta$. Also, since $\eta$ is contained in the $\delta$-neighborhood of $\hat{\sigma} \cup \sigma$, we can pick $q \in \eta$ such that 
$$
\max\{ \dist_X(q, \sigma), \dist_X(q,\hat{\sigma}) \} \leq \delta. 
$$
Fix $t_0, \hat{t}_0 \in \Rb$ such that $\dist_X(q,\sigma(t_0)) \leq \delta$ and $\dist_X(q,\hat{\sigma}(\hat{t}_0)) \leq \delta$.

If $t_0 \leq 0$ or $\hat{t}_0 \leq 0$, then Observation~\ref{lem:asymp_geodesics} implies that 
$$
\dist_X(\sigma(0), \eta) \leq 6\delta,
$$ 
which is impossible in Case 2.  If $t_0$ or $\hat{t}_0$ is contained in $[0,T]$, then 
$$
\dist_X(\sigma(0), \eta) \leq 2\delta + T,
$$
which is impossible in Case 2. Thus $t_0, \hat{t}_0 > T$.  Then by Lemma~\ref{lem:quantitative convergence}
\begin{align*}
\dist_{\Gr_k(\Kb^d)} \left( A_\sigma^{-1} \xi^k(\sigma^+), A_{\sigma}^{-1} \xi^k(y) \right) & \leq C_2 L_0^2 e^{-c t_0}  + C_2 L_\delta^2 e^{-c \hat{t}_0}  \\
& \phantom{\leq} +\dist_{\Gr_k(\Kb^d)}\left( U_k\left(A_{\sigma}^{-1}A_{\geodflow^{t_0}(\sigma)}\right), U_k\left(A_{\sigma}^{-1}A_{\geodflow^{\hat{t}_0}(\hat{\sigma})}\right) \right).
\end{align*}
By Lemma~\ref{lem: U_k under product}, Observation~\ref{obs:the A_sigma have local bounds}, and Equation~\eqref{eqn: repeating lem:decay of ratio of singular values}
$$
\dist_{\Gr_k(\Kb^d)}\left( U_k\left(A_{\sigma}^{-1}A_{\geodflow^{t_0}(\sigma)}\right), U_k\left(A_{\sigma}^{-1}A_{\geodflow^{\hat{t}_0}(\hat{\sigma})}\right) \right) \leq L_{2\delta}^2 Ce^{-ct_0}.
$$
Since $t_0 \geq \dist_X(\sigma(0), \eta) -\delta$ and $\hat{t}_0 \geq \dist_X(\sigma(0), \eta) -2\delta$, we then have 
\begin{align*}
\dist_{\Gr_k(\Kb^d)}\left( A_\sigma^{-1} \xi^k(\sigma^+), A_{\sigma}^{-1} \xi^k(y) \right) & \leq (C_2 L_{0}^2 e^{\delta c} + C_2 L_{1}^2 e^{2\delta c} + CL_{2\delta}^2 e^{\delta c} ) e^{-c \dist_X(\sigma(0), \eta)} \\
& \leq (2C_2+C)L_{2\delta}^2 e^{2\delta c}  e^{-c \dist_X(\sigma(0), \eta)}.
\end{align*}
This completes the proof of the lemma. 
\end{proof} 

Fix $p_0 \in X$ and a visual distance $\dist_\infty$ on $\partial_\infty X$. By definition, there exist $C_4 > 1$, $\lambda > 0$ such that 
$$
\frac{1}{C_4} e^{-\lambda \dist_X(p_0, \sigma_{xy}) } \leq \dist_\infty(x,y) \leq C_4 e^{-\lambda  \dist_X(p_0, \sigma_{xy})}
$$
for all $x,y \in \partial_\infty X$ and all geodesic lines $\sigma_{xy}$ with $\sigma_{xy}^+ = y$ and $\sigma_{xy}^- =x$.  

Also, fix a compact set $K \subset \Gc(X)$ such that
$$
\partial_\infty X = \{ \sigma^+ : \sigma \in K\}.
$$
By continuity, there exists $C_{K}>1$ so that if $\sigma \in K$, then $\norm{\cdot}_\sigma$ is $C_{K}$-bilipschitz to the standard Euclidean norm $\norm{\cdot}_2$ on $\mathbb K^d$. Then Equation~\eqref{eqn:definition of Asigma in Section uniform} implies that
$$
\frac{\mu_1}{\mu_d}(A_\sigma) \leq C_{K}^2
$$
for all $\sigma \in K$. Finally, let  $R_2: = \max\{ \dist_X(p_0, \sigma(0)) : \sigma \in K\}$.

\begin{lemma}\label{lem:Holder regularity} There exist $C_5 > 0$ (which only depends on $\delta$, $d$, $c$, $C$, $L_{2\delta}$, $C_4$, $\lambda$, $C_{K}$, and $R_2$) such that: if $x,y \in \partial_\infty X$, then
$$
\dist_{\Gr_k(\Kb^d)}\left( \xi^k(x),\xi^k(y) \right) \leq C_5 \dist_\infty(x,y)^{c/\lambda}.
$$
\end{lemma} 

\begin{proof} By compactness, there exists $C^\prime > 1$ (which only depends on $C_{K}$ and $d$) such that: if $g \in \SL(d,\Kb)$ and $\frac{\mu_1}{\mu_d}(g) \leq C_{K}^2$, then 
$$
\dist_{\Gr_k(\Kb^d)}\left( g V_1, gV_2 \right) \leq C^\prime \dist_{\Gr_k(\Kb^d)}(V_1,V_2)
$$
for all $V_1, V_2 \in \Gr_k(\Kb^d)$. 

Fix $x,y \in \partial_\infty X$ distinct. Then fix $\sigma \in K$ such that $\sigma^+ = x$. Then
\begin{align*}
\dist_{\Gr_k(\Kb^d)}\left( \xi^k(x),\xi^k(y) \right) & \leq C^\prime \dist_{\Gr_k(\Kb^d)}\left( A_\sigma^{-1} \xi^k(\sigma^+), A_{\sigma}^{-1} \xi^k(y) \right) \\
& \leq  C^\prime C_3e^{-c \dist_X(\sigma(0), \eta)} \leq C^\prime C_3e^{cR_2} e^{-c\dist_X(p_0, \eta)} \\
& \leq C^\prime C_3 C_4^{c/\lambda}  \dist_\infty(x,y)^{c/\lambda}
\end{align*}
where $\eta \in \Gc(X)$ is a geodesic line with $\eta^- = \sigma^+$ and $\eta^+ = y$. 
\end{proof}

\section{Uniform relatively Anosov and relatively Morse representations} \label{sec:relMorse_relunifAnosov}

Relatively Morse representations were introduced in \cite{KL}, building on definitions and work in \cite{KLP2018b}. In this section we will show that they are closely related to the uniform relatively Anosov representations introduced in this paper. 

In what follows let, endow $M:=\SL(d,\Kb) / \SU(d,\Kb)$ with the symmetric space distance defined by Equation~\eqref{eqn:symmetric distance in prelims} and let $\mathsf{K}:=\SU(d,\Kb)$.

\begin{definition}
Suppose that $I \subset \Rb$ is a finite or infinite interval. A quasi-geodesic $q\colon I \to M$ is \emph{$P_k$-Morse} with constants $\alpha, \beta > 0$ if 
$$
\log \frac{\mu_k}{\mu_{k+1}}(h_s^{-1} h_t) \geq \alpha \log \frac{\mu_1}{\mu_{d}}(h_s^{-1} h_t) - \beta
$$
for any $s,t \in I$ and $h_s,h_t \in \SL(d,\Kb)$ with $q(s) = h_s \mathsf{K}$ and $q(t) =h_t \mathsf{K}$.
\end{definition}

We remark that this is in fact Kapovich--Leeb--Porti's definition of a ``uniformly regular'' quasi-geodesics, and it is a consequence of the higher-rank Morse lemma \cite[Th.\ 1.1]{KLP2018b} that a quasi-geodesic is Morse if and only if it is uniformly regular. 

\begin{definition}[{\cite[Def.\ 8.1]{KL}}] \ 
\begin{itemize}
\item Let $X$ be a proper geodesic Gromov-hyperbolic metric space. A map $f\colon X \to M$ is a \emph{$\Psf_k$-Morse quasi-isometric embedding} if there exist constants $\alpha,\beta>0$ such that $f$ sends geodesics in $X$ to $\Psf_k$-Morse quasi-geodesics with constants $\alpha,\beta$.

\item Let $(\Gamma,\peripherals)$ be relatively hyperbolic and let $X$ be a weak cusp space for $(\Gamma,\peripherals)$. A representation $\rho\colon \Gamma \to \SL(d,\Kb)$ is \emph{$\Psf_k$-Morse relative to $X$} if there exists a $\rho$-equivariant $\Psf_k$-Morse quasi-isometric embedding of $X$ into $M$.
\end{itemize}
\end{definition}

\begin{proposition} \label{prop:uniformly rel Anosov implies rel Morse}
Suppose that $(\Gamma,\peripherals)$ is relatively hyperbolic and $X$ is a weak cusp space for $(\Gamma,\peripherals)$. If $\rho \colon \Gamma \to \SL(d,\Kb)$ is uniformly $\Psf_k$-Anosov relative to $X$, then $\rho$ is $\Psf_k$-Morse relative to $X$.
\end{proposition}

\begin{proof}
Let $F \colon X \rightarrow M$ be the $\rho$-equivariant quasi-isometry and let $\left\{\bar{A}_\sigma : \sigma \in \Gc(X)\right\}$ be the matrices constructed in Section~\ref{sec:QI of entire weak cusp space}. By Lemma~\ref{lem: qi embedded approximated by one} it suffices to prove: there exist constants $\alpha, \beta > 0$ such that for any $\sigma \in \Gc(X)$ and $t \in \Rb$, 
$$ 
\log \frac{\mu_k}{\mu_{k+1}} \left( \bar{A}_\sigma^{-1} \bar{A}_{\geodflow^t(\sigma)}  \right) \geq \alpha \log \frac{\mu_1}{\mu_d} \left(\bar{A}_\sigma^{-1}  \bar{A}_{\geodflow^t(\sigma)} \right) - \beta.
$$

By Equations~\eqref{eqn: repeating lem:decay of ratio of singular values} and \eqref{eqn: upper bd singular value ratio in uniform proof}, there exist $\alpha_1, \beta_1 > 0$ such that: for any $\sigma \in \Gc(X)$ and $t>0$, 
\begin{equation}
\label{eqn:in the proof that uniform=>Morse}
 \log \frac{\mu_k}{\mu_{k+1}} \left( \bar{A}_\sigma^{-1} \bar{A}_{\geodflow^t(\sigma)}  \right) \geq \alpha_1 \log \frac{\mu_1}{\mu_d} \left(\bar{A}_\sigma^{-1}  \bar{A}_{\geodflow^t(\sigma)} \right) - \beta_1.
\end{equation}

For $\sigma \in \Gc(X)$, let $I(\sigma) \in \Gc(X)$ be the geodesic defined by $I(\sigma)(t)=\sigma(-t)$. Then by Observation~\ref{obs:the A_sigma have local bounds},
$$
\frac{\mu_k}{\mu_{k+1}} \left( \bar{A}_{\sigma}^{-1} \bar{A}_{\geodflow^t(\sigma)} \right) \asymp \frac{\mu_k}{\mu_{k+1}} \left( \bar{A}_{I(\sigma)}^{-1} \bar{A}_{I(\geodflow^t(\sigma))} \right)=\frac{\mu_k}{\mu_{k+1}} \left( \bar{A}_{I(\sigma)}^{-1} \bar{A}_{\geodflow^{-t}(I(\sigma))} \right)
$$
for any $\sigma \in \Gc(X)$ and $t \in \Rb$. So by Equation \eqref{eqn:in the proof that uniform=>Morse}, there exist $\alpha_2, \beta_2 > 0$ such that: for any $\sigma \in \Gc(X)$ and $t<0$, 
$$ 
\log \frac{\mu_k}{\mu_{k+1}} \left(  \bar{A}_{\sigma}^{-1} \bar{A}_{\geodflow^t(\sigma)} \right) \geq \alpha_2 \log \frac{\mu_1}{\mu_d} \left(\bar{A}_\sigma^{-1}  \bar{A}_{\geodflow^t(\sigma)} \right) - \beta_2.
$$

Then $\alpha :=\min\{\alpha_1,\alpha_2\}$ and $\beta:=\max\{\beta_1,\beta_2\}$ suffice. 
\end{proof}

\begin{proposition} \label{prop:rel Morse implies uniformly rel Anosov}
Suppose that $(\Gamma,\peripherals)$ is relatively hyperbolic and $X$ is a weak cusp space for $(\Gamma,\peripherals)$. If $\rho\colon \Gamma \to \SL(d,\Kb)$ is representation and there exists a \textbf{continuous} $\rho$-equivariant $\Psf_k$-Morse quasi-isometric embedding $F \colon X \rightarrow M$, then $\rho$ is uniformly $\Psf_k$-Anosov relative to $X$.
\end{proposition}

\begin{proof} For each $\sigma \in \Gc(X)$, fix $A_\sigma \in \SL(d,\Kb)$ with $F(\sigma(0)) = A_\sigma \mathsf{K}$. Then define a metric on the fibers of $\Gc(X) \times \Kb^d \to \Gc(X)$ by
$$
\norm{A_\sigma(\cdot)}_\sigma = \norm{\cdot}_2.
$$
Notice that $\norm{\cdot}_\sigma$ is continuous in $\sigma$ and $\rho$-equivariant. Since $F$ is an quasi-isometric embedding, for any $r > 0$ the set 
$$
K_r:=\overline{\left\{A_{\sigma_1}^{-1} A_{\sigma_2} : \sigma_1, \sigma_2 \in \Gc(X),\  \dist_X(\sigma_1(0), \sigma_2(0)) \leq r \right\}}
$$ 
is compact in $\SL(d,\Kb)$. Hence $\norm{\cdot}_\sigma$ descends to a locally uniform metric on the vector bundle $\wh{E}_\rho(X) \rightarrow \wh{\Gc}(X)$.

Since $F$ is $\Psf_k$-Morse, there exist constants $C_0, c_0>0$ such that for $\sigma \in \Gc(X)$ and all $m,n \in \Zb$ with $m \geq n$ we have
$$
\frac{\mu_{d-k+1}}{\mu_{d-k}} \left( A_{\geodflow^{m+1}(\sigma)}^{-1} A_{\geodflow^{n}(\sigma)} \right) = \frac{\mu_{k+1}}{\mu_{k}} \left( A_{\geodflow^{n}(\sigma)}^{-1} A_{\geodflow^{m+1}(\sigma)} \right) \leq C_0e^{c_0(m-n+1)} 
$$
From this estimate and the compactness of $K_1$, for any given $\sigma$ the sequence $\left( A_{\geodflow^{n+1}(\sigma)}^{-1} A_{\geodflow^{n}(\sigma)} \right)_{n \in \Zb}$ lies in the compact flow space $\mathcal{D}$ in the hypotheses of \cite[Prop.\ 2.4]{BPS} with $p=d-k$. Hence, by that proposition,
\begin{itemize}
\item for any $\sigma \in \Gc(X)$ the limits 
$$
E^{cs}(\sigma):= \lim_{n \rightarrow \infty} U_k\left( A_\sigma^{-1} A_{\geodflow^{n}(\sigma)}\right) \quad \text{and} \quad E^{cu}(\sigma):= \lim_{n \rightarrow \infty} U_{d-k}\left( A_\sigma^{-1} A_{\geodflow^{-n}(\sigma)}\right)
$$
exist, depend continuously on $\sigma$, and $E^{cs}(\sigma) \oplus E^{cu}(\sigma) = \Kb^d$, 
\item there exist $C_1,c_1>0$ such that 
\begin{equation}
\label{eqn:exp decay in BPS citation}
 \frac{\norm{A_{\geodflow^n(\sigma)}^{-1} A_\sigma Y}_2}{\norm{A_{\geodflow^n(\sigma)}^{-1} A_\sigma Z}_2} \leq C_1 e^{-c_1 n} \frac{\norm{Y}_2}{\norm{Z}_2} 
\end{equation}
for all $\sigma \in \Gc(X)$, $n \in \Nb$, $Y \in E^{cs}(\sigma)$, and non-zero $Z \in E^{cu}(\sigma)$.
\end{itemize}

We claim that $A_\sigma E^{cs}(\sigma)$ only depends on $\sigma^+$. Fix $\sigma_1, \sigma_2 \in \Gc(X)$ with $\sigma^+_1 = \sigma^+_2$. Then 
$$
r:=\sup_{n \in \Nb} \dist_X(\sigma_1(n), \sigma_2(n)) 
$$
is finite. Then, since $K_r$ is compact, the set
$$
\left\{ A_{\geodflow^{n}(\sigma_1)}^{-1} A_{\geodflow^{n}(\sigma_2)} : n \in \Nb\right\} \subset K_r
$$
is relatively compact. So Lemma~\ref{lem: U_k under product} implies that
$$
A_{\sigma_1} E^{cs}(\sigma_1) = \lim_{n \rightarrow \infty} U_k\left( A_{\geodflow^{n}(\sigma_1)}\right)=\lim_{n \rightarrow \infty} U_k\left( A_{\geodflow^{n}(\sigma_2)}\right)=A_{\sigma_2} E^{cs}(\sigma_2).
$$
Thus $A_\sigma E^{cs}(\sigma)$ only depends on $\sigma^+$. 

A similar argument shows that $A_\sigma E^{cu}(\sigma)$ only depends on $\sigma^-$. So there exists a continuous transverse map 
$$
\xi=(\xi^k, \xi^{d-k}) \colon \partial(\Gamma,\peripherals) \to \Gr_k(\Kb^d) \times \Gr_{d-k}(\Kb^d)
$$
such that $\xi^k(\sigma^+) = A_\sigma E^{cs}(\sigma)$ and $\xi^{d-k}(\sigma^-)= A_\sigma E^{cu}(\sigma)$ for all $\sigma \in \Gc(X)$. Further, since $F$ is $\rho$-equivariant so is $\xi$. 

Now fix $\sigma \in \Gc(X)$, $t \geq 0$, $Y \in \xi^k(\sigma^+)$, and non-zero $Z \in \xi^{d-k}(\sigma^-)$. Let $n := \lfloor t \rfloor$ and $\sigma_1 := \geodflow^{t-n}(\sigma)$. Then by Equation~\eqref{eqn:exp decay in BPS citation} and the compactness of $K_1$, 
\begin{align*}
\frac{\norm{Y}_{\geodflow^t(\sigma)}}{\norm{Z}_{\geodflow^t(\sigma)}} & =\frac{\norm{Y}_{\geodflow^n(\sigma_1)}}{\norm{Z}_{\geodflow^n(\sigma_1)}} =  \frac{\norm{A_{\geodflow^n(\sigma_1)}^{-1} A_{\sigma_1} A_{\sigma_1}^{-1}Y}_2}{\norm{A_{\geodflow^n(\sigma_1)}^{-1} A_{\sigma_1} A_{\sigma_1}^{-1} Z}_2} \leq C_1 e^{-c_1 t} \frac{\norm{A_{\sigma_1}^{-1}Y}_2}{\norm{A_{\sigma_1}^{-1} Z}_2} \\
&  \asymp C_1 e^{-c_1 t} \frac{\norm{A_{\sigma}^{-1}Y}_2}{\norm{A_{\sigma}^{-1} Z}_2}=C_1 e^{-c_1 t} \frac{\norm{Y}_{\sigma}}{\norm{Z}_{\sigma}}. 
\end{align*}

So Proposition~\ref{prop: Hom bundles contraction/expansions} implies that $\rho$ is uniformly $\Psf_k$-Anosov relative to $X$.
\end{proof}

\section{Relative stability}\label{sec:stability} 

In this section we prove Theorem~\ref{thm:stability}, which we restate below. Then in Section~\ref{sec:proof of singular value and eigenvalue estimates} and Section~\ref{sec:uniform stability} we establish the stability assertions in Theorems~\ref{thm:singular value and eigenvalue estimates} and~\ref{thm:uniform_stability_case} respectively.

\begin{theorem}\label{thm:stability_in_body}
Suppose that $(\Gamma, \peripherals)$ is relatively hyperbolic and $X$ is a weak cusp space for $(\Gamma, \peripherals)$. If $\rho_0\colon\Gamma \to \SL(d,\Kb)$ is $\Psf_k$-Anosov relative to $X$, then there exists an open neighborhood $\Oc$ of $\rho_0$ in $\Hom_{\rho_0}(\Gamma, \SL(d,\Kb))$ such that every representation in $\Oc$ is $\Psf_k$-Anosov relative to $X$. 

Moreover:
\begin{enumerate}
\item If $\xi_\rho$ is the Anosov boundary map of $\rho \in \Oc$, then the map 
$$
(\rho, x) \in \Oc \times \partial(\Gamma, \peripherals) \mapsto \xi_\rho(x) \in \Gr_k(\Kb^d) \times \Gr_{d-k}(\Kb^d)
$$
is continuous. 
\item If $h\colon M \to \mathcal{O}$ is a real-analytic family of representations and $x \in \partial(\Gamma,\peripherals)$, then the map 
$$
u \in M \mapsto \xi_{h(u)}(x) \in \Gr_k(\Kb^d) \times \Gr_{d-k}(\Kb^d)
$$
is real-analytic.
\end{enumerate}
\end{theorem}

The proof has three main steps. First we set up a flow space to work with, second we verify that the flow space has a dominated splitting, and finally we use the dominated splitting to construct the Anosov boundary maps. The arguments in the first two steps are similar to the proof of stability for relatively Anosov representations of geometrically finite Fuchsian groups in~\cite[Sec.\ 8]{CZZ2021}, but the argument in the third step is different (and more complicated). 

For the rest of the section fix $(\Gamma, \peripherals)$, $X$, and $\rho_0\colon\Gamma \to \SL(d,\Kb)$ as in the statement of Theorem~\ref{thm:stability_in_body}. Since $X$ is fixed for the entire section, for ease of notation we write 
$$
\Gc := \Gc(X), \quad \wh{\Gc} := \wh{\Gc}(X), \quad E:=E(X), \quad \text{and} \quad \wh{E}_\rho := \wh{E}_\rho(X).
$$

Let 
$$
\xi_{\rho_0} = (\xi_{\rho_0}^k, \xi_{\rho_0}^{d-k}) \colon \partial(\Gamma,\peripherals) \to \Gr_k(\Kb^d) \times \Gr_{d-k}(\Kb^d)
$$
denote the Anosov boundary map associated to $\rho_0$. Then let $\wh{E}_{\rho_0} = \wh{\Theta}^k_{\rho_0} \oplus \wh{\Xi}^{d-k}_{\rho_0}$ denote the Anosov splitting induced by $\xi_{\rho_0}$.

\subsection*{Step 1: Setting up the flow space}
\label{sec:stability step 1}

By hypothesis there exist a metric $\norm{\cdot}^{(0)}$ on the vector bundle $\wh{E}_{\rho_0} \to \wh{\Gc}$ and constants $C,c>0$ such that 
\begin{equation}
\label{eqn:initial contraction constants in stability proof}
\norm{\homflow^t(f)}_{\geodflow^t(\sigma)}^{(0)} \leq C e^{-ct} \norm{f}_\sigma^{(0)}
\end{equation}
for all $\sigma \in \wh{\Gc}$, $t \geq 0$, and $f \in \Hom(\wh{\Xi}_{\rho_0}^{d-k}, \wh{\Theta}_{\rho_0}^k)|_{\sigma}$.

Given an open neighborhood $\Oc \subset \Hom_{\rho_0}(\Gamma, \SL(d,\Kb))$, define $E(\Oc) := \Oc \times \Gc \times \Kb^d$ and 
$$
\wh{E}(\Oc) := \Gamma \backslash E(\Oc)
$$
where $\Gamma$ acts by $\gamma \cdot (\rho, \sigma, Y) = (\rho, \gamma \circ \sigma, \rho(\gamma)Y)$. Note that the map $E(\Oc) \rightarrow \Oc \times \Gc$ descends to a vector bundle
$$
\wh{E}(\Oc) \to \Oc \times \wh{\Gc}.
$$
Moreover, 
$$
\wh{E}(\Oc)|_\rho = \bigcup_{\sigma \in \wh{\Gc}(X)} \wh{E}(\Oc)|_{(\rho, \sigma)}
$$ 
naturally identifies with $\wh{E}_\rho$.

The flow $\geodflow^t$ on $\Gc$ extends to a flow $\flatflow^t$ on $E(\Oc)$ by acting trivially in the other factors and descends to a flow also denoted $\flatflow^t$ on $\wh{E}(\Oc)$. 

As before, let $\peripherals^\Gamma := \{ \gamma P \gamma^{-1} : P \in \peripherals, \gamma \in \Gamma\}$. Using the equivalent formulation of relative hyperbolicity given in~\cite[Prop.\ 6.13]{Bowditch_relhyp} there exists a collection of open sets $\{ H_P\}_{P \in \peripherals^\Gamma}$  in $X$ with the following properties: 
\begin{itemize} 
\item $\gamma H_P = H_{\gamma P \gamma^{-1}}$ for all $P \in \peripherals^{\Gamma}$ and $\gamma \in \Gamma$, 
\item $H_P$ accumulates on a single point in $\partial_\infty X$: the fixed point of $P$,
\item $\overline{H}_{P} \cap \overline{H}_Q = \varnothing$ for all distinct $P,Q \in \peripherals^\Gamma$, 
\item $\Gamma$ acts cocompactly on $X \smallsetminus \bigcup_{P \in \peripherals^\Gamma} H_P$. 
\end{itemize} 
(In the case when $X$ is a simply connected negative curved Riemannian manifold, the $H_P$ can be chosen to be horoballs.)

Then let
$$
\Cc := \left\{ P \backslash H_{P} : P \in \peripherals \right\}.
$$ 
Informally, $\Cc$ denotes the set of ``cusps'' of the quotient $\Gamma \backslash X$. 

If $C=P \backslash H_P \in \Cc$, then by shrinking $\Oc$ if necessary, we may assume that there is a continuous map $g_C \colon \Oc \to \SL(d, \Kb)$ such that
$$ 
g_C(\rho) \rho_0(g) g_C(\rho)^{-1} = \rho(g) 
$$
for all $g \in P$ and all $\rho \in \Oc$. Notice, if $\rho \in \Oc$, then the map 
\begin{align*}
\Phi^H_\rho & \colon E_{\rho} |_{\Gc|_H} \to E_{\rho_0} |_{\Gc|_H} \\
\Phi^H_\rho & (\sigma,Y) = (\sigma, g_C(\rho)(Y))
\end{align*} 
is a bundle isomorphism which descends to a bundle isomorphism
$$
\wh\Phi^C_\rho \colon \wh{E}_{\rho_0} |_{\wh\Gc|_C} \to \wh{E}_{\rho_0} |_{\wh\Gc|_C}.
$$
Moreover, if $\flatflow^s(Y) \in \wh{E}_{\rho_0} |_{\wh\Gc|_C}$ for all $s \in [0,t]$, then
\begin{equation} \label{eqn:cusp_flow_conjugacy}
\wh\Phi^C_\rho(\flatflow^t(Y)) = \flatflow^t(\wh\Phi^C_\rho(Y)).
\end{equation}

We use these isomorphisms on the cusps to extend the Anosov splitting $\wh{E}_{\rho_0} = \wh\Theta^k_{\rho_0} \oplus \wh\Xi^{d-k}_{\rho_0}$ to a (not necessarily flow-invariant) splitting 
$$
\wh{E}(\Oc) = \wh{F}^k \oplus \wh{G}^{d-k}
$$
by first defining
\begin{align*}
\wh{F}^k |_{(\rho_0,\sigma)} = \wh\Theta^k_{\rho_0}|_{\sigma} \quad \text{and} \quad \wh{G}^{d-k} |_{(\rho_0,\sigma)} = \wh\Xi^{d-k}_{\rho_0}|_{\sigma}
\end{align*}
for all $\sigma \in \wh{\Gc}$, then defining 
\begin{align*}
\wh{F}^k |_{(\rho,\sigma)} = \wh\Phi^C_\rho \left( \wh\Theta^k_{\rho_0} |_\sigma \right) \quad \text{and} \quad \wh{G}^{d-k} |_{(\rho,\sigma)} = \wh\Phi^C_\rho \left( \wh\Xi^{d-k}_{\rho_0} |_\sigma \right)
\end{align*}
for all $\rho \in \Oc$, $C \in \Cc$, and $\sigma \in \wh{\Gc}|_C$, and finally extending this splitting globally (further shrinking $\Oc$ and each $C \in \Cc$ if necessary).

By further shrinking $\Oc$ and each $C \in \Cc$, we may also fix a metric on the vector bundle $\wh{E}(\Oc) \to \Oc \times \wh{\Gc}$ such that: 
\begin{align}
\norm{\cdot}_{(\rho_0,\sigma)} = \norm{\cdot}^{(0)}_\sigma & \quad\quad \text{for all } \sigma \in \wh{\Gc}  \nonumber \\
 \norm{\Phi^C_\rho(\cdot)}_{(\rho, \sigma)} = \norm{\cdot}^{(0)}_\sigma & \quad\quad \text{for all } \rho \in \Oc, \ C \in \Cc, \text{ and }  \sigma \in \wh{\Gc}|_{C}. \label{eqn:defn of norm on thin parts}
\end{align}

\subsection*{Step 2: Using the contraction mapping theorem to obtain a $\flatflow^t$-invariant global splitting.}\label{sec:stability step 2}

Relative to the initial splitting $\wh{E}(\Oc) = \wh{F}^k \oplus \wh{G}^{d-k}$, we may decompose the flow 
$$
\flatflow^t = \begin{pmatrix} A_t & B_t \\ C_t & D_t \end{pmatrix}.
$$
Notice that the splitting is invariant if $B_t \equiv C_t \equiv 0$.
Consider the bundle 
$$
\Hom( \wh{G}^{d-k}, \wh{F}^k) \to \Oc \times \wh\Gc
$$ 
with the operator norm induced from our metric on $\wh{E}(\Oc)$, and let $\Rc_r \subset \Hom( \wh{G}^{d-k}, \wh{F}^k)$ denote the $r$-ball bundle about the zero section.

Fix $\epsilon \in (0, \frac12)$ so that
\begin{equation}
\label{eqn:epsilon is super small}
 \frac{1+2\epsilon}{1-2\epsilon} \leq 2, \quad \frac{1}{1-\epsilon} + \frac{\epsilon(1+\epsilon)}{(1-\epsilon)^2} \leq 2, \quad \text{and} \quad \epsilon\frac{(1+2\epsilon)^2(1+2\epsilon^2)}{(1-2\epsilon)^2(1-2\epsilon^2)} \leq \frac{1}{2}.
\end{equation}
Using these bounds on $\epsilon$, the proof of \cite[Prop.\ 8.3]{CZZ2021} yields the following. 

\begin{proposition}\label{prop:CZZ83} After possibly replacing $\Oc$ by a smaller neighborhood of $\rho_0$, there exists $T > 0$ such that: for all $t \in [T,2T]$ there is a well-defined map $\homflow^t\colon \Rc_1 \to \Rc_{2 \epsilon}$ given by
$$ \homflow^t(f) = (B_t + A_t f) (D_t + C_t f)^{-1}. 
$$
Furthermore, 
$$
\norm{\homflow^t(f_1) - \homflow^t(f_2)} \leq 2\epsilon \norm{f_1-f_2}
$$ 
for all $\rho \in \Oc$, $\sigma\in\wh\Gc$, and $f_1,f_2 \in \Rc_1|_{(\rho,\sigma)}$.
\end{proposition}

\begin{remark}\label{remark:graph versus section}
	One can verify that the map $\homflow^t$ has the property that $\mathrm{Graph}(\homflow^t(f)) = \flatflow^t(\mathrm{Graph}(f))$ for all $t \in [T,2T]$, $\rho \in \Oc$, $\sigma\in\wh\Gc$, and $f \in \Rc_1|_{(\rho,\sigma)}$. Further, if the splitting is flow-invariant, then $\homflow^t$ coincides with the map 
	$$
	f \mapsto \flatflow^t \circ f \circ \flatflow^{-t}.
	$$
\end{remark}

Let $S(\Rc_r)$ be the space of continuous sections of the fiber bundle $\Rc_r \to \Oc \times \hat{\Gc}$. This is a complete metric space with the distance 
$$
\dist_S(\mathscr{s}_1, \mathscr{s}_2) = \sup_{(\rho, \sigma)} \norm{ \mathscr{s}_1(\rho,\sigma)- \mathscr{s}_2(\rho,\sigma)}.
$$
Also, when $t \in [T, 2T]$, $\homflow^t$ induces a map $\homflow^t_S\colon S(\Rc_1) \to S(\Rc_{2\epsilon})$ given by
$$ 
\homflow^t_S(\mathscr{s})(\rho,\sigma) := \homflow^t (\mathscr{s} (\rho, \flatflow^{-t}(\sigma))) .
$$
By Proposition \ref{prop:CZZ83}, this is a contraction mapping for each $t \in [T,2T]$. Hence for each $t$ in this range,  there exists a unique $\homflow^t_S$-invariant section $\mathscr{s}^{(t)}_0$ of the bundle $\Rc_{2\epsilon}$. 

Arguing as in \cite[pp.\ 43-44]{CZZ2021}, the section  $\mathscr{s}^{(t)}_0$ does not depend on $t \in [T,2T]$. Then 
$$ 
\wh\Xi^{d-k}|_{(\rho,\sigma)} := \mathrm{Graph}\left( \mathscr{s}^{(t)}_0(\rho,\sigma) \right)
$$
defines a flow-invariant $(d-k)$-dimensional subbundle $\wh\Xi^{d-k}$ of $\wh{E}(\Oc)$, see Remark~\ref{remark:graph versus section}.

Applying the same arguments to $\Hom(\wh{F}^k, \wh{G}^{d-k})$ (although flowing in the other direction) and further shrinking $\Oc$ if needed, we obtain a flow-invariant $k$-dimensional subbundle $\wh\Theta^k$ of $\wh{E}(\Oc)$.

Arguing as in \cite[p.\ 44]{CZZ2021}, the fibers $\wh\Xi^{d-k} |_{(\rho,\sigma)}$ and $\wh\Theta^k |_{(\rho,\sigma)}$ are transverse for every $(\rho,v) \in \Oc \times \wh\Gc$. Hence, by dimension counting, we obtain a flow-invariant splitting 
$$
\wh{E}(\Oc) = \wh\Theta^k \oplus \wh\Xi^{d-k}.
$$

The next two lemmas verify that this is a dominated splitting. 

\begin{lemma}\label{lem:bounded distortion on short time periods}
There exist $C_0 > 0$ such that: if $t \in [0,T]$, $\rho \in \Oc$, $\sigma \in \wh\Gc$, $Y \in \wh\Theta^{k}|_{(\rho,\sigma)}$, and $Z \in \wh\Xi^{d-k}|_{(\rho,\sigma)}$ is non-zero, then 
\begin{equation*}
\frac{\norm{\flatflow^t(Y)}_{ (\rho,\geodflow^t(\sigma))}}{\norm{\flatflow^t(Z)}_{(\rho, \geodflow^t(\sigma))}} \leq C_0 \frac{\norm{Y}_{(\rho,\sigma)}}{\norm{Z}_{(\rho,\sigma)}}.
\end{equation*}
\end{lemma} 

\begin{proof} Let 
$$
\wh{\Gc}_{thinner} := \left\{ \sigma \in \wh{\Gc}: \sigma(t) \in \bigcup_{C \in \Cc} C \text{ for all } t \in [0,T] \right\}
$$
and 
$$
\wh{\Gc}_{thicker} := \wh{\Gc} \smallsetminus \wh{\Gc}_{thinner}. 
$$
Notice that $\wh{\Gc}_{thicker}$ is compact and so there exists $C_0^\prime > 0$ such that 
\begin{equation*}
\frac{\norm{\flatflow^t(Y)}_{ (\rho,\geodflow^t(\sigma))}}{\norm{\flatflow^t(Z)}_{(\rho, \geodflow^t(\sigma))}} \leq C_0^\prime \frac{\norm{Y}_{(\rho,\sigma)}}{\norm{Z}_{(\rho,\sigma)}}
\end{equation*}
for all $t \in [0,T]$, $\rho \in \Oc$, $\sigma \in \wh{\Gc}_{thicker}$, $Y \in \wh\Theta^{k}|_{(\rho,\sigma)}$, and non-zero $Z \in \wh\Xi^{d-k}|_{(\rho,\sigma)}$.

Suppose that $t \in [0,T]$, $\rho \in \Oc$, $\sigma \in \wh{\Gc}_{thinner}$, $Y \in \wh\Theta^{k}|_{(\rho,\sigma)}$, and $Z \in \wh\Xi^{d-k}|_{(\rho,\sigma)}$ is non-zero. Let 
 $$
 Y=Y_1 + Y_2 \quad \text{and} \quad Z=Z_1 + Z_2
 $$
 be the decomposition relative to $\wh{E}(\Oc) = \wh{F}^{k} \oplus \wh{G}^{d-k}$. Then, by the construction of $\wh{\Theta}^k$ and $\wh{\Xi}^{d-k}$, we have $\norm{Y_2}_{(\rho,\sigma)} \leq 2 \epsilon \norm{Y_1}_{(\rho,\sigma)}$ and $\norm{Z_1}_{(\rho,\sigma)} \leq 2\epsilon \norm{Z_2}_{(\rho,\sigma)}$. Further, by Equation~\eqref{eqn:cusp_flow_conjugacy},
 \begin{align*}
\flatflow^t(Y_1) \in \wh F^k|_{(\rho,\sigma)} \quad \text{and} \quad \flatflow^t(Y_2) \in \wh G^{d-k}|_{(\rho,\sigma)}.
 \end{align*}
Then, since $\flatflow^t(Y)=\flatflow^t(Y_1) +\flatflow^t(Y_2)  \in \wh{\Theta}^k|_{(\rho, \phi_t(v)))}$, we have 
 $$
\norm{\flatflow^t(Y_2) }_{(\rho, \phi^t(\sigma))} \leq 2\epsilon \norm{ \flatflow^t(Y_1)}_{(\rho, \phi^t(\sigma))}
$$
and hence
 $$
\norm{\flatflow^t(Y) }_{(\rho, \phi^t(\sigma))} \leq (1+2\epsilon) \norm{ \flatflow^t(Y_1)}_{(\rho, \phi^t(\sigma))}.
$$
Similar reasoning shows that 
\begin{align*}
\norm{\flatflow^t(Z)}_{(\rho, \phi^t(\sigma))} & \geq  (1-2\epsilon) \norm{\flatflow^t(Z_2)}_{(\rho, \phi^t(\sigma))}.
\end{align*}
Then by Equations~\eqref{eqn:initial contraction constants in stability proof}, ~\eqref{eqn:cusp_flow_conjugacy}, and~\eqref{eqn:epsilon is super small} we have 
\begin{align*}
\frac{\norm{\flatflow^t(Y)}_{ (\rho,\geodflow^t(\sigma))}}{\norm{\flatflow^t(Z)}_{(\rho, \geodflow^t(\sigma))}}& \leq \frac{1+2\epsilon}{1-2\epsilon} \frac{\norm{ \flatflow^t(Y_1)}_{(\rho, \phi^t(\sigma))}}{\norm{ \flatflow^t(Z_2)}_{(\rho, \phi^t(\sigma))}} \leq \frac{1+2\epsilon}{1-2\epsilon} Ce^{-ct} \frac{\norm{ Y_1}_{(\rho, \sigma)}}{\norm{ Z_2}_{(\rho, \sigma)}}\\
& \leq \frac{(1+2\epsilon)^2}{(1-2\epsilon)^2} Ce^{-ct} \frac{\norm{ Y}_{(\rho, \sigma)}}{\norm{ Z}_{(\rho, \sigma)}} \leq 4C \frac{\norm{ Y}_{(\rho, \sigma)}}{\norm{ Z}_{(\rho, \sigma)}}.
\end{align*}

So $C_0:=\max\{ C_0^\prime, 4C\}$ suffices. 
\end{proof} 

\begin{lemma}\label{lem:contraction on hom bundle in stability} 
There exist $C_1, c_1 > 0$ such that: if $t \geq 0$, $\rho \in \Oc$, $\sigma \in \wh\Gc$, $Y \in \wh\Theta^{k}|_{(\rho,\sigma)}$, and $Z \in \wh\Xi^{d-k}|_{(\rho,\sigma)}$ is non-zero, then 
\begin{equation*}
\frac{\norm{\flatflow^t(Y)}_{ (\rho,\geodflow^t(\sigma))}}{\norm{\flatflow^t(Z)}_{(\rho, \geodflow^t(\sigma))}} \leq C_1e^{-c_1 t} \frac{\norm{Y}_{(\rho,\sigma)}}{\norm{Z}_{(\rho,\sigma)}}.
\end{equation*}
\end{lemma}

\begin{proof} Using the estimates in Equation~\eqref{eqn:epsilon is super small}, the proof of \cite[Prop.\ 8.5]{CZZ2021} implies that 
\begin{equation*}
\frac{\norm{\flatflow^t(Y)}_{ (\rho,\geodflow^t(\sigma))}}{\norm{\flatflow^t(Z)}_{(\rho, \geodflow^t(\sigma))}} \leq \frac{1}{2} \frac{\norm{Y}_{(\rho,\sigma)}}{\norm{Z}_{(\rho,\sigma)}}
\end{equation*}
for all $t \in [T,2T]$, $\rho \in \Oc$, $\sigma \in \wh\Gc$, $Y \in \wh\Theta^{k}|_{(\rho,\sigma)}$, and non-zero $Z \in \wh\Xi^{d-k}|_{(\rho,\sigma)}$. 

Then, by repeatedly using the above estimate and Lemma~\ref{lem:bounded distortion on short time periods}, the lemma holds with 
\begin{equation*}
c_1 := -\frac{\log(2)}{T} \quad \text{and} \quad C_1 := \max\{ 1, C_0 \} \cdot e^{c_1 T} . \qedhere
\end{equation*}
 \end{proof}

\subsection*{Step 3: Finding the  Anosov boundary maps from this flow-invariant splitting.}

To complete the proof of Theorem~\ref{thm:stability_in_body} we need to show that each splitting 
\begin{align}
\label{eqn:splitting downstairs}
\wh{E}_\rho = \wh{E}(\Oc)|_{\rho} = \wh\Theta^k|_\rho \oplus \wh\Xi^{d-k}|_{\rho}
\end{align}
arises from a boundary map. To ease notation, for the rest of this step we fix some $\rho \in \Oc$. We lift the splitting in Equation~\eqref{eqn:splitting downstairs} to 
$$
\Gc \times \Kb^d = E = \Theta^k_\rho \oplus \Xi^{d-k}_\rho.
$$
We also lift the metric on $\wh{E}_\rho = \wh{E}(\Oc)|_{\rho}$ to a metric on $E \to \Gc$. Then Lemma~\ref{lem:contraction on hom bundle in stability} implies that
\begin{equation}
\label{eqn:contraction upstairs for rho}
\frac{\norm{Y}_{ \geodflow^t(\sigma)}}{\norm{Z}_{\geodflow^t(\sigma)}} \leq C_1e^{-c_1 t} \frac{\norm{Y}_{\sigma}}{\norm{Z}_{\sigma}}
\end{equation}
for all $t \geq 0$, $\sigma \in \Gc$, $Y \in \Theta^{k}_\rho(\sigma)$, and non-zero $Z \in \Xi^{d-k}_\rho(\sigma)$.

Since each $\norm{\cdot}_\sigma$ is induced by an inner product, for each $\sigma \in \Gc$ there exists a matrix $A_\sigma \in \GL(d,\Kb)$ such that 
$$
\norm{A_\sigma(\cdot)}_{\sigma} =\norm{\cdot}_2.
$$ 

\begin{lemma}\label{lem: gap between singular values in step 3} For any $\sigma \in \Gc$, 
$$
\lim_{t \to \infty} \frac{\mu_{k+1}}{\mu_{k}}\left(A_\sigma^{-1}A_{\geodflow^{t}(\sigma)}\right) =0.
$$
\end{lemma}

\begin{proof} This is exactly the same as the proof of Lemma~\ref{lem:decay of ratio of singular values}. 

\end{proof} 

\begin{lemma}\label{lem: boundary maps as a limit in Step 3} If $\sigma \in \Gc$, then $\Theta^{k}_\rho(\sigma) = \lim_{t \to \infty} U_{k}\left(A_{\geodflow^{t}(\sigma)}\right)$. \end{lemma}

\begin{proof} Fix $\sigma \in \Gc$ and suppose not. Then there exists  $t_n \to \infty$ where
$$
V: = \lim_{n \to \infty} U_{k}\left(A_{\geodflow^{t_n}(\sigma)}\right) \in \Gr_{k}(\Kb^d)
$$
exists and does not equal $\Theta^{k}_\rho(\sigma)$. Fix $Y \in \Theta_\rho^{k}(\sigma) \smallsetminus V$ non-zero. Then 
$$
\norm{A_{\geodflow^{t_n}(\sigma)}^{-1} Y}_2 \gtrsim \frac{1}{\mu_{k+1}\left(A_{\geodflow^{t_n}(\sigma)}\right)}\norm{Y}_2 = \mu_{d-k}\left(A_{\geodflow^{t_n}(\sigma)}^{-1}\right)\norm{Y}_2.
$$
Also, by the max-min/min-max theorem for singular values, for every $n$ there exists a non-zero $Z_n \in \Xi^{d-k}_\rho(\sigma)$ such that 
$$
\norm{A_{\geodflow^{t_n} (\sigma)}^{-1} Z_n}_2 \leq \mu_{d-k}\left(A_{\geodflow^{t_n} (\sigma)}^{-1}\right) \norm{Z_n}_2.
$$
Then by Equation~\eqref{eqn:contraction upstairs for rho}
$$
0 = \lim_{n \to \infty} \frac{\norm{Y}_{\geodflow^{t_n}(\sigma)}}{\norm{Z_n}_{\geodflow^{t_n}(\sigma)}}\frac{\norm{Z_n}_{\sigma}}{\norm{Y}_{\sigma}} \asymp \lim_{n \to \infty} \frac{\norm{A_{\geodflow^{t_n}(\sigma)}^{-1}Y}_{2}}{\norm{A_{\geodflow^{t_n}(\sigma)}^{-1}Z_n}_{2}}\frac{\norm{Z_n}_{2}}{\norm{Y}_{2}} \gtrsim 1. 
$$
(notice the implicit constants depend on $\sigma$) and we have a contradiction. 
\end{proof}

\begin{lemma} If $\sigma \in \Gc$, then $\Theta_\rho^{k}(\sigma)$ only depends on $\sigma^+$. \end{lemma}

\begin{proof} It suffices to consider the following two cases.

\medskip

\noindent \textbf{Case 1:} Assume $\sigma^+$ is a conical limit point. Then there exist a sequence $(\gamma_n)_{n \geq 1}$ in $\Gamma$ and $t_n \to \infty$ such that $\{\gamma_n\geodflow^{t_n}(\sigma) : n \geq 1\}$ is relatively compact in $\Gc$. 

Fix $\eta \in \Gc$ with $\eta^+ = \sigma^+$. Then  
\begin{align*}
\sup_{t \geq 0} \dist_{X}(\sigma(t), \eta(t)) < +\infty
\end{align*}
and so $\{\gamma_n\geodflow^{t_n}(\eta) : n \geq 1\}$ is also relatively compact in $\Gc$. So for every non-zero $Y \in \Kb^d$ we have
$$
\frac{\norm{A_{\geodflow^{t_n}(\sigma)}^{-1} Y}_2}{\norm{A_{\geodflow^{t_n} (\eta)}^{-1} Y}_2} =  \frac{\norm{Y}_{\geodflow^{t_n}(\sigma)}}{\norm{Y}_{\geodflow^{t_n}(\eta)}
} = \frac{\norm{\rho(\gamma_n)Y}_{\gamma_n\geodflow^{t_n}(\sigma)}}{\norm{\rho(\gamma_n)Y}_{\gamma_n\geodflow^{t_n}(\eta)}
} \asymp \frac{\norm{\rho(\gamma_n)Y}_{2}}{\norm{\rho(\gamma_n)Y}_{2}} = 1$$
where the implicit constants are independent of $n$. So
$$
\frac{ \norm{A_{\geodflow^{t_n}(\sigma)}^{-1}A_{\geodflow^{t_n}(\eta)} Y}_2}{\norm{Y}_2}  = \frac{\norm{A_{\geodflow^{t_n}(\sigma)}^{-1} (A_{\geodflow^{t_n}(\eta)} Y)}_2}{\norm{A_{\geodflow^{t_n} (\eta)}^{-1} (A_{\geodflow^{t_n}(\eta)} Y)}_2} \asymp 1
$$
and hence 
$$
\mu_j\left( A_{\geodflow^{t_n}(\sigma)}^{-1}A_{\geodflow^{t_n}(\eta)}\right) \asymp 1
$$
for all $1 \leq j \leq d$ and all $n$. Thus by Lemma \ref{lem: U_k under product}
$$
\dist\Big(U_{k}\left(A_{\geodflow^{t_n}(\sigma)}\right), U_{k}\left(A_{\geodflow^{t_n}(\eta)}\right)\Big) \lesssim \frac{\mu_{k+1}}{\mu_{k}}\left(A_{\geodflow^{-t_n}(\sigma)}\right).
$$
Hence by Lemma~\ref{lem: gap between singular values in step 3}
$$
\Theta_\rho^{k}(\sigma) = \lim_{n \to \infty} U_{k}\left(A_{\geodflow^{t_n}(\sigma)}\right)=\lim_{n \to \infty} U_{k}\left(A_{\geodflow^{t_n}(\eta)}\right)=\Theta_\rho^{k}(\eta).
$$

\noindent \textbf{Case 2:} Assume $\sigma^+$ is not a conical limit point. Then $\sigma^+$ is the fixed point of a subgroup $P \in \Pc^{\Gamma}$. By definition $\rho(P) = g\rho_0(P)g^{-1}$ for some $g \in \GL(d,\Kb)$. We claim that 
$$ 
\Theta_{\rho}^{k}(\sigma) = \lim_{t\to\infty} U_{k}\left(A_{\geodflow^{t}(\sigma)}\right) = g \xi^{k}_{\rho_0}(\sigma^+)
$$
where $\xi_{\rho_0}$ is the boundary map of $\rho_0$. Let $\norm{\cdot}^{(0)}_{\eta \in \Gc}$ denote the lift of our initial metric on $\wh{E}_{\rho_0}$. Then for each $\eta \in \Gc$ fix $A_{\eta}^{(0)} \in \GL(d,\Kb)$ satisfying 
$$
\norm{A_\eta^{(0)}(\cdot)}_{\eta}^{(0)} = \norm{\cdot}_2.
$$ 
Since $\sigma^+$ is not a conical limit point and $\Gamma$ acts cocompactly on $X \smallsetminus \bigcup_{P \in \Pc^{\Gamma}} H_P$, we must have $\sigma(t) \in H_P$ when $t$ is sufficiently large. Hence by Equation~\eqref{eqn:defn of norm on thin parts} we can assume that $A_{\geodflow^{t}(\sigma)}^{(0)}=g^{-1}A_{\geodflow^{t}(\sigma)} g$ for $t$ sufficiently large. 

By Lemma~\ref{lem: boundary maps as a limit in Step 3} applied to $\rho_0$ we have 
$$
\xi^{k}_{\rho_0}(\sigma^+) = \Theta_{\rho_0}^{k}(\sigma) = \lim_{t\to\infty} U_{k}\left(A^{(0)}_{\geodflow^{t}(\sigma)}\right).
$$
So by Lemma \ref{lem: U_k under product}
$$ 
\Theta_{\rho}^{k}(\sigma) = \lim_{t\to\infty} U_{k}\left(A_{\geodflow^{t}(\sigma)}\right) =  \lim_{t\to\infty} g U_{k}\left(A^{(0)}_{\geodflow^{t}(\sigma)}\right) = g \xi^{k}_{\rho_0}(\sigma^+)
$$
and hence $\Theta_{\rho}^{k}(\sigma)$ only depends on $\sigma^+$. 
\end{proof}

Repeating the arguments in the last three lemmas, but switching the roles of $k$ and $d-k$ implies the following. 

\begin{lemma} If $\sigma \in \Gc$, then $\Xi_\rho^{d-k}(\sigma)$ only depends on $\sigma^-$. \end{lemma}

Finally we can define continuous transverse $\rho$-equivariant maps 
$$
\xi_\rho = (\xi_\rho^k, \xi_\rho^{d-k}) \colon \partial_\infty X \to \Gr_k(\Kb^d) \times \Gr_{d-k}(\Kb^d)
$$
such that $\Theta^k_\rho(\sigma) = \xi_\rho^k(\sigma^+)$ and $\Xi^{d-k}_\rho(\sigma) = \xi_\rho^{d-k}(\sigma^-)$. This combined with Lemma~\ref{lem:contraction on hom bundle in stability} proves that $\rho$ is $\Psf_k$-Anosov relative to $X$. Since $\rho \in \Oc$ was arbitrary this completes the proof of the main assertion in Theorem~\ref{thm:stability_in_body}.

\subsection{The ``moreover'' parts of Theorem~\ref{thm:stability_in_body}}

By construction, the subspaces $\Theta^k|_{(\rho,\sigma)}$ and $\Xi^{d-k}|_{(\rho,\sigma)}$ depend continuously on $(\rho, \sigma) \in \Oc \times \Gc$. Hence the map 
$$
(\rho, x) \in \Oc \times \partial(\Gamma, \peripherals) \mapsto \xi_\rho(x) \in \Gr_k(\Kb^d) \times \Gr_{d-k}(\Kb^d)
$$
is continuous. 

The second part is slightly more involved. 

\begin{proposition}
If $h\colon M \to \mathcal{O}$ is a real-analytic family of representation and $x \in \partial(\Gamma,\peripherals)$, then the map 
$$
u \in M \mapsto \xi_{h(u)}(x) \in \Gr_k(\Kb^d) \times \Gr_{d-k}(\Kb^d)
$$
is real-analytic. 
\end{proposition}

\begin{proof}
This closely follows the proof of the analogous statement in \cite[pp.\ 50-51]{CZZ2021} for geometrically finite Fuchsian groups. 

Using the inclusion $\SL(d,\Cb) \hookrightarrow \SL(2d,\Rb)$, we may assume that $\Kb=\Rb$. Fix a finite generating set $S \subset \Gamma$ and let $N := \abs{S}$. Then we may view $\Hom(\Gamma,\SL(d,\Rb))$ as a Zariski-closed subset of $\SL(d,\Rb)^N$  and consider $h$ as a real-analytic map from $M$ to $\SL(d,\Rb)^N$.

We can then realize $M$ as a totally real submanifold of a complex manifold $M^{\Cb}$ and assume that $h$ extends to a complex analytic map $h\colon M^{\Cb} \to \SL(d,\Cb)^N$. Note that $h(M)$ and $h(M^{\Cb})$ have the same Zariski closure in $\SL(d,\Cb)^N$. In particular, 
$$
h(M^{\Cb}) \subset \Hom(\Gamma, \SL(d,\Cb)).
$$

We claim that $\Hom_{\rho_0}(\Gamma, \SL(d,\Cb))$ is locally closed in the Zariski topology (i.e.\ open in its closure). For $P \in \peripherals$, let 
$$
U_P := \{ \tau \in \Hom(P, \SL(d,\Cb)) : \tau \text{ is conjugate to } \rho_0|_P\}.
$$
Then, $U_P$ is the orbit of $\rho_0|_P$ under the conjugation action of $\SL(d,\Cb)$ and hence is locally closed in the Zariski topology. Next define 
\begin{align*}
f &\colon  \Hom(\Gamma, \SL(d,\Cb)) \to \prod_{P \in \peripherals} \Hom(P, \SL(d,\Cb))
\end{align*}
by $f(\rho) = (\rho|_P)_{P \in \peripherals}$. Then, by definition, 
$$
\Hom_{\rho_0}(\Gamma, \SL(d,\Cb)) = f^{-1}\left( \prod_{P \in \peripherals} U_P \right)
$$
and so $\Hom_{\rho_0}(\Gamma, \SL(d,\Cb))$ is also locally closed in the Zariski topology.

Then, after possibly shrinking $M^{\Cb}$, we may assume that 
$$
h(M^{\Cb}) \subset \Hom_{\rho_0}(\Gamma, \SL(d,\Cb)).
$$
Then, by possibly shrinking again and using the first part of Theorem~\ref{thm:stability}, we may assume that  every representation in $h(M^{\Cb})$ is $\Psf_k$-Anosov. Hence it suffices to show that for any $x \in \partial(\Gamma,\peripherals)$, the map $M^{\Cb} \to \Gr_k(\Cb^d)$ given by $u \mapsto \xi^k_{h(u)}(x)$ is complex analytic in $x$, this implies that the restriction of this map to $M$ is real-analytic.

If $\gamma$ is a hyperbolic element and $\rho \in h(M^{\Cb})$, then Proposition~\ref{prop:eigenvalue data in rel Anosov repn} implies that $\rho(\gamma)$ is $\Psf_k$-proximal and $\xi_\rho^k(\gamma^+)$ is the attracting $k$-plane of $\rho(\gamma)$. It then follows that the map $M^{\Cb} \to \Gr_k(\Cb^d)$ given by $u \mapsto \xi_{h(u)}^k(\gamma^+)$ is complex analytic. More generally, if $x \in \partial(\Gamma,\peripherals)$, there exists a sequence $(\gamma_n)_{n \geq 1}$ of hyperbolic elements of $\Gamma$ such that $\gamma_n^+ \to x$. Then, since the map from $M^{\Cb} \times \partial(\Gamma,\peripherals) \to \Gr_k(\Cb^d)$ given by $(u,y) \mapsto \xi^k_{h(u)}(y)$ is continuous, $u \mapsto \xi_{h(u)}^k(x)$ is a locally uniform limit of complex analytic functions, and hence complex analytic.
\end{proof}

\subsection{Stability in the context of Theorem~\ref{thm:singular value and eigenvalue estimates}}\label{sec:proof of singular value and eigenvalue estimates}

Suppose that $X = \Cc_{GM}(\Gamma, \peripherals, S)$ is a Groves--Manning cusp space and fix $x_0 \in X$. Also fix a neighborhood $\Oc^\prime \subset \Oc$ of $\rho_0$ which is relatively compact in $\Oc$. 

Notice that the metric $\norm{\cdot}_{(\rho,\sigma)}$ on $\wh{E}(\Oc) \to \wh{\Gc}$ constructed in \hyperref[sec:stability step 1]{Step 1} is continuous and the contraction constants for the flow can be chosen to be independent of $\rho \in \Oc$. So by the explicit constants in Lemma~\ref{lem:D-} there exist $\alpha_1, \beta_1 > 0$ such that 
\begin{equation}
\label{eqn:lower bound in GM moreover stuff}
-\beta_1+\alpha_1 d_X(\gamma(x_0), x_0) \leq \log \frac{\mu_k}{\mu_{k+1}}(\rho(\gamma))
\end{equation}
 for all $\rho \in \Oc^\prime$ and $\gamma \in \Gamma$. Then we also have 
\begin{align*}
\alpha_1 \ell_X(\gamma) &= \lim_{n \to \infty} \frac{\alpha_1 d_X(\gamma^n(x_0), x_0)}{n}   \leq \lim_{n \to \infty} \frac{1}{n} \log \frac{\mu_k}{\mu_{k+1}}(\rho(\gamma)^n)=\log \frac{\lambda_k}{\lambda_{k+1}}(\rho(\gamma))
\end{align*}
 for all $\rho \in \Oc^\prime$ and $\gamma \in \Gamma$. 
 
By Lemmas~\ref{lem:structure of weakly unipotent} and~\ref{lem: growth estimates in weakly unipotent groups} there exists $\hat{\beta}_0 > 0$ such that: if $P \in \peripherals$ and $g \in P$, then 
\begin{equation*}
\log \frac{\mu_1}{\mu_d}\left( \rho_0(g) \right) \leq 2(d-1) \log \abs{g}_{S \cap P} + \hat{\beta}_0. 
\end{equation*}
Then, since $\rho|_P$ is conjugate to $\rho_0|_P$ for any $\rho \in \Oc$ and $P \in \peripherals$, there exists $\hat{\beta} > 0$ such that: if $\rho \in \Oc^\prime$, $P \in \peripherals$, and $g \in P$, then 
\begin{equation*}
\log \frac{\mu_1}{\mu_d}\left( \rho(g) \right) \leq 2(d-1) \log \abs{g}_{S \cap P} + \hat{\beta}. 
\end{equation*}
So by Lemma~\ref{lem:GM upper bound on mu1/mud}, see Equation~\eqref{eqn:constant in lem:GM upper bound on mu1/mud}, there exist $\alpha_2,\beta_2 > 0$ such that 
\begin{equation}
\label{eqn:upper bound in GM moreover stuff}
\log \frac{\mu_1}{\mu_{d}}(\rho(\gamma)) \leq \alpha_2 d_X(\gamma(x_0), x_0) +\beta_2
\end{equation}
 for all $\rho \in \Oc^\prime$ and $\gamma \in \Gamma$.  
 
Then by Equations~\eqref{eqn:symmetric distance in prelims}, \eqref{eqn:lower bound in GM moreover stuff}, and \eqref{eqn:upper bound in GM moreover stuff}: For any $p_0 \in \SL(d,\Kb)/ \SU(d,\Kb)$ and $\rho \in \Oc^\prime$ the orbits $\Gamma(x_0)$ and $\rho(\Gamma)(p_0)$ are quasi-isometric with the quasi-isometry constants independent of  $\rho \in \Oc^\prime$.

\subsection{Stability in the context of Theorem~\ref{thm:uniform_stability_case}}\label{sec:uniform stability} Now suppose that the initial family of norms $\norm{\cdot}^{(0)}_\sigma$ on $\wh{E}_{\rho_0} \to \wh{\Gc}$ is locally uniform. Then the family of norms $\norm{\cdot}_{\rho,\sigma}$ constructed in \hyperref[sec:stability step 1]{Step 1} restricts to a locally uniform family of norms on each $\wh{E}_{\rho}=\wh{E}(\Oc)|_{\rho}$. Thus any representation $\rho \in \Oc$ is uniformly $\Psf_k$-Anosov relative to $X$. 
Fix a neighborhood $\Oc^\prime \subset \Oc$ of $\rho_0$ which is relatively compact in $\Oc$. By the construction of the norms in \hyperref[sec:stability step 1]{Step 1} and the estimates in~\hyperref[sec:stability step 2]{Step 2},  we can choose the contraction constants $C$ and $c$ and the locally uniform metric constants $(L_r)_{r>0}$  implicit in Definition~\ref{defn:locally uniform norms}, as well as the constant $C_{K}$  appearing in Section~\ref{subsec:uniform Holder reg} to be independent of $\rho \in \Oc^\prime$. 

By Lemma~\ref{lem:unform qiembed entire space}, for any $\rho \in \Oc^\prime$ we have a $\rho$-equivariant quasi-isometric embedding $X \to \SL(d,\Kb) / \SU(d,\Kb)$ whose quasi-isometry constants depend only on properties of $X$, the dimension $d$,  $L_r$ for a particular $r > 0$, $C$, and  $c$. In particular, they may be chosen to be uniform over $\rho \in \Oc^\prime$.
 
By Lemma~\ref{lem:Holder regularity}, for any $\rho \in \Oc^\prime$, the Anosov boundary map 
$$
\xi_\rho \colon \partial_\infty X \to \Gr_k(\Kb^d) \times \Gr_{d-k}(\Kb^d)
$$
is H\"older (relative to any visual metric on $\partial_\infty X$ and any Riemannian distance on $\Gr_k(\Kb^d) \times \Gr_{d-k}(\Kb^d)$), with constants depending only on the properties of $X$, the dimension $d$, $L_r$ for a particular $r > 0$, $C_{K}$, $c$, and $C$. In particular, the H\"older constants can be chosen to be independent of $\rho \in \Oc^\prime$.

\section{The general semisimple case}\label{sec:general case}

In this section, we consider relatively Anosov representations into general semisimple Lie groups. The main result of this section, Proposition~\ref{prop:adapted representations} below, is an extension of Proposition 4.3 in \cite{GW} and will allow us to reduce the general case to the case of representations into the special linear group. For geometrically finite Fuchsian groups, this reduction was established in~\cite[App.\ B]{CZZ2021} and our exposition is based on the arguments there. 

For the rest of the section, we will assume that $\mathsf{G}$ is a semisimple Lie group of non-compact type with finite center. Fix a parabolic subgroup $\Psf^+ \leq \mathsf G$ and an opposite parabolic subgroup $\Psf^- \leq \mathsf G$, then let $\Fc^\pm := \mathsf G/\Psf^{\pm}$ be the associated flag varieties. We say that $F_1 \in \Fc^+$ is \emph{transverse} to $F_2 \in \Fc^-$ if $(F_1,F_2)$ is contained in the unique open $\mathsf{G}$-orbit in $\Fc^+ \times \Fc^-$. 

Definition~\ref{defn:Pk Anosov} then naturally extends as follows: 

\begin{definition} Suppose that $(\Gamma,\peripherals)$ is relatively hyperbolic with Bowditch boundary $\partial(\Gamma, \peripherals)$. A representation $\rho \colon \Gamma \to \mathsf{G}$ is \emph{$\Psf^\pm$-Anosov relative to $\peripherals$} if there exists a continuous map 
$$
\xi = (\xi^+, \xi^-) \colon \partial(\Gamma, \peripherals) \to \Fc^+ \times \Fc^-
$$
which is 
\begin{enumerate} 
\item \emph{$\rho$-equivariant}: if $\gamma \in \Gamma$, then $\rho(\gamma) \circ \xi = \xi \circ \gamma$, 
\item \emph{transverse}: if $x,y \in \partial(\Gamma, \peripherals)$ are distinct, then $\xi^+(x)$ and $\xi^-(y)$ are transverse, 
\item \emph{strongly dynamics-preserving}: if $(\gamma_n)_{n \geq 1}$ is a sequence of elements in $\Gamma$ where $\gamma_n \to x \in \partial(\Gamma, \peripherals)$ and $\gamma_n^{-1} \to y \in \partial(\Gamma, \peripherals)$, then 
$$
\lim_{n \to \infty} \rho(\gamma_n)F = \xi^+(x)
$$
uniformly on compact subsets of $\{ F \in \Fc^+ : F \text{ transverse to } \xi^-(y)\}$. 
\end{enumerate}
\end{definition}

\begin{example}\label{ex:abstract formulation for SL(d,k)}  Let $e_1,\dots, e_d$ denote the standard basis of $\Kb^d$ and fix $k \in \{1,\dots, d\}$. Then 
$$
\Psf^+:= \{ g \in \SL(d,\Kb) : g\ip{e_1,\dots,e_k}=\ip{e_1,\dots,e_k}\}
$$ 
and 
$$
\Psf^-:=\{g \in \SL(d,\Kb) : g\ip{e_{k+1},\dots, e_d}=\ip{e_{k+1},\dots, e_d}\}
$$
are opposite parabolic subgroups where the associated flag varieties $\Fc^+$, $\Fc^-$ naturally identify with $\Gr_k(\Kb^d)$, $\Gr_{d-k}(\Kb^d)$ respectively. Further, under this identification, transversality in the sense above is equivalent to transversality in the usual linear algebra sense. So $\Psf^\pm$-Anosov representations coincide with the $\Psf_k$-Anosov representations defined in Definition~\ref{defn:Pk Anosov}. 
\end{example} 

If $\Psi\colon \mathsf{G}\to\mathsf{SL}(V)$ is a finite-dimensional irreducible representation, we say that $\Psi$ {\em is adapted to $\Psf^\pm$}  if there exists a decomposition $V=L_0\oplus W_0$ where $L_0$ is a line, $W_0$ is a hyperplane, 
\begin{align*}
\Psf^+ = \{ g \in \mathsf{G} : \Psi(g)(L_0) = L_0 \}, \quad \text{ and } \quad \Psf^- = \{ g \in\mathsf  G : \Psi(g)(W_0) = W_0\}.
\end{align*}
Given such a representation, we may define embeddings $\zeta^+_{\Psi} \colon \Fc^+ \to \proj(V)$ and $\zeta^-_{\Psi} \colon \Fc^- \to \Gr_{\mathrm{dim}(V)-1}(V)$ by 
\begin{align*}
\zeta^+_{\Psi}(g\Psf^+) = \Psi(g)(L_0) \quad \text{ and } \quad \zeta^-_{\Psi}(g\Psf^-) = \Psi(g)( W_0).
\end{align*}
Then let $\zeta_\Psi := (\zeta^+_\Psi, \zeta^-_{\Psi})$. 

\begin{remark} Such representations can be constructed as follows: if $\gL$ is the Lie algebra of $\mathsf{G}$, $\mathfrak{n}^+$ is the nilpotent radical of the Lie algebra of $\Psf^+$, and $n := \dim \mathfrak{n}^+$, then the representation defined by $\Psi(g) := \wedge^n \mathrm{Ad}(g)$ and $V := \Spanset \{ \Psi(\mathsf G) (\bigwedge^n \mathfrak{n}^+)\} \subset \bigwedge^n \gL$ is adapted to $\Psf^\pm$, see \cite[Rem.\ 4.12]{GW}.
\end{remark}

The main result of this section is the following extension of \cite[Prop.\ 4.3]{GW}, which was previously established for geometrically finite Fuchsian groups in~\cite{CZZ2021}. 

\begin{proposition}\label{prop:adapted representations} Suppose that $\Psi \colon \mathsf{G}\to\mathsf{SL}(V)$ is a finite-dimensional irreducible representation which is adapted to $\Psf^\pm$. If $(\Gamma,\peripherals)$ is relatively hyperbolic and $\rho \colon \Gamma \to \mathsf{G}$ is a representation, then the following are equivalent: 
\begin{enumerate}
\item $\rho$ is $\Psf^\pm$-Anosov relative to $\peripherals$,
\item $\Psi \circ \rho$ is $\Psf_1$-Anosov relative to $\peripherals$. 
\end{enumerate}
Moreover, when the above conditions hold, if $\xi_\rho, \xi_{\Psi \circ \rho}$ are the Anosov boundary maps of $\rho, \Psi \circ \rho$ respectively, then $\xi_{\Psi \circ \rho}=\zeta_\Psi \circ \xi_\rho$. 
\end{proposition} 

We will combine Proposition~\ref{prop:adapted representations} and Observation~\ref{obs: k and d-k duality} to prove the following corollary (which can also be deduced directly from the structure theory of the flag manifolds $\Fc^\pm$). 

\begin{corollary}[see Section~\ref{sec: Proof of cor: P+ and P- duality}] \label{cor: P+ and P- duality} If $(\Gamma,\peripherals)$ is relatively hyperbolic and $\rho \colon \Gamma \to \mathsf{G}$ is a representation, then the following are equivalent: 
\begin{enumerate}
\item $\rho$ is $\Psf^\pm$-Anosov relative to $\peripherals$,
\item $\rho$ is $\Psf^\mp$-Anosov relative to $\peripherals$.
\end{enumerate}
Moreover, when the above conditions hold, if $\xi=(\xi^+, \xi^-)$ is the $\Psf^\pm$-Anosov boundary map of $\rho$, then $\hat{\xi}=(\xi^-, \xi^+)$ is the $\Psf^\mp$-Anosov boundary map of $\rho$. 
\end{corollary} 

Proposition~\ref{prop:adapted representations} also allows us to generalize our results about linear relatively Anosov representations to the general setting. As in the $\SL(d,\Kb)$ case, if $(\Gamma,\peripherals)$ is relatively hyperbolic and $\rho_0\colon \Gamma\to \mathsf{G}$ is a representation, let $\mathrm{Hom}_{\rho_0}(\Gamma,\mathsf{G})$ denote the space of representations $\rho\colon\Gamma\to \mathsf{G}$ so that  if $P \in \peripherals$, then $\rho|_P$ is conjugate to $\rho_0|_P$. 

As a consequence of Proposition~\ref{prop:adapted representations} and Theorem~\ref{thm:stability}, we have the following corollary. 

\begin{corollary}\label{cor:stability in general}
Suppose that $(\Gamma, \peripherals)$ is relatively hyperbolic and $\rho_0\colon\Gamma \to \mathsf{G}$ is $\Psf^\pm$-Anosov relative to $\peripherals$. Then there exists an open neighborhood $\Oc$ of $\rho_0$ in $\Hom_{\rho_0}(\Gamma, \mathsf{G})$ such that every representation in $\Oc$ is $\Psf^\pm$-Anosov relative to $\peripherals$. 

Moreover, 
\begin{enumerate}
\item If $\xi_\rho$ is the Anosov boundary map of $\rho \in \Oc$, then the map 
$$
(\rho, x) \in \Oc \times \partial(\Gamma, \peripherals) \mapsto \xi_\rho(x) \in \Fc^+ \times \Fc^-
$$
is continuous. 
\item If $h\colon M \to \mathcal{O}$ is a real-analytic family of representation and $x \in \partial(\Gamma,\peripherals)$, then the map 
$$
u \in M \mapsto \xi_{h(u)}(x) \in \Fc^+ \times \Fc^-
$$
is real-analytic.
\end{enumerate}
\end{corollary}

We also can use Proposition~\ref{prop:adapted representations} to extend Theorem~\ref{thm:main} to general $\Psf^\pm$-Anosov representations.

Given a weak cusp space $X$ for a relatively hyperbolic group $(\Gamma, \peripherals)$ and a representation $\rho \colon \Gamma \to \mathsf{G}$, we define the bundles
\begin{align*}
\wh B_\rho^\pm := \Gamma \backslash (\Gc(X) \times \Fc^{\pm})\quad\mathrm{and}\quad
\wh V_\rho^\pm := \Gamma \backslash(\Gc(X) \times T\Fc^{\pm}),
\end{align*}
where $ T\Fc^{\pm}$ is the tangent bundle of $ \Fc^{\pm} $. Observe that $\wh{V}_\rho^\pm$ is a vector bundle over $\wh{B}_\rho^\pm$ of rank $\dim(\Fc^\pm)$. The flow $\geodflow^t$ on $\Gc(X)$ extends to flows $\flatflow^t$ on $B_\rho^\pm:=\Gc(X) \times \Fc^{\pm}$ and $\homflow^t$ on $V_\rho^\pm:=\Gc(X) \times T\Fc^{\pm}$ whose action is trivial on the second factor.
These in turn descends to flows on $\wh B_\rho^\pm$ and $\wh V_\rho^\pm$, also denoted by $\flatflow^t$ and $\homflow^t$ respectively, which cover the geodesic flow on $\wh{\Gc}(X)$. 

Given a continuous $\rho$-equivariant transverse map $\xi$ we define sections 
$$
\mathscr{s}_\xi^\pm \colon \Gc(X)\to B_\rho^\pm=\Gc(X)\times \Fc^\pm
$$
by $\mathscr{s}_\xi^\pm(\sigma)=\big(\sigma,\xi^\pm(\sigma^\pm)\big)$. Since $\xi$ is $\rho$-equivariant, $\mathscr{s}_\xi^\pm$ descend to sections
$\wh{\mathscr{s}}_\xi^\pm\colon\wh{\Gc}(X)\to \wh B_\rho^\pm$. Finally we consider the vector bundles  $(\wh{\mathscr{s}}_\xi^\pm)^*\wh V_\rho^\pm \to \wh{\Gc}(X)$. By construction, the flow $\homflow^t$ on $\wh V_\rho^\pm$ pulls back to a flow on this bundle which we also denote by $\homflow^t$.

\begin{proposition}[see Section~\ref{sec: proof of prop:equivalence of definitions general case}] \label{prop:equivalence of definitions general case}
Suppose that $(\Gamma,\peripherals)$ is relatively hyperbolic, $\rho\colon \Gamma \to \mathsf{G}$ is a representation, and 
\[ \xi = (\xi^+, \xi^-) \colon \partial(\Gamma, \peripherals) \to \Fc^+ \times \Fc ^- \]
is a continuous $\rho$-equivariant transverse map. Then the following are equivalent: 
\begin{enumerate}
\item $\rho$ is $\Psf^\pm$-Anosov relative to $\peripherals$ with Anosov boundary map $\xi$. 
\item There exist a weak cusp space $X$ for $(\Gamma, \peripherals)$ and a family of norms $\norm{\cdot}$ on the fibers of the associated bundle $(\wh{\mathscr{s}}_\xi^-)^*\wh V_\rho^-$ such that the flow $\homflow^t$ is exponentially contracting.
\item There exist a weak cusp space $X$ for $(\Gamma, \peripherals)$ and a family of norms $\norm{\cdot}$ on the fibers of the associated bundle $(\wh{\mathscr{s}}_\xi^+)^*\wh V_\rho^+$ such that the flow $\homflow^t$ is exponentially expanding. 
\item For any Groves--Manning cusp space $X$ for $(\Gamma,\peripherals)$, there exists a family of norms $\norm{\cdot}$ on the fibers of the associated bundle $(\wh{\mathscr{s}}_\xi^-)^*\wh V_\rho^-$ such that the flow $\homflow^t$ is exponentially contracting.
\item For any Groves--Manning cusp space $X$ for $(\Gamma,\peripherals)$, there exists a family of norms $\norm{\cdot}$ on the fibers of the associated bundle $(\wh{\mathscr{s}}_\xi^+)^*\wh V_\rho^+$ such that the flow $\homflow^t$ is exponentially expanding. 
\end{enumerate}
\end{proposition}

\begin{example}\label{ex: prop:equivalence of definitions general case in special case}  Let $\Psf^+, \Psf^- \leq \SL(d,\Kb)$ be as in Example~\ref{ex:abstract formulation for SL(d,k)}. Suppose that $(\Gamma,\peripherals)$ is relatively hyperbolic, $\rho\colon \Gamma \to \mathsf{G}$ is a representation, and
\[ 
\xi = (\xi^+, \xi^-) \colon \partial(\Gamma, \peripherals) \to \Fc^+ \times \Fc ^- 
\]
is a continuous $\rho$-equivariant transverse map. Then $\left. (\mathscr{s}_\xi^+)^* V^+_\rho\right|_{\sigma} = T_{\xi^+(\sigma^+)} \Gr_k(\Kb^d)$ and, since $\xi$ is transverse, there is a natural isomorphism 
$$
T_{\xi^+(\sigma^+)} \Gr_k(\Kb^d) \simeq \Hom( \xi^+(\sigma^+), \xi^-(\sigma^-)).
$$
Using the notation from Section~\ref{sec:intro to flow}, this implies that there is a bundle isomorphism $(\mathscr{s}_\xi^+)^* V^+_\rho \simeq \Hom(\Theta^k, \Xi^{d-k})$ which descends to a bundle isomorphism 
$$
(\wh{\mathscr{s}}_\xi^+)^*\wh V_\rho^+ \simeq \Hom\left(\wh{\Theta}^k, \wh{\Xi}^{d-k}\right).
$$
Moreover, this isomorphism intertwines the flows. The same reasoning implies that $(\wh{\mathscr{s}}_\xi^-)^*\wh V_\rho^-$ is isomorphic to $\Hom\left(\wh{\Xi}^{d-k},\wh{\Theta}^k\right)$. So, in this special case, the implications (1) $\implies$ (4) and (1) $\implies$ (5) in Proposition~\ref{prop:equivalence of definitions general case} follow from Theorem~\ref{thm:main}. Notice that Theorem~\ref{thm:main} does not imply the converse, since in Theorem~\ref{thm:main} we assumed the family of norms on $\Hom\left(\wh{\Theta}^k, \wh{\Xi}^{d-k}\right)$ are induced by a metric on the bundle $\wh{E}_\rho(X)$. 
\end{example}

\subsection{Proof of Proposition~\ref{prop:adapted representations}} We start with a lemma. 

\begin{lemma}\label{lem: adapted representations are nice} Suppose that $\Psi \colon \mathsf{G}\to\mathsf{SL}(V)$ is a finite-dimensional irreducible representation which is adapted to $\Psf^\pm$. If $(F^+, F^-) \in \Fc^+ \times \Fc^-$, then: 
\begin{enumerate}
\item $F^+$ and $F^-$ are transverse if and only if $\zeta_\Psi^+(F^+)$ and $\zeta_\Psi^-(F^-)$ are transverse. 
\item If $(g_n)_{n \geq 1}$ is a sequence in $\mathsf{G}$, then the following are equivalent: 
\begin{enumerate}
\item $g_n F \to F^+$ uniformly on compact subsets of $$\{ F \in \Fc^+ : F \text{ transverse to } F^-\}.$$ 
\item $\Psi(g_n) v \to \zeta_\Psi^+(F^+)$ uniformly on compact subsets of $$\{ v \in \proj(V) : v\text{ transverse to } \zeta_\Psi^-(F^-)\}.$$ 

\end{enumerate}
\end{enumerate}
\end{lemma} 

\begin{proof} 
Part (1) follows from \cite[Prop.\ 3.5]{GGKW} or \cite[Obs.\ B.10]{CZZ2021}. Since $\zeta_\Psi$ is an embedding, the implication (b) $\implies$ (a) in part (2) follows immediately from part (1). 

To show that (a) $\implies$ (b) in part (2), suppose $(g_n)_{n \geq 1}$ is a sequence in $\mathsf{G}$ and $g_n F \to F^+$ uniformly on compact subsets of 
$$
\{ F \in \Fc^+ : F \text{ transverse to } F^-\}.
$$
 To show that $\Psi(g_n) L \to \zeta_\Psi^+(F^+)$ uniformly on compact subsets of 
 $$
 \{ v \in \proj(V) : v\text{ transverse to } \zeta_\Psi^-(F^-)\},
 $$  
 it suffices to show that the sequence $[\Psi(g_n)] \in \PGL(V) \subset \proj(\End(V))$ converges to the projective linear transformation $T \in \proj(\End(V))$ with ${\rm image}(T)= \zeta_\Psi^+(F^+)$ and $\ker(T) = \zeta_\Psi^-(F^-)$. Since $\proj(\End(V))$ is compact, it is enough to consider the case where $[\Psi(g_n)]$ converges to some $S \in \proj(\End(V))$. 

Since $\Psi$ is irreducible, $\zeta_\Psi^+(\Fc^+)$ spans $V$. So we can pick $F_1,\dots, F_m \in \Fc^+$ such that 
$$
\ker S \oplus \zeta_\Psi^+(F_1) \oplus \dots \oplus \zeta_\Psi^+(F_m) = V.
$$
By perturbing, we may also assume that each $F_j$ is transverse to $F^-$. Then 
$$
S(\zeta_\Psi^+(F_j)) = \lim_{n \to \infty} \Psi(g_n)\zeta_\Psi^+(F_j)=\lim_{n \to \infty} \zeta_\Psi^+( g_nF_j) = \zeta_\Psi^+(F^+). 
$$
So ${\rm image}(S) = \zeta_\Psi^+(F^+)$.

To compute the kernel, we notice that $\Gr_{\dim(V)-1}(V)$ may be identified with $\proj(V^*)$ by identifying a hyperplane $Q$ in $V$ with the projective
class of linear functionals with kernel $Q$.  Notice that $[{^*\Psi(\gamma_n)}]$ converges to ${^*S}$ in $\proj(\End(V^*))$. Further, since $g_n F \to F^+$ uniformly on compact subsets of 
$$
\{ F \in \Fc^+ : F \text{ transverse to } F^-\},
$$ 
one can show that $g_n F \rightarrow F^-$ uniformly on compact subsets of 
$$
\{ F \in \Fc^- : F \text{ transverse to } F^+\},
$$ 
see for instance~\cite[Appendix A]{CZZ3}.  So repeating the argument above shows that $\mathrm{Image}({^*S})=\zeta_\Psi^-(F^-)$, so  the kernel of $S$ is $\zeta_\Psi^-(F^-)$. 

Since $T$ and $S$ have rank one and the same image and kernel, we see that $T=S$. 

\end{proof} 

\begin{proof}[Proof of Proposition~\ref{prop:adapted representations}] First suppose that $\rho$ is $\Psf^\pm$-Anosov relative to $\peripherals$ with boundary map $\xi_\rho$. Then Lemma~\ref{lem: adapted representations are nice} implies that $\Psi \circ \rho$ is $\Psf_1$-Anosov relative to $\peripherals$ with boundary map $\zeta_\Psi \circ \xi_\rho$.

Next suppose that  $\Psi \circ \rho$ is $\Psf_1$-Anosov relative to $\peripherals$ with boundary map $\xi_{\Psi \circ \rho}$. We claim that $\xi_{\Psi \circ \rho}$ has image in $\zeta_\Psi(\Fc^+ \times \Fc^-)$. Fix $x \in \partial(\Gamma, \peripherals)$. Then there exists a sequence $(\gamma_n)_{n \geq 1}$ such that $\gamma_n \to x$. Passing to a subsequence we can suppose that $\gamma_n^{-1} \to y$. Since $\Psi$ is irreducible, there exists $F \in \Fc^+$ such that $\zeta_\Psi^+(F)$ is transverse to $\xi_{\Psi \circ \rho}^{\dim(V)-1}(y)$. Then 
$$
\xi_{\Psi \circ \rho}^1(x) = \lim_{n \to \infty} (\Psi \circ \rho)(\gamma_n) \zeta_\Psi^+(F) = \lim_{n \to \infty} \zeta_\Psi^+( \rho(\gamma_n) F) \in \zeta_\Psi^+(\Fc^+)
$$
So $\xi_{\Psi \circ \rho}^1$ has image in $\zeta_\Psi^+(\Fc^+)$. A similar argument shows that $\xi_{\Psi \circ \rho}^{d-1}$ has image in $\zeta_\Psi^-(\Fc^-)$. Thus $\xi_{\Psi \circ \rho}$ has image in $\zeta_\Psi(\Fc^+ \times \Fc^-)$.

Then $\xi_\rho : = \zeta_\Psi^{-1} \circ \xi_{\Psi \circ \rho}$ is well defined and Lemma~\ref{lem: adapted representations are nice} implies that $\rho$ is $\Psf^\pm$-Anosov relative to $\peripherals$ with boundary map $\xi_\rho$.

\end{proof} 

\subsection{Proof of Corollary~\ref{cor: P+ and P- duality}}\label{sec: Proof of cor: P+ and P- duality} Fix a finite-dimensional irreducible representation  $\Psi\colon\mathsf{G}\to\mathsf{SL}(V)$  which is adapted to $\Psf^\pm$. The proof of Proposition~\ref{prop:adapted representations} can be used to show the following. 

\begin{lemma}  If $(\Gamma,\peripherals)$ is relatively hyperbolic and $\rho \colon \Gamma \to \mathsf{G}$ is a representation, then the following are equivalent: 
\begin{enumerate}
\item $\rho$ is $\Psf^\mp$-Anosov relative to $\peripherals$,
\item $\Psi \circ \rho$ is $\Psf_{\dim(V)-1}$-Anosov relative to $\peripherals$. 
\end{enumerate}
Moreover, when the above conditions hold, if $\xi_\rho, \xi_{\Psi \circ \rho}$ are the Anosov boundary maps of $\rho, \Psi \circ \rho$ respectively, then $\xi_{\Psi \circ \rho}=\zeta_\Psi \circ \xi_\rho$. 
\end{lemma}

Then the Corollary follows from this Lemma, Proposition~\ref{prop:adapted representations}, and Observation~\ref{obs: k and d-k duality}.

\subsection{Proof of Proposition~\ref{prop:equivalence of definitions general case}}\label{sec: proof of prop:equivalence of definitions general case} Suppose that $(\Gamma,\peripherals)$ is relatively hyperbolic, $\rho\colon \Gamma \to \mathsf{G}$ is a representation, and 
\[ \xi = (\xi^+, \xi^-) \colon \partial(\Gamma, \peripherals) \to \Fc^+ \times \Fc ^- \]
is a continuous $\rho$-equivariant transverse map. 

We note that (4)$\implies$(2) and (5)$\implies$(3) are by definition. As the next proof demonstrates, two of the other implications follow quickly from Proposition~\ref{prop:adapted representations} and Theorem~\ref{thm:main}. 

\begin{lemma} [(1)$\implies$(4),(5)] If $\rho$ is $\Psf^\pm$-Anosov relative to $\peripherals$ with Anosov boundary map $\xi$ and $X$ is a Groves--Manning cusp space for $(\Gamma,\peripherals)$, then there exists a family of norms $\norm{\cdot}$ on the fibers of the associated bundle $(\wh{\mathscr{s}}_\xi^\pm)^*\wh V_\rho^\pm$ such that the flow $\homflow^t$ is exponentially expanding/contracting. 
\end{lemma} 

\begin{proof} Let  $\Psi\colon\mathsf{G}\to\mathsf{SL}(V)$ be a finite-dimensional irreducible representation which is adapted to $\Psf^\pm$. Then Proposition~\ref{prop:adapted representations} implies that $\Psi \circ \rho$ is $\Psf_1$-Anosov with Anosov boundary map $\xi_{\Psi \circ \rho}=\zeta_\Psi \circ \xi$. 

By Theorem~\ref{thm:main} (see Example~\ref{ex: prop:equivalence of definitions general case in special case}), there exists a family of norms $\norm{\cdot}$ on the fibers of the associated bundle associated bundle $(\wh{\mathscr{s}}_{\xi_{\Psi\circ\rho}}^\pm)^*\wh V_{\Psi \circ \rho}^\pm$ such that the flow $\phi^t$ is exponentially expanding/contracting. 

Notice that the maps $\zeta^\pm$ induce bundle embeddings $\iota^\pm \colon \wh V_\rho^\pm \into \wh V_{\Psi\circ\rho}^\pm$ which intertwines the flows on the two bundles. Further, 
$$
(\wh{\mathscr{s}}_\xi^\pm)^*\wh V_\rho^\pm=(\iota^\pm)^*\left((\wh{\mathscr{s}}_{\xi_{\Psi\circ\rho}}^\pm)^*\wh V_{\Psi\circ\rho}^\pm\right)
$$
and so if we equip $(\wh{\mathscr{s}}_\xi^\pm)^*\wh V_\rho^\pm$  with the pullback norm, then the flow is exponentially expanding/contracting. 
\end{proof}

We will complete the proof of Proof of Proposition~\ref{prop:equivalence of definitions general case} by showing that (3) $\implies$ (1) and (2) $\implies$ (1). To prove these directions we need to set some additional notation. Let $\gL$ be the Lie algebra of $\mathsf{G}$ and let $\pL^\pm$ be the Lie algebra of $\Psf^\pm$. Then there exists a Cartan decomposition $\gL = \kL \oplus \pL$, a Cartan subspace $\aL \subset \pL$, and an element $H_\star \in \aL$ so that 
$$\pL^\pm = \gL_0 \oplus \bigoplus_{\alpha(\pm H_\star) \geq 0} \gL_\alpha$$
where 
$$\gL = \gL_0 \oplus \bigoplus_{\alpha \in \Sigma} \gL_\alpha$$
is the root space decomposition associated to $\aL$. Let $\nL^\pm = \bigoplus_{\alpha(\pm H_\star) > 0} \gL_\alpha$ and define
$$
T \colon \nL^{-} \to \Fc^+ \quad \text{by} \quad T(Y) = e^{Y}\Psf^{+}.
$$

We use the following observation whose proof can be found in~\cite[Obs.\ B.13 and Lem.\ B.14]{CZZ2021}. 

\begin{observation}\label{obs:properties_of_the_map} \
\begin{enumerate}
\item $T(\nL^-) = \{ F\in \Fc^+ : F  \text{ is transverse to } \Psf^-\}$.
\item $d(T)_0 \colon \nL^{-} \to T_{\Psf^+} \Fc^+$ is a linear isomorphism. 
\item If $H \in \aL$, then $e^H \circ T = T\circ \Ad(e^H)$. 
\item If $(H_n)_{n \geq 1}$ is a sequence in $\aL$ with $\lim_{n \to \infty} \alpha(H_n) = -\infty$ for all $\alpha \in \Sigma$ with $\alpha(H_\star) < 0$, then 
\begin{align*}
\lim_{n \to \infty} e^{H_n} (F) = \Psf^+
\end{align*}
uniformly on compact subsets of $$\{ F \in \Fc^+ : F \text{ transverse to } \Psf^-\}.$$ 
\end{enumerate}
\end{observation}

\begin{remark} To be precise, for part (4) it was only claimed in ~\cite{CZZ2021}   that $\lim_{n \to \infty} e^{H_n} (F) = \Psf^+$ for all $F\in \Fc^+$ transverse to $\Psf^-$, however the proof implies the stronger form of convergence stated above. \end{remark}

\begin{lemma}[(3)$\implies$(1)]\label{lem: 3 implies 1 in general annoying proof} If there exist a weak cusp space $X$ for $(\Gamma,\peripherals)$ and a family of norms $\norm{\cdot}$ on the fibers of the associated bundle $(\wh{\mathscr{s}}_\xi^+)^*\wh V_\rho^+$ such that the flow $\homflow^t$ is exponentially expanding, then $\rho$ is $\Psf^\pm$-Anosov relative to $\peripherals$ with Anosov boundary map $\xi$.
\end{lemma}

\begin{proof} The following argument is similar to the proof of Lemma B.9 in~\cite{CZZ2021}. 

Notice that we only have to show that $\xi$ is strongly dynamics-preserving. So consider an escaping sequence $(\gamma_n)_{n \geq 1}$ in $\Gamma$ with $\gamma_n \to x$ and $\gamma_n^{-1} \to y$.

Let $\mathsf{K} \leq \mathsf{G}$ be the maximal compact subgroup with Lie algebra $\kL$, fix a $\mathsf{K}$-invariant Riemannian metric on $\Fc^+$, and let $|\cdot|$ denote the induced family of norms on the fibers of $T\Fc^+ \to \Fc^+$. 

Recall that $\mathscr{s}_\xi^+ \colon \Gc(X) \to V_\rho^+=\Gc(X) \times \Fc^+$ is given by
$$
\mathscr{s}_\xi^+(\sigma) =( \sigma, \xi^+(\sigma^+)).
$$
By hypothesis, there exists a $\rho$-equivariant continuous family of norms on the fibers of $(\mathscr{s}_\xi^+)^*V_\rho^+$ and constants  $C, c> 0$ such that 
\begin{align*}
\norm{Z}_{\flatflow^{-t}(\mathscr{s}_\xi^+(\sigma))} \leq Ce^{-c t} \norm{Z}_{\mathscr{s}_\xi^+(\sigma)}
\end{align*}
for all $t > 0$, $\sigma \in \Gc(X)$, and $Z \in T_{\mathscr{s}_\xi^+(\sigma^+)} \Fc^+$. 

\medskip
\noindent \textbf{Case 1:} If  $x \neq y$, then $\gamma_n$ is loxodromic when $n$ is sufficiently large, $\gamma_n^+ \to x$, and $\gamma_n^{-} \to y$. Furthermore, we can find a bounded sequence $(\sigma_n)_{n \geq 1}$ in $\Gc(X)$ such that $\sigma_n^\pm = \gamma_n^\pm$, and a bounded sequence $(g_n)_{n \geq 1}$ in $\mathsf{G}$ such that 
\begin{align*}
g_n (\xi^+(\gamma_n^+), \xi^-(\gamma_n^-)) = (\Psf^+, \Psf^-).
\end{align*}
Then 
$$g_n \rho(\gamma_n) g_n^{-1} \Psf^\pm = \Psf^\pm\quad\mathrm{so}\quad
g_n \rho(\gamma_n) g_n^{-1} \in \mathsf{L}:=\Psf^+ \cap \Psf^-.$$
for all $n$.
Notice that  
\begin{align*}
\gL_0 \oplus \bigoplus_{\alpha(H_\star)=0} \gL_\alpha
\end{align*}
is a root space decomposition of the Lie algebra of $\mathsf{L}$. Then, using the Cartan decomposition of the reductive group $\mathsf{L}$, for every $n \geq 1$ there exist $m_{n}, \ell_{n} \in \mathsf{K} \cap \mathsf{L}$ and $H_n \in \aL$ so that 
\begin{align*}
g_n \rho(\gamma_n) g_n^{-1} = m_{n} e^{H_n} \ell_{n}.
\end{align*}

\noindent \textbf{Claim:} If $\alpha \in \Sigma$ and $\alpha(H_\star) < 0$, then $\lim_{n \to \infty} \alpha(H_n) = -\infty$. 

\medskip

Since $(\gamma_n \sigma_n)^\pm = \sigma_n^\pm$, we can find  $t_n \to \infty$ and a bounded sequence $(\hat{\sigma}_n)_{n \geq 1}$ in $\Gc(X)$ such that $\gamma_n^{-1}\sigma_n = \phi^{-t_n}(\hat{\sigma}_n)$. 
Since $\{\sigma_n\}\cup\{\hat{\sigma}_n \}$ is bounded there exists $C_1 > 1$ such that: if $\sigma \in \{\sigma_n\}\cup\{\hat{\sigma}_n\}$, then 
\begin{align*}
\frac{1}{C_1} |Z|_{\xi^+(\sigma^+)} \leq \norm{Z}_{\mathscr{s}^+_\xi(\sigma)} \leq C_1 |Z|_{\xi^+(\sigma^+)}
\end{align*}
for all $Z \in T_{\xi^+(\sigma^+)}\Fc^+$. Likewise, there exists $C_2 > 1$ such that 
\begin{align*}
\frac{1}{C_2} |Z|_F \leq |g_n(Z)|_{g_n(F)}  \leq C_2 |Z|_F
\end{align*}
for all $n \geq 1$, $F \in \Fc^+$, and $Z \in T_F\Fc^+$.

Since both $m_n$ and $\ell_n$ fix $\Psf^+$ and $|\cdot|$ is a $\mathsf{K}$-invariant family of norms, it follows that for any $Z \in T_{\Psf^+} \Fc^+$, we have 
\begin{align}\label{eqn: long}
|e^{H_n} (Z)|_{\Psf^+} & = |m_{n}^{-1} g_n \rho(\gamma_n) g_n^{-1} \ell_{n}^{-1}( Z)|_{\Psf^+} \leq C_2 | \rho(\gamma_n) g_n^{-1} \ell_{n}^{-1}( Z)|_{\xi^+(\sigma_n^+)} \nonumber\\
 &\leq C_1C_2  \norm{\rho(\gamma_n) g_n^{-1} \ell_{n}^{-1} (Z)}_{\mathscr{s}_\xi^+(\sigma_n)}= C_1C_2  \norm{g_n^{-1} \ell_{n}^{-1}(Z)}_{\varphi^{-t_n}(\mathscr{s}_\xi^+(\hat{\sigma}_n))}\nonumber\\
 &\leq C_1C_2Ce^{-c t_n} \norm{g_n^{-1} \ell_{n}^{-1} (Z)}_{\mathscr{s}_\xi^+(\hat{\sigma}_n)}\leq C_1^2C_2Ce^{-c t_n} |g_n^{-1} \ell_{n}^{-1} (Z)|_{\xi^+(\hat{\sigma}_n^+)}\\
& \leq C_1^2C_2^2Ce^{-c t_n} |Z|_{\Psf^+}.\nonumber
\end{align}

Fix $\alpha\in\Sigma$ with $\alpha(H_\star)<0$. Then fix $Y \in\mathfrak g_\alpha$. Then Observation~\ref{obs:properties_of_the_map}(2) implies that $Z:= d(T)_0(Y)\in T_{\Psf^+}\Fc^+$. Further
\[
e^{H_n}(Z)=d(e^{H_n}\circ T)_0(Y)=\left.\frac{d}{dt}\right|_{t=0}e^{H_n}\circ T(tY)
\]
and by Observation~\ref{obs:properties_of_the_map}(3) 
\[
\left.\frac{d}{dt}\right|_{t=0}e^{H_n}\circ T(tY)=\left.\frac{d}{dt}\right|_{t=0}T(te^{\alpha(H_n)}Y)=e^{\alpha(H_n)}\left.\frac{d}{dt}\right|_{t=0}T(tY)=e^{\alpha(H_n)}Z.
\]
Thus, $e^{H_n}(Z)=e^{\alpha(H_n)}Z$, so the inequality in Equation~\eqref{eqn: long} implies that 
\begin{align*}
\lim_{n \to \infty} \alpha(H_n) = -\infty.
\end{align*}
This completes the proof of the claim.

Then, by Observation~\ref{obs:properties_of_the_map}(4) 
\begin{align*}
\lim_{n \to \infty} e^{H_n}( F) = \Psf^+
\end{align*}
uniformly on compact subsets of $$\{ F \in \Fc^+ : F \text{ transverse to } \Psf^-\}.$$ 
Since $g_n (\xi^+(x), \xi^-(y)) \to (\Psf^+,\Psf^-)$, $m_{n}\Psf^\pm =\Psf^\pm=\ell_n \Psf^\pm$, and 
$\rho(\gamma_n) = g_n^{-1} m_n e^{H_n} \ell_n g_n$ we then have 
\begin{align*}
\lim_{n \to \infty} \rho(\gamma_n) (F)= \xi^+(x)
\end{align*}
uniformly on compact subsets of $$\{ F \in \Fc^+ : F \text{ transverse to } \xi^-(y)\}.$$

\medskip
\noindent \textbf{Case 2:}  If $x = y$, pick $\beta \in \Gamma$ so that $z:=\beta^{-1}( x) \neq x$. Then $\gamma_n\beta \to x$ and  $(\gamma_n\beta)^{-1} \to z \neq x$. 
By the first case, $\rho(\gamma_n\beta)(F) \to \xi^+(x)$ uniformly on compact subsets of $$\{ F \in \Fc^+ : F \text{ transverse to } \xi^-(z) = \rho(\beta^{-1})\xi^-(x)\}.$$ 
Equivalently, $\rho(\gamma_n)(F) \to \xi^+(x)$ uniformly on compact subsets of $$\{ F \in \Fc^+ : F \text{ transverse to } \xi^-(x)\}.$$ 
\end{proof} 

\begin{lemma}[(2) $\implies$ (1)] If there exist a weak cusp space $X$ for $(\Gamma,\peripherals)$ and a family of norms $\norm{\cdot}$ on the fibers of the associated bundle $(\wh{\mathscr{s}}_\xi^-)^*\wh V_\rho^-$ such that the flow $\homflow^t$ is exponentially contracting, then $\rho$ is $\Psf^\pm$-Anosov relative to $\peripherals$ with Anosov boundary map $\xi$.
\end{lemma}

\begin{proof} Suppose $\xi = (\xi^+, \xi^-)$ and then define $\hat{\xi}:=(\xi^-,\xi^+)$. Then $(\wh{\mathscr{s}}_{\hat{\xi}}^+)^*\wh V_\rho^+ = (\wh{\mathscr{s}}_\xi^-)^*\wh V_\rho^-$ and so Lemma~\ref{lem: 3 implies 1 in general annoying proof} implies that  $\rho$ is $\Psf^\mp$-Anosov relative to $\peripherals$ with Anosov boundary map $\hat{\xi}$. Then Corollary~\ref{cor: P+ and P- duality} implies that $\rho$ is $\Psf^\pm$-Anosov relative to $\peripherals$ with Anosov boundary map $\xi$.
\end{proof}

\appendix

\section{Proofs for Sections~\ref{sec:SVD background} and~\ref{sec: background on proximal and weakly unipotent elements}}\label{sec: proofs for section SVD background}

In this appendix we prove three observations stated in Sections~\ref{sec:SVD background} and~\ref{sec: background on proximal and weakly unipotent elements}.

\begin{observation}[Observation \ref{obs:strongly_dynamics_pres_div_cartan}] \label{obs:strongly_dynamics_pres_div_cartan in appendix}
Suppose that $(g_n)_{n \geq 1}$ is a sequence in $\SL(d,\Kb)$, $V_0 \in \Gr_k(\Kb^d)$, and $W_0 \in \Gr_{d-k}(\Kb^d)$. Then the following are equivalent: 
\begin{enumerate}
\item $g_n(V) \to V_0$ uniformly on compact subsets of 
$$
\left\{ V \in \Gr_k(\Kb^d) : V \text{ transverse to } W_0\right\}.
$$ 
\item $\frac{\mu_k}{\mu_{k+1}}(g_n) \to \infty$, $U_k(g_n) \to V_0$, and $U_{d-k}(g_n^{-1}) \to W_0$. 
\item There exist open sets $\Oc \subset \Gr_k(\Kb^d)$ and $\Oc^\prime \subset \Gr_{d-k}(\Kb^d)$ such that $g_n(V) \to V_0$ for all $V \in \Oc$ and $g_n^{-1}(W) \to W_0$ for all $W \in \Oc^\prime$. 
\end{enumerate}
\end{observation}

\begin{proof} Let $g_n = m_n a_n \ell_n$ denote a singular value decomposition of $g_n$. Notice that, if $\frac{\mu_k}{\mu_{k+1}}(g_n) > 0$, then $U_k(g_n) = m_n \ip{e_1,\dots,e_k}$ and $U_{d-k}(g_n^{-1}) = \ell_n^{-1}\ip{e_{k+1}, \dots, e_d}$. 

Also let $M_{d-k,k}(\Kb)$ denote the subspace of $(d-k)$-by-$k$ matrices with entries in $\Kb$ and let $T \colon M_{d-k,k}(\Kb) \rightarrow \Gr_k(\Kb^d)$ denote the map
$$
T(A) = \left\{ (v,Av) : v \in \Kb^k\right\}.
$$
Then $T$ induces a homeomorphism 
$$
M_{d-k,k}(\Kb) \cong \Oc_{\ip{e_{k+1},\dots, e_d}} := \left\{ V \in \Gr_k(\Kb^d) : V \text{ transverse to } \ip{e_{k+1}, \dots, e_d}\right\}.
$$
Further,
\begin{equation}
\label{eqn:action of a_n on chart}
a_n \cdot T( [A_{i, j}]) = T\left( \left[ \frac{\mu_{k+i}(g_n)}{\mu_j(g_n)} A_{i,j} \right] \right)
\end{equation}
for all $[A_{i,j}] \in M_{d-k,k}(\Kb)$. 

(3) $\implies$ (2): By compactness, it suffices to consider the case where the limits 
$$
m := \lim_{n \rightarrow \infty} m_n  \quad \text{and} \quad \ell : = \lim_{n \rightarrow \infty}\ell_n
$$
exist. 

Notice that if $C \subset \ell \Oc$ is compact, then for $N$ sufficiently large the set 
$$
\ell^{-1} C \cup \bigcup_{n \geq N} \ell_n^{-1} C
$$
is a compact subset of $ \Oc$. So 
\begin{equation}
\label{eqn:limit of a_n} 
\lim_{n \rightarrow \infty} a_n(V) =m^{-1}\lim_{n \rightarrow \infty} g_n(\ell_n^{-1}V)= m^{-1}(V_0)
\end{equation}
uniformly on compact subsets of $\ell\Oc$. 

Fix a subsequence $(n_t)_{t \geq 1}$ such that 
$$
\lim_{t \rightarrow \infty} \frac{\mu_{k+1}}{\mu_k} (g_{n_t}) = \limsup_{n \rightarrow \infty} \frac{\mu_{k+1}}{\mu_k} (g_n) \in [0,1].
$$
Passing to a subsequence we can also suppose that 
$$
c_{i,j}:=\lim_{t \rightarrow \infty} \frac{\mu_{k+i}}{\mu_j} (g_{n_t}) 
$$
exists for all $1 \leq i \leq d-k$ and $1 \leq j \leq k$. Then, by Equations~\eqref{eqn:limit of a_n}  and ~\eqref{eqn:action of a_n on chart},
\begin{equation*}
m^{-1}(V_0)=\lim_{t \rightarrow \infty} a_{n_t} \cdot T\left( [A_{i,j}]\right) = T\left( \left[  c_{i,j} A_{i,j} \right] \right)
\end{equation*}
for all $[A_{i,j}] \in T^{-1}(  \Oc_{\ip{e_{k+1},\dots, e_d}} \cap \ell \Oc)$. Since $\Oc_{\ip{e_{k+1},\dots, e_d}} \cap \ell \Oc$ is dense in $\ell \Oc$, we must have 
$$
c_{i,j} = 0 \quad \text{and} \quad m^{-1}(V_0) = T(0) = \ip{e_1,\dots, e_k}.
$$
So $\frac{\mu_k}{\mu_{k+1}}(g_n) \to \infty$ and 
$$
\lim_{n \rightarrow \infty} U_k(g_n) = \lim_{n \rightarrow \infty} m_n\ip{e_1,\dots, e_k} = m\ip{e_1,\dots, e_k} = V_0.
$$

Using the exact argument for the action of $g_n^{-1}$ on $\Gr_{d-k}(\Kb^d)$ we see that 
$$
\lim_{n \rightarrow \infty} U_{d-k}(g_n^{-1}) =W_0. 
$$

(1) $\implies$ (2): By compactness, it suffices to consider the case where the limits 
$$
m := \lim_{n \rightarrow \infty} m_n  \quad \text{and} \quad \ell : = \lim_{n \rightarrow \infty}\ell_n
$$
exist. Since 
$$
\left\{ V \in \Gr_k(\Kb^d) : V \text{ transverse to } W_0\right\}
$$ 
is open, arguing as in the proof that (3) $\implies$ (2), we see that $\frac{\mu_k}{\mu_{k+1}}(g_n) \to \infty$ and 
$$
\lim_{n \rightarrow \infty} U_k(g_n) = \lim_{n \rightarrow \infty} m_n\ip{e_1,\dots, e_k} = m\ip{e_1,\dots, e_k} = V_0.
$$

Now suppose for a contradiction that $\ell^{-1}\ip{e_{k+1}, \dots, e_d} \neq W_0$. Then there exists $V \in \Gr_k(\Kb^d)$ which is transverse to $\ell W_0$ but not $\ip{e_{k+1},\dots, e_d}$. Then arguing as in  Equation~\eqref{eqn:limit of a_n}, we see that 
$$
\ip{e_1,\dots, e_k} = m^{-1} V_0 = \lim_{n \rightarrow \infty} a_n(V). 
$$
However, $C : = \Gr_k(\Kb^d) \smallsetminus \Oc_{\ip{e_{k+1},\dots, e_d}}$ is closed and $a_n(C) = C$ for all $n$. So we also have
$$
\ip{e_1,\dots, e_k} = \lim_{n \rightarrow \infty} a_n(V) \in C. 
$$
This is clearly impossible and hence $\ell^{-1}\ip{e_{k+1}, \dots, e_d} = W_0$. Then 
$$
\lim_{n \rightarrow \infty} U_{d-k}(g_n^{-1}) = \lim_{n \rightarrow \infty} \ell_n^{-1}\ip{e_{k+1},\dots, e_d} = \ell^{-1}\ip{e_{k+1},\dots, e_d} = W_0.
$$

(2) $\implies$ (1) and (3): Since $\frac{\mu_k}{\mu_{k+1}}(g_n) \to \infty$, Equation~\eqref{eqn:action of a_n on chart} implies that
$$
\lim_{n \rightarrow \infty} a_n(V) = \ip{e_1,\dots,e_k}
$$ 
uniformly on compact subsets of $\left\{ V \in \Gr_k(\Kb^d) : V \text{ transverse to } \ip{e_{k+1},\dots,e_d}\right\}$. Then, since $m_n \ip{e_1,\dots,e_k} \rightarrow V_0$ and $\ell_n^{-1}\ip{e_{k+1}, \dots, e_d} \rightarrow W_0$, we have that $g_n(V) \to V_0$ uniformly on compact subsets of $\left\{ V \in \Gr_k(\Kb^d) : V \text{ transverse to } W_0\right\}$. So (1) holds. 

The same reasoning shows that $g_n^{-1}(W) \to W_0$ uniformly on compact subsets of $\left\{ W \in \Gr_{d-k}(\Kb^d) : W \text{ transverse to } V_0\right\}$. Hence (3) holds. 

\end{proof}

\begin{observation}[Observation \ref{obs:dynamics of Pk proximal}] \label{obs:dynamics of Pk proximal in appendix} If $g \in \SL(d,\Kb)$, then the following are equivalent: 
\begin{enumerate}
\item $g$ is $\mathsf{P}_k$-proximal,
\item there exist $V_0 \in \Gr_k(\Kb^d)$, $W_0 \in \Gr_{d-k}(\Kb^d)$ such that $V_0 \oplus W_0 = \Kb^d$ and
$$g^n(V) \to V_0$$ 
uniformly on compact subsets of $\left\{ V \in \Gr_k(\Kb^d) : V \text{ transverse to } W_0\right\}$.
\end{enumerate}
Moreover, if the above conditions are satisfied, then $V_0=V_g^+$ and $W_0 = W_g^-$. 
\end{observation}

\begin{proof} Let $f \colon \Gr_k(\Kb^d) \rightarrow \proj\left( \bigwedge^k \Kb^d\right)$ denote the Pl\"ucker embedding.

(1) $\implies$ (2): Fix a basis $\{v_1, \dots, v_d\}$ of $\Kb^d$ such that 
$$
V_g^+ = \ip{v_1, \dots, v_k} \quad \text{and} \quad W_g^- = \ip{v_{k+1}, \dots, v_d}. 
$$
Then relative to the basis $\{ v_{i_1} \wedge \cdots \wedge v_{i_k} : 1 \leq i_1 < \cdots < i_k \leq d\}$ we have 
$$
\wedge^k g = \begin{pmatrix} \lambda & \\ & A \end{pmatrix}
$$
where $\abs{\lambda} =  \lambda_1(\wedge^k g)$ and $\lambda_1(A) < \abs{\lambda}$. Since $\lambda_1(A) = \lim_{n \rightarrow \infty} \mu_1(A^n)^{1/n}$, then 
$$
\lim_{n \rightarrow \infty} (\wedge^k g)^n w  =[v_1 \wedge \cdots \wedge v_k] = f(V_g^+)
$$
for all $w \in \proj\left( \bigwedge^k \Kb^d \right)$ not in the projectivization of 
$$
W:=\ip{ v_{i_1} \wedge \cdots \wedge v_{i_k} : (i_1, \dots, i_k) \neq (1,\dots, k)}.
$$
Moreover, the convergence is uniform on compact subsets of $\proj\left( \bigwedge^k \Kb^d \right) \smallsetminus \proj(W)$. Also, notice that $V \in \Gr_k(\Kb^d)$ is transverse to $W_g^-$ if and only if $f(V) \notin \proj(W)$. Hence,  $g^n(V) \to V_g^+$ uniformly on compact subsets of 
$$
\left\{ V \in \Gr_k(\Kb^d) : V \text{ transverse to } W_g^-\right\}.
$$

(2) $\implies$ (1):  Fix a compact neighborhood $K$ of $V_0$ homeomorphic to a closed ball and where every element of $K$ is transverse to $W_0$. Then Observation~\ref{obs:strongly_dynamics_pres_div_cartan in appendix} implies that there exists $N \geq 1$ such that $g^n(K) \subset K$ for all $n \geq N$. Then for each $n \geq N$, $g^n$ has a fixed point $V_n \in K$.  Using Observation~\ref{obs:strongly_dynamics_pres_div_cartan in appendix}, we have 
$$
V_n = \lim_{m \rightarrow \infty} \left(g^{n}\right)^{m} V_n = V_0. 
$$
So $V_0$ is $g^n$-invariant for each $n \geq N$. So $V_0$ is $g$-invariant. The same argument applied to $g^{-1}$ shows that $W_0$ is $g$-invariant. 

Relative to the decomposition $\Kb^d = V_0 \oplus W_0$ we can write 
$$
g = \begin{pmatrix} A & \\ & B \end{pmatrix}
$$
where $A \in \GL(V_0)$ and $B \in \GL(W_0)$. Then pick unit vectors $v \in V_0$ and $w \in W_0$ such that $\norm{A^n v} = \lambda_k(A)^n$ and $\norm{B^n w} = \lambda_1(B)^n$ for all $n \geq 1$. Extend $v$ to a basis $\{ v, v_2,\dots, v_k\}$ of $V_0$, then consider the subspace 
$$
V := \ip{v+w, v_2, \dots, v_k} \in \Gr_k(\Kb^d). 
$$
Since $V$ is transverse to $W_0$, we have $g^nV \rightarrow V_0$. This is only possible if $\lambda_k(A) > \lambda_1(B)$. Hence $g$ is $\Psf_k$-proximal, $V_g^+=V_0$, and $W_g^-=W_0$.

\end{proof}

\begin{observation}[Observation \ref{obs:dynamics of weakly unipotent}] Suppose that $g \in \SL(d,\Kb)$, $V_0^\pm \in \Gr_k(\Kb^d)$, $W_0^\pm \in \Gr_{d-k}(\Kb^d)$, and 
$$
g^{\pm n} V \to V_0^\pm
$$
uniformly on compact subsets of $\left\{ V \in \Gr_k(\Kb^d) : V \text{ transverse to } W_0^{\pm}\right\}$. Then  $g$ is weakly unipotent if and only if $V_0^+=V_0^-$. 
\end{observation}  

\begin{proof} 
$(\implies)$: Let $h = \wedge^k g$. Then $h$ is also weakly unipotent and so if $h = h_{ss}h_u$ is the Jordan decomposition, then $h_{ss}$ is elliptic. So we can fix a subsequence $(n_j)_{j \geq 1}$ such that
$$
\id: = \lim_{j \rightarrow \infty} h_{ss}^{ \pm n_j}.
$$
Passing to a further subsequence, we can suppose that the limits 
$$
T_{\pm}:=\lim_{j \rightarrow  \infty} h_{u}^{\pm n_j} 
$$
exist in $\proj\left( \End\left( \bigwedge^k \Kb^d \right)\right)$. Since $h_u$ is unipotent, if we fix a basis of $ \bigwedge^k \Kb^d$, then the entries in the matrix representation of $h_u^n$ are polynomials in $n$. So $T_{+} = T_-$. 

Let $f \colon \Gr_k(\Kb^d) \rightarrow \proj\left( \bigwedge^k \Kb^d\right)$ denote the Pl\"ucker embedding.  The sets
$$
\Oc_1 : = \{ V \in  \Gr_k(\Kb^d)  : V \text{ transverse to } W_0^+ \text{ and } W_0^-\}
$$
and 
$$
\Oc_{2} : = \{ V \in  \Gr_k(\Kb^d)  : f(V) \notin \proj(\ker(T_-)) \cup \proj(\ker(T_+)) \}
$$
are open and dense. So we can fix $V \in \Oc_1 \cap \Oc_{2}$. Then 
\begin{align*}
f(V_0^+) & = \lim_{n \rightarrow \infty} f(g^n V) = \lim_{j \rightarrow \infty}h_u^{n_j}  h_{ss}^{n_j}f(V) = T_+(f(V)) = T_-(f(V)) \\
& =  \lim_{j \rightarrow \infty}h_u^{-n_j}  h_{ss}^{-n_j}f(V) = \lim_{n \rightarrow \infty} f(g^{-n} V)= f(V_0^-). 
\end{align*}
So $V_0^+ = V_0^-$. 

$(\impliedby)$: Suppose for a contradiction that $g$ is not weakly unipotent. Then $g$ is $\Psf_m$-proximal for some $1 \leq m \leq d-1$. Let $V_g^+ \in \Gr_m(\Kb^d)$ and $W_g^- \in \Gr_{d-m}(\Kb^d)$ denote the attracting/repelling subspaces. By possibly replacing $g$ by $g^{-1}$, we can assume that $m \leq k$. 

Using Observations~\ref{obs:strongly_dynamics_pres_div_cartan in appendix} and ~\ref{obs:dynamics of Pk proximal in appendix}, we have 
$$
V_g^+ = \lim_{n \rightarrow \infty} U_m(g^n) \subset  \lim_{n \rightarrow \infty} U_k(g^n) = V_0^+.
$$
Applying the same argument to $g^{-1}$ we see that $V_0^- \subset W_g^-$. So $V_0^+ \neq V_0^-$ and we have a contradiction. 
\end{proof}

\section{Basic properties of Gromov-hyperbolic metric spaces}\label{appendix: basic properties} 

In this appendix we collect some basic (and probably well-known) facts about Gromov-hyperbolic metric spaces. 

For the rest of this section suppose that $X$ is a proper geodesic Gromov-hyperbolic metric space. Fix $\delta > 0$ such that every (possibly ideal) geodesic triangle in $X$ is $\delta$-slim (i.e.\ each side is contained in the $\delta$-neighborhood of the union of the two other sides).

\begin{observation}\label{lem:asymp_geodesics} If $\sigma_1,\sigma_2 \ [0,\infty) \to X$ are geodesic rays and $\sigma_1^+=\sigma_2^+$, then 
$$
\sup_{t \geq 0} \dist_X(\sigma_1(t), \sigma_2(t)) \leq \dist_X(\sigma_1(0), \sigma_2(0)) + 4\delta.
$$
\end{observation} 

\begin{proof} By definition there exists $C > 0$ such that 
$$
\dist_X(\sigma_1(t), \sigma_2(t)) \leq C
$$
for all $t \geq 0$. 

Fix $t_0 \geq 0$. Let $T := t_0 + \delta+1 +C$. Then let
\begin{enumerate}
\item $\eta_1$ be a geodesic segment joining $\sigma_1(0)$ and $\sigma_2(0)$,
\item $\eta_2$ be a geodesic segment joining $\sigma_1(T)$ and $\sigma_2(T)$, and 
\item $\sigma_3$ denote the geodesic joining $\sigma_1(0)$ to $\sigma_2(T)$. 
\end{enumerate}
Since $\sigma_1|_{[0,T]} \cup \eta_2 \cup \sigma_3$ is $\delta$-slim,  there exists $q \in \eta_2 \cup \sigma_3$ such that $\dist_X(\sigma_1(t_0), q) \leq \delta$. By construction, 
$$
\dist_X(\sigma_1(t_0), \eta_2) \geq (T-t_0) - C \geq \delta+1 > \delta
$$
and so $q \in \sigma_3$. Since $\sigma_3 \cup \eta_1 \cup \sigma_2|_{[0,T]}$ is $\delta$-slim,   there exists $q^\prime \in \eta_1 \cup \sigma_2$ such that $\dist_X(q,q^\prime) \leq \delta$.  

If $q^\prime \in \eta_1$, then $\dist_X(q^\prime, \sigma_1(0)) \geq t_0 - 2\delta$ and so 
$$
\dist_X(q^\prime, \sigma_2(0))=\dist_X(\sigma_1(0), \sigma_2(0))-\dist_X(q^\prime, \sigma_1(0)) \leq \dist_X(\sigma_1(0), \sigma_2(0)) - (t_0 - 2\delta).
$$
Thus 
\begin{align*}
\dist_X(\sigma_1(t_0), \sigma_2(t_0)) & \leq \dist_X(\sigma_1(t_0), q^\prime)+ \dist_X(q^\prime, \sigma_2(0))+\dist_X(\sigma_2(0), \sigma_2(t_0))\\
&  \leq \dist_X(\sigma_1(0), \sigma_2(0)) +4\delta. 
\end{align*}
Otherwise, $q^\prime = \sigma_2(s)$ for some $s \geq 0$. Then 
$$
\abs{t_0-s} \leq \dist_X(\sigma_1(0), \sigma_2(0)) +2\delta
$$
and so 
\begin{equation*}
\dist_X(\sigma_1(t_0), \sigma_2(t_0)) \leq 2\delta + \abs{t_0-s} \leq \dist_X(\sigma_1(0), \sigma_2(0)) +4\delta.  \qedhere
\end{equation*}
\end{proof} 

The following result can be viewed as a metric analogue of~\cite[Theorem 4.1]{AMS}. It is certainly well known, but we know of no reference. Recall, that a discrete subgroup of $\Isom(X)$ is non-elementary if its limit set consists of at least three points. In this case, the group acts minimally on its limit set and the limit set is uncountable (since it is a perfect closed set). 

\begin{lemma}\label{lem:AMS_Gromov_hyp} Suppose that $\dist_\infty$ is a visual metric on $\partial_\infty X$  and $\Gamma \leq \Isom(X)$ is a non-elementary discrete subgroup. Then there exist $\epsilon > 0$ and a finite set $F \subset \Gamma$ with the following property: for any $\gamma \in \Gamma$ there is some $f \in F$ where $\gamma f$ is loxodromic and $\dist_\infty( (\gamma f)^+, (\gamma f)^-) > \epsilon$. 
\end{lemma} 

\begin{proof} For $x \in \partial_\infty X$ and $r > 0$ let $\Bc(x,r) := \{ y \in \partial_\infty X : \dist_\infty(x,y) < r\}$. 

Fix four distinct points $x_1, x_2, x_3,x_4 \in \partial_\infty X$ in the limit set of $\Gamma$. Let $\epsilon = \frac{1}{4} \min_{1 \leq i < j \leq 4} \dist_\infty(x_i, x_j)$. Since $\Gamma$ acts minimally on its limit set, for every distinct $1 \leq i,j \leq 4$ we can find an element $g_{i,j} \in \Gamma$ such that 
$$
g_{i,j}\Big( \partial_\infty X \smallsetminus \Bc\left(x_j, \epsilon\right)\Big) \subset \mathcal{B}\left(x_i, \epsilon\right) \quad \text{and} \quad g_{i,j}^{-1}\Big( \partial_\infty X \smallsetminus \mathcal{B}\left(x_i, \epsilon\right)\Big) \subset \mathcal{B}\left(x_j, \epsilon\right).
$$

We claim that there exists a finite set $F_0 \subset \Gamma$ such that: if $\gamma \in \Gamma \smallsetminus F_0$, then there exist distinct $1 \leq i,j \leq 4$ such that $\gamma g_{i,j}$ is loxodromic and 
$$
\dist_\infty( (\gamma g_{i,j})^+, (\gamma g_{i,j})^-) > \epsilon. 
$$
Suppose not. Then there exists an escaping sequence $(\gamma_n)_{n \geq 1}$ in $\Gamma$ where each $\gamma_n$ does not have this property. 

Fix a point $p_0 \in X$. Passing to a subsequence we can suppose that $\gamma_n(p_0) \to a \in \partial_\infty X$ and $\gamma_n^{-1}(p_0) \to b \in \partial_\infty X$. Then $\gamma_n(x) \to a$ for all $x \in \partial_\infty X \smallsetminus \{b\}$ and the convergence is uniform on compact subsets of $\partial_\infty X \smallsetminus \{b\}$.

Since the balls $\{ \mathcal{B}(x_i, 2\epsilon)\}_{1 \leq i \leq 4}$ are pairwise disjoint we can pick distinct $1 \leq i,j \leq 4$ such that $a,b \notin \mathcal{B}(x_i, 2\epsilon) \cup \mathcal{B}(x_j, 2\epsilon)$.  Then $\gamma_n g_{i,j}(p_0) \to a$ and 
$$
(\gamma_n g_{i,j})^{-1}(p_0) = g_{i,j}^{-1}\gamma_n^{-1}(p_0)  \to g_{i,j}^{-1}(b) \in \mathcal{B}(x_j, \epsilon). 
$$
Then, by our choice of $i,j$,
$$
\dist_\infty( a, g_{i,j}^{-1}(b)) > \epsilon.
$$
Thus $ \gamma_n g_{i,j}$ is loxodromic for $n$ sufficiently large. Further, $( \gamma_n g_{i,j})^+ \to a$ and $(\gamma_n g_{i,j} )^- \to g_{i,j}^{-1}(b)$. So
for $n$ sufficiently large we have 
$$
\dist_\infty( (\gamma_ng_{i,j} )^+, (\gamma_ng_{i,j} )^-) > \epsilon.
$$
Thus we have a contradiction. Thus there exists a finite set $F_0 \subset \Gamma$ with the desired property. 

Now fix a loxodromic element $h$ with $\dist_\infty(h^+, h^-) > \epsilon$. Then the set 
$$
F := \{ g_{i,j} : 1 \leq i, j \leq 4 \text{ distinct} \} \cup \{ f^{-1}h : f \in F_0\}
$$
satisfies the lemma. 
\end{proof} 

\begin{lemma}\label{lem:finding_axes} Suppose that $Y \subset X$ is a subset where every point in $Y$ is  contained within a bounded distance of a geodesic line in $X$. Then there exists $R > 0$ such that: for any $p,q \in Y$ there is a geodesic line $\sigma : \Rb \rightarrow X$ with 
$$
p,q \in \Nc_X(\sigma, R).
$$
\end{lemma}

\begin{proof} Fix $R_0> 0$ such that: for any $p \in Y$ there is a geodesic line $\sigma : \Rb \rightarrow X$ with $\dist_X(p,\sigma) < R_0$. We claim that $R:=R_0+2\delta$ suffices. 

Fix $p,q \in Y$. Then there exist geodesic lines $\sigma_p, \sigma_q$ and $p^\prime \in \sigma_p$, $q^\prime \in \sigma_q$ with 
$$
\dist_X(p,p^\prime) < R_0 \quad \text{and} \quad \dist_X(q,q^\prime) < R_0.
$$
For $a,b \in \{+,-\}$, let  $\eta_{ab}$ be a geodesic line joining $\sigma_p^a$ and $\sigma_q^b$. 

The ideal geodesic triangle $\sigma_p \cup \eta_{++} \cup \eta_{-+}$ is $\delta$-slim, so there exists $p^{\prime\prime} \in \eta_{++} \cup \eta_{-+}$ such that 
$$
\dist_X(p^\prime, p^{\prime\prime}) \leq \delta. 
$$
\noindent \textbf{Case 1:} Assume $p^{\prime\prime} \in \eta_{++}$. The ideal geodesic triangle $\sigma_q \cup \eta_{++} \cup \eta_{+-}$ is $\delta$-thin, so there exists $q^{\prime\prime} \in \eta_{++} \cup \eta_{+-}$ such that 
$$
\dist_X(q^\prime, q^{\prime\prime}) \leq \delta. 
$$
If $q^{\prime\prime} \in \eta_{++}$, then 
$$
p,q \in \Nc_X(\eta_{++}, R_0+\delta)
$$
and the proof is complete. Otherwise, $q^{\prime\prime} \in \eta_{+-}$. Again using the fact that the ideal geodesic triangle $\sigma_q \cup \eta_{++} \cup \eta_{+-}$ is $\delta$-thin,  there exists $p^{\prime\prime\prime} \in \sigma_q \cup \eta_{+-}$ such that 
$$
\dist_X(p^{\prime\prime}, p^{\prime\prime\prime}) \leq \delta. 
$$
If $p^{\prime\prime\prime} \in \sigma_q$, then 
$$
p,q \in \Nc(\sigma_q, R_0+2\delta)
$$
and if $p^{\prime\prime\prime} \in  \eta_{+-}$, then 
$$
p,q \in \Nc(\eta_{+-}, R_0+2\delta).
$$
So the proof is complete in Case 1. 

\medskip

\noindent \textbf{Case 2:} Assume $p^{\prime\prime} \in \eta_{-+}$. The ideal geodesic triangle $\sigma_q \cup \eta_{-+} \cup \eta_{--}$ is $\delta$-thin, so there exists $q^{\prime\prime} \in \eta_{-+} \cup \eta_{--}$ such that 
$$
\dist_X(q^\prime, q^{\prime\prime}) \leq \delta. 
$$
If $q^{\prime\prime} \in \eta_{-+}$, then 
$$
p,q \in \Nc_X(\eta_{-+}, R_0+\delta)
$$
and the proof is complete. Otherwise, $q^{\prime\prime} \in \eta_{--}$. Again using the fact that the ideal geodesic triangle $\sigma_q \cup \eta_{-+} \cup \eta_{--}$ is $\delta$-thin,  there exists $p^{\prime\prime\prime} \in \sigma_q \cup \eta_{--}$ such that 
$$
\dist_X(p^{\prime\prime}, p^{\prime\prime\prime}) \leq \delta. 
$$
If $p^{\prime\prime\prime} \in \sigma_q$, then 
$$
p,q \in \Nc(\sigma_q, R_0+2\delta)
$$
and if $p^{\prime\prime\prime} \in  \eta_{--}$, then 
$$
p,q \in \Nc(\eta_{--}, R_0+2\delta).
$$
So the proof is complete in Case 2. 
\end{proof}

\bibliographystyle{alpha}
\bibliography{geom} 

\begin{thebibliography}{BCKM21b}

\bibitem[AMS95]{AMS}
Herbert Abels, Gregory~A. Margulis, and Gregory~A. So\u{i}fer.
\newblock Semigroups containing proximal linear maps.
\newblock {\em Israel J. Math.}, 91(1-3):1--30, 1995.

\bibitem[BCKM21a]{BCKM2021a}
Harrison {Bray}, Richard {Canary}, Lien-Yung {Kao}, and Giuseppe {Martone}.
\newblock {Counting, equidistribution and entropy gaps at infinity with
  applications to cusped Hitchin representations}.
\newblock {To appear in \textit{J. Reine Angew. Math.}}, 2021.

\bibitem[BCKM21b]{BCKM2021b}
Harrison {Bray}, Richard {Canary}, Lien-Yung {Kao}, and Giuseppe {Martone}.
\newblock {Pressure metrics for cusped Hitchin components}.
\newblock {\em arXiv e-prints}, page arXiv:2111.07493, November 2021.

\bibitem[BCLS15]{BCLS2015}
Martin Bridgeman, Richard Canary, Fran\c{c}ois Labourie, and Andres Sambarino.
\newblock The pressure metric for {A}nosov representations.
\newblock {\em Geom. Funct. Anal.}, 25(4):1089--1179, 2015.

\bibitem[Ben05]{BenoistIII}
Yves Benoist.
\newblock Convexes divisibles. {III}.
\newblock {\em Ann. Sci. \'{E}cole Norm. Sup. (4)}, 38(5):793--832, 2005.

\bibitem[BGS85]{BGS1985}
Werner Ballmann, Mikhael Gromov, and Viktor Schroeder.
\newblock {\em Manifolds of nonpositive curvature}, volume~61 of {\em Progress
  in Mathematics}.
\newblock Birkh\"{a}user Boston, Inc., Boston, MA, 1985.

\bibitem[BH99]{BH1999}
Martin~R. Bridson and Andr\'{e} Haefliger.
\newblock {\em Metric spaces of non-positive curvature}, volume 319 of {\em
  Grundlehren der mathematischen Wissenschaften [Fundamental Principles of
  Mathematical Sciences]}.
\newblock Springer-Verlag, Berlin, 1999.

\bibitem[BH20]{HealyHruska}
Brendan {Burns Healy} and G.~Christopher {Hruska}.
\newblock {Cusped spaces and quasi-isometries of relatively hyperbolic groups}.
\newblock {\em arXiv e-prints}, page arXiv:2010.09876, October 2020.

\bibitem[Bor91]{Borel_linalg_groups}
Armand Borel.
\newblock {\em Linear algebraic groups}, volume 126 of {\em Graduate Texts in
  Mathematics}.
\newblock Springer-Verlag, New York, second edition, 1991.

\bibitem[Bow99]{Bowditch_convergence_grps}
B.~H. Bowditch.
\newblock Convergence groups and configuration spaces.
\newblock In {\em Geometric group theory down under ({C}anberra, 1996)}, pages
  23--54. de Gruyter, Berlin, 1999.

\bibitem[{Bow}12]{Bowditch_relhyp}
Brian~H. {Bowditch}.
\newblock Relatively hyperbolic groups.
\newblock {\em International Journal of Algebra and Computation},
  22(03):1250016, 2012.

\bibitem[BPS19]{BPS}
Jairo Bochi, Rafael Potrie, and Andr\'{e}s Sambarino.
\newblock Anosov representations and dominated splittings.
\newblock {\em J. Eur. Math. Soc. (JEMS)}, 21(11):3343--3414, 2019.

\bibitem[Can22]{Canary_notes}
Richard~D. Canary.
\newblock Anosov representations: Informal lecture notes.
\newblock 2022.

\bibitem[Cha94]{C1994}
Christophe Champetier.
\newblock Petite simplification dans les groupes hyperboliques.
\newblock {\em Ann. Fac. Sci. Toulouse Math. (6)}, 3(2):161--221, 1994.

\bibitem[Cho]{Choi}
Suhyoung Choi.
\newblock Real projective orbifolds with ends and their deformation spaces.

\bibitem[CLT18]{CLT2018}
Daryl Cooper, Darren Long, and Stephan Tillmann.
\newblock Deforming convex projective manifolds.
\newblock {\em Geom. Topol.}, 22(3):1349--1404, 2018.

\bibitem[CZZ22a]{CZZ2021}
Richard Canary, Tengren Zhang, and Andrew Zimmer.
\newblock Cusped {H}itchin representations and {A}nosov representations of
  geometrically finite {F}uchsian groups.
\newblock {\em Adv. Math}, 404:108439, 2022.

\bibitem[CZZ22b]{CZZ2022}
Richard {Canary}, Tengren {Zhang}, and Andrew {Zimmer}.
\newblock {Entropy rigidity for cusped Hitchin representations}.
\newblock {\em arXiv e-prints}, page arXiv:2201.04859, January 2022.

\bibitem[CZZ23]{CZZ3}
Richard Canary, Tengren Zhang, and Andrew Zimmer.
\newblock Patterson-{S}ullivan measures for transverse groups.
\newblock To appear in 2023.

\bibitem[DGK17]{DGK_convex_cocopt_realproj}
Jeffrey {Danciger}, Fran{\c{c}}ois {Gu{\'e}ritaud}, and Fanny {Kassel}.
\newblock {Convex cocompact actions in real projective geometry}.
\newblock {\em arXiv e-prints}, page arXiv:1704.08711, April 2017.

\bibitem[DS05]{DS2005}
Cornelia Dru\c{t}u and Mark Sapir.
\newblock Tree-graded spaces and asymptotic cones of groups.
\newblock {\em Topology}, 44(5):959--1058, 2005.
\newblock With an appendix by Denis Osin and Mark Sapir.

\bibitem[FHMM16]{FHMM2016}
Goulwen Fichou, Johannes Huisman, Fr\'{e}d\'{e}ric Mangolte, and Jean-Philippe
  Monnier.
\newblock Fonctions r\'{e}gulues.
\newblock {\em J. Reine Angew. Math.}, 718:103--151, 2016.

\bibitem[GGKW17]{GGKW}
Fran\c{c}ois Gu\'eritaud, Olivier Guichard, Fanny Kassel, and Anna Wienhard.
\newblock Anosov representations and proper actions.
\newblock {\em Geom. Topol.}, 21(1):485--584, 2017.

\bibitem[GM08]{GrovesManning}
Daniel {Groves} and Jason~F. {Manning}.
\newblock Dehn filling in relatively hyperbolic groups.
\newblock {\em Israel J. Math.}, 168(1):317--429, 2008.

\bibitem[GW12]{GW}
Olivier Guichard and Anna Wienhard.
\newblock Anosov representations: domains of discontinuity and applications.
\newblock {\em Invent. Math.}, 190(2):357--438, 2012.

\bibitem[Hea20]{Healy2020}
Brendan~Burns Healy.
\newblock Rigidity properties for hyperbolic generalizations.
\newblock {\em Canad. Math. Bull.}, 63(1):66--76, 2020.

\bibitem[KL18]{KL}
Michael {Kapovich} and Bernhard {Leeb}.
\newblock {Relativizing characterizations of Anosov subgroups, I}.
\newblock {\em arXiv e-prints}, page arXiv:1807.00160, June 2018.

\bibitem[KLP17]{KLP2017}
Michael Kapovich, Bernhard Leeb, and Joan Porti.
\newblock Anosov subgroups: dynamical and geometric characterizations.
\newblock {\em Eur. J. Math.}, 3(4):808--898, 2017.

\bibitem[KLP18a]{KLP2018}
Michael Kapovich, Bernhard Leeb, and Joan Porti.
\newblock Dynamics on flag manifolds: domains of proper discontinuity and
  cocompactness.
\newblock {\em Geom. Topol.}, 22(1):157--234, 2018.

\bibitem[KLP18b]{KLP2018b}
Michael Kapovich, Bernhard Leeb, and Joan Porti.
\newblock A {M}orse lemma for quasigeodesics in symmetric spaces and
  {E}uclidean buildings.
\newblock {\em Geom. Topol.}, 22(7):3827--3923, 2018.

\bibitem[Kos68]{Koszul1968}
J.-L. Koszul.
\newblock D\'{e}formations de connexions localement plates.
\newblock {\em Ann. Inst. Fourier (Grenoble)}, 18(fasc. 1):103--114, 1968.

\bibitem[KP22]{KP2022}
Fanny Kassel and Rafael Potrie.
\newblock Eigenvalue gaps for hyperbolic groups and semigroups.
\newblock {\em J. Mod. Dyn.}, 18(0):161--208, 2022.

\bibitem[Lab06]{L2006}
Fran{\c{c}}ois Labourie.
\newblock Anosov flows, surface groups and curves in projective space.
\newblock {\em Invent. Math.}, 165(1):51--114, 2006.

\bibitem[Mar10]{Marquis2010}
Ludovic Marquis.
\newblock Espace des modules marqu\'{e}s des surfaces projectives convexes de
  volume fini.
\newblock {\em Geom. Topol.}, 14(4):2103--2149, 2010.

\bibitem[Min05]{M2005}
Igor Mineyev.
\newblock Flows and joins of metric spaces.
\newblock {\em Geom. Topol.}, 9:403--482, 2005.

\bibitem[Osi06]{Osin}
Denis~V. Osin.
\newblock {\em Relatively Hyperbolic Groups: Intrinsic Geometry, Algebraic
  Properties, and Algorithmic Problems}.
\newblock Number v. 179, no. 843 in American Mathematical Society. American
  Mathematical Society, 2006.

\bibitem[Pet16]{P2016}
Peter Petersen.
\newblock {\em Riemannian geometry}, volume 171 of {\em Graduate Texts in
  Mathematics}.
\newblock Springer, Cham, third edition, 2016.

\bibitem[Pra94]{P1994}
Gopal Prasad.
\newblock {${\bf R}$}-regular elements in {Z}ariski-dense subgroups.
\newblock {\em Quart. J. Math. Oxford Ser. (2)}, 45(180):541--545, 1994.

\bibitem[Rag72]{Raghunathan}
Madabusi~S. Raghunathan.
\newblock {\em Discrete subgroups of {L}ie groups}.
\newblock Springer-Verlag, New York-Heidelberg, 1972.
\newblock Ergebnisse der Mathematik und ihrer Grenzgebiete, Band 68.

\bibitem[Shu87]{Shub1987}
Michael Shub.
\newblock {\em Global stability of dynamical systems}.
\newblock Springer-Verlag, New York, 1987.
\newblock With the collaboration of Albert Fathi and R\'{e}mi Langevin,
  Translated from the French by Joseph Christy.

\bibitem[{Tso}20]{Kostas2020}
Konstantinos {Tsouvalas}.
\newblock {Anosov representations, strongly convex cocompact groups and weak
  eigenvalue gaps}.
\newblock {\em arXiv e-prints}, page arXiv:2008.04462, August 2020.

\bibitem[{Wei}22]{W2022}
Theodore {Weisman}.
\newblock {An extended definition of Anosov representation for relatively
  hyperbolic groups}.
\newblock {\em arXiv e-prints}, page arXiv:2205.07183, May 2022.

\bibitem[{Yam}04]{Yaman}
Asli {Yaman}.
\newblock {A topological characterisation of relatively hyperbolic groups}.
\newblock {\em J. reine angew. Math. (Crelles Journal)}, 566:41--89, 2004.

\bibitem[Zhu21a]{reldomreps}
Feng Zhu.
\newblock Relatively dominated representations.
\newblock {\em Ann. Inst. Fourier (Grenoble)}, 71(5):2169--2235, 2021.

\bibitem[{Zhu}21b]{rdr2}
Feng {Zhu}.
\newblock {Relatively dominated representations from eigenvalue gaps and limit
  maps}.
\newblock {\em arXiv e-prints}, page arXiv:2102.10611, February 2021.

\bibitem[Zim18]{AZimmer_rigid_Ccvx}
Andrew~M. Zimmer.
\newblock Rigidity of complex convex divisible sets.
\newblock {\em J. Topol. Anal.}, 10(4):817--851, 2018.

\bibitem[Zim21]{Z2017}
Andrew Zimmer.
\newblock Projective {A}nosov representations, convex cocompact actions, and
  rigidity.
\newblock {\em J. Differential Geom.}, 119(3):513--586, 2021.

\bibitem[ZZ22]{ZZ2022b}
Feng Zhu and Andrew Zimmer.
\newblock Relative {A}nosov representations via flows {II}: examples.
\newblock 2022.

\end{thebibliography}

\end{document}